\documentclass{siamart}

\usepackage{graphicx} % needed for siamart
\usepackage[rightcaption]{sidecap}
\usepackage{wrapfig}
\usepackage{float}
\usepackage{imakeidx}
\usepackage{comment}
\usepackage{commath}
\usepackage[titletoc]{appendix}
\usepackage[english]{babel}
\usepackage{tabu}
\usepackage{xfrac}
\usepackage{graphicx,bm}
\usepackage{verbatim}
\usepackage{lscape}
\usepackage{relsize}
\usepackage{enumitem}
\usepackage{textcomp}
\usepackage{makecell}
\usepackage{longtable,tabularx}
\usepackage{multirow}
\usepackage{doi}
\usepackage{fancyhdr}
\usepackage{algpseudocode}
\usepackage{setspace}
\usepackage{footnote}
\PassOptionsToPackage{hyphens}{url}
\hypersetup{colorlinks,linkcolor={blue},citecolor={blue},urlcolor={blue},bookmarksopen=true}
\usepackage[numbers]{natbib}
\usepackage{mathtools}
\usepackage[disable]{todonotes}
\usepackage[framemethod=tikz]{mdframed}
\usepackage{booktabs,xcolor,siunitx}
\usepackage{soul}

\usepackage[makeroom]{cancel}

\usepackage{amssymb}  % needed for \mathbb

\newcommand{\ra}[1]{\renewcommand{\arraystretch}{#1}}

\newcommand{\overbar}[1]{\mkern 1.5mu\overline{\mkern-1.5mu#1\mkern-1.5mu}\mkern 1.5mu}
\newcommand{\f}[2]{\frac{#1}{#2}}
\newcommand{\mb}[1]{\mathbf{#1}}

\DeclareMathAlphabet\mathbfcal{OMS}{cmsy}{b}{n}
\renewcommand{\d}{\mathop{}\!\mathrm{d}} % total derivative

\definecolor{rev1}{HTML}{FF999A} % red
\definecolor{rev2}{HTML}{F3F298} % yellow
\definecolor{rev3}{HTML}{B2E0AE} % green
\definecolor{other}{HTML}{C8C7FF} % blue

\headers{Time-Spectral Resolvent Analysis}{M. Howell and S. He}

\makeatletter
\def\cref@override@label@type#1\@nil#2{#2}
\makeatother

\begin{document}

\title{Time-Spectral Resolvent Analysis for \\
Periodic Dynamical Systems}

\author{Max Howell\thanks{Department of Mechanical and Aerospace Engineering, University of Tennessee, Knoxville, TN 37996, USA (\email{mhowel30@vols.utk.edu}).}
\and Sicheng He\footnotemark[1]}

\maketitle

\begin{abstract}
Traditional resolvent analysis is a powerful framework for identifying the most amplified input-output structures in fluid flows from a stationary base state.
Extending this resolvent analysis to periodic base flows poses computational challenges due to quasi-periodic responses and expensive linearization around a time-varying base flow.
This work proposes a time-spectral resolvent operator formulated using the time-spectral method and Fourier collocation that operates directly in the time domain.
Rather than mapping between truncated Fourier coefficients as in frequency-domain approaches, the proposed operator maps forcing and response envelopes defined on a discrete temporal grid, enabling direct Jacobian evaluation at collocation points without computing Fourier coefficients of the base flow.
The time-spectral resolvent achieves spectral convergence and offers simplified implementation that integrates easily with existing scientific computing tools.
The time-spectral resolvent method is validated numerically in three examples including the parametrically forced Mathieu oscillator, the autonomous van der Pol oscillator and the complex Ginzburg-Landau partial differential equation to show that the proposed method accurately predicts the maximum energy amplification and optimal response mode when the system is subject to optimal quasi-periodic forcing.
The proposed framework provides a foundation for extending resolvent-based analysis and control to high-dimensional periodic dynamical systems.
\end{abstract}

\begin{keywords}
time-spectral method, resolvent analysis, periodic systems, Fourier collocation, parametric forcing
\end{keywords}

\begin{AMS}
37M05, 65P99, 37N10
\end{AMS}

\section{Introduction}

The classical resolvent analysis was formulated as a tool to identify the most amplified output structures of fluid flows when subject to a harmonic forcing.
\citet{Trefethen1993} first analyzed the linearized Navier-Stokes operator to show that large transient amplifications can still occur in linearly stable flows.
It was shown that these amplifications were the result of a linear mechanism caused by the non-normal nature of the Navier-Stokes operator.
\citet{Bamieh2001} then demonstrated that different input structures can amplify specific output modes and cause large energy amplifications for the Navier-Stokes operator, even in subcritical conditions.
The exact input modes which caused the most amplified responses in the velocity field for channel flows were then explored by \citet{JOVANOVI2005}, showing that the magnitude of amplification depends on the direction of forcing.
The resolvent analysis was then formalized in the context of turbulent pipe flow by~\citet{McKeon2010}, who showed that the most amplified input-output structures of a flow could be identified by the resolvent operator.
The resolvent operator maps between an external harmonic forcing and a deviation from a steady base flow and the singular value decomposition (SVD) of the resolvent operator identifies the most amplified input-output structures.
Subsequent studies applied the resolvent framework to turbulent channels, boundary layers, and the more advanced transonic buffet \cite{hwang2010linear,Sipp2012,Kojima2019}.
Recent reviews \cite{taira2017modal} have established resolvent analysis as a central tool for connecting linear amplification mechanisms with nonlinear flow dynamics.

The most basic form of resolvent analysis only considers the response of a linearized system in the vicinity of a steady-state solution.
However, it is often desired to analyze systems in the vicinity of an unsteady solution or a time-periodic base flow.
Extending resolvent analysis to flows with time-periodic base states has posed significant challenges such as quasi-periodic responses and the computational expense of linearizing around the entire base flow.
Also, as described by~\citet{Leclercq2023}, another challenge is that the response of the system is parameterized by the phase of the base flow when the forcing is first applied, making the response of the system to forcing non-unique.
Several approaches have been developed to address these challenges including mean flow and mean resolvent analysis, Koopman-based methods, and harmonic resolvent formulations that account for the time-varying nature of the response.

One method to address theses challenges is with the mean resolvent analysis proposed by~\citet{Leclercq2023,bongarzone2025}.
As opposed to mean-flow resolvent analysis which first linearizes the dynamics about the time-averaged steady mean flow (as in~\cite{Sipp2012}), the mean resolvent analysis first linearizes the dynamics about an ensemble of trajectories on the systems attractor.
The mean resolvent operator then predicts the mean response which is statistically optimal in the sense that it is the best linear time invariant (LTI) approximation to the response, thus eliminating the phase dependence of the problem.
While the mean resolvent approach accounts for interactions between the perturbations and the unsteady components of the base flow, it filters out the phase-dependent frequency modulation.
Consequently, the mean resolvent operator is unable to predict the full quasi-periodic response often desired when analyzing systems with time-periodic base flows.

% koopman resolvent
Recent theoretical work by \citet{Susuki2021} has extended Laplace domain analysis to nonlinear limit cycles using the Koopman resolvent.
This operator is the Laplace realization of the Koopman generator~\cite{Koopman1931}, providing a framework to analyze nonlinear autonomous systems without the need for linearization.
In the case of time-periodic base flows, the Koopman resolvent decomposes the spectrum into stationary modes which correspond to base-flow harmonics and nonstationary modes which are governed by Floquet exponents.
However, the rigorous application of such resolvent operators to quasi-periodic or rotating systems is non-trivial;~\citet{Eiter2022} demonstrated that for irrational frequency ratios, the associated operator may lack a uniform bound due to the accumulation of singularities.
These theoretical hurdles motivate the need for a formulation that can robustly capture the dynamics of time-varying systems by explicitly resolving the coupling between the base flow and perturbation harmonics.

% 3. harmonic resolvent for periodic base flows
Another approach to predict the exact phase dependent input-output structures for time varying base flows is the harmonic resolvent presented by~\citet{Padovan2020, Padovan2022}.
The harmonic resolvent formulation involves linearizing around the exact time-periodic base flow of the system then forming the truncated harmonic resolvent operator by taking the Fourier coefficients of the time-varying Jacobian.
This operator has a block-Toeplitz structure and maps between an exponentially modulated periodic (EMP) forcing and response.
When comparing the harmonic resolvent framework with that of the mean resolvent analysis~\cite{Leclercq2023}, it can be noticed that the mean resolvent operator is found along the diagonal block of the harmonic resolvent operator which does not contain any phase-dependent terms.
Because the harmonic resolvent maps between EMP signals, it is able to predict the full, phase dependent response, thereby rigorously extending the resolvent framework to time-periodic base flows.
However, the frequency-domain formulation requires computing Fourier coefficients of the periodic Jacobian and assembling block-Toeplitz operators, which complicates implementation and integration with time-domain solvers.

% The harmonic resolvent formalism explicitly accounts for the full periodic base flow and frequency interactions.~\citet{Padovan2020} developed this framework to analyze amplification mechanisms and cross-frequency interactions in nonlinear flows, showing how forcing at one frequency can induce responses at other harmonics through quadratic nonlinearity. 
% In later works~\cite{Padovan2022, Padovan2024} the harmonic resolvent framework was extended to include sub-harmonics, coupling a forcing frequency with a non-commensurate base frequency.
% This formulation constructs a resolvent operator with a block-Toeplitz structure that maps the Fourier coefficients of the forcing to the Fourier coefficients of the quasi-periodic response, thereby extending resolvent analysis beyond the single-frequency, steady-base setting.
% This analysis was later used by~\citet{Islam2024} to study frequency interactions in compressible cavity flow across an airfoil, where the authors identify how the nonlinear flow permits energy transfer across frequencies in a flow with multiple resonant frequencies.
% However, the frequency-domain formulation requires computing Fourier coefficients of the periodic Jacobian and assembling block-Toeplitz operators, which complicates implementation and integration with time-domain solvers.

% 4. More efficient methods
Beyond the formulation challenges, resolvent analysis for high-dimensional systems also faces computational challenges associated with handling large operators, such as memory allocation.
This has been addressed using matrix-free methods~\cite{Bagheri2009, MONOKROUSOS2010, Loiseau2018,Martini2021}, however these approaches only consider the mean-flow resolvent.
In order to address this challenge in the context of the harmonic resolvent,~\citet{Farghadan2024} proposed an efficient matrix-free time-stepping algorithm that avoids explicit formation of the large block-Toeplitz operator, making the method tractable for high-dimensional applications.
Other approaches include data-driven resolvent analysis~\cite{Herrmann2021} and spectral methods for computing the base flow~\cite{Mundis2013}.

% 5. Spectral methods for periodic systems
Rather than directly computing the base flow of a nonlinear system in the time domain using a time-marching solver, time-spectral methods~\cite{Trefethen2000, boyd2013chebyshev} have been shown to be more efficient in solving for the base flow.
While spectral methods have traditionally only been used for fully periodic flows, extensions have been made to quasi-periodic cases~\cite{Wirth1997} and later shown to scale up to aerospace problems~\cite{Mavriplis2011, Mundis2013}.~\citet{Junge2021} extended spectral methods to cases with multiple non-commensurate frequencies in high dimensional problems by mapping the quasi-periodic time dynamics onto a multi-dimensional torus.  
In this work, we leverage the accuracy and efficiency of the time-spectral method to extend harmonic resolvent analysis to discrete time.
This allows for the exploitation of the computational efficiency of the time-spectral method as well as spectral convergence to the true solution.

% 6. Contribution and layout
The contribution of this paper is to develop a time-domain resolvent formulation for periodic base flows of both autonomous and non-autonomous systems using the time-spectral method and Fourier collocation.
Unlike the frequency-domain harmonic resolvent, which requires computing Fourier coefficients of the periodic Jacobian and assembling block-Toeplitz operators, the time-domain formulation evaluates the Jacobian directly at collocation points.
This simplifies implementation and enables easy integration with existing scientific computing infrastructure, as many solvers already provide time integration and time-spectral capabilities.
The resulting operator directly maps the harmonic forcing envelope to the quasi-periodic response envelope on a discrete temporal grid, and properly accounts for non-commensurate forcing and base frequencies.
The time-spectral formulation is validated with direct time domain simulation to confirm the physical accuracy of the method.

This paper is organized as follows: in section~\ref{sec:classic_resolvent} we present the classical resolvent analysis for stationary base flows.
We then review the frequency domain harmonic resolvent operator from~\citet{Padovan2022} in section~\ref{sec:harmonic_resolvent_overview}.
The formulation for the full time-spectral resolvent (TSR) operator using the time-spectral method is presented in section~\ref{sec:disc_time_op} and the special handling and projection required for autonomous systems is detailed in~\ref{sec:auto_sysms}.
The spectral convergence properties of the TSR and harmonic resolvent operators are compared in section~\ref{sec:convergence}.
Finally, we validate our TSR operator against direct time accurate integration in section~\ref{sec:num_ex} for three examples: a linear system with zero base flow using the Mathieu oscillator, an autonomous non-linear system with the van der Pol oscillator, and a high dimensional autonomous system using the complex Ginzburg-Landau partial differential equation (PDE).

\section{Steady-State Resolvent Analysis}
\label{sec:classic_resolvent}

This section reviews the classical resolvent framework for systems with a steady base state, establishing the notation and concepts that will be extended to periodic base flows in subsequent sections.
We first introduce the dynamical system setup, then define the resolvent operator and its singular value decomposition following~\cite[Chapter~4.5]{Schmid2001}.

\subsection{Dynamical System Setup}
Consider a general nonlinear system given by
\begin{equation}
\dot{\mb{w}} = \mb{r}(\mb{w}, t)  ,
\label{eq:nonlinear_system}
\end{equation}
where $\mb{w} \in \mathbb{R}^n$ is the state variable with $n$ denoting the state dimension, $t \in \mathbb{R}$ is time, and $\mb{r}: \mathbb{R}^n \times \mathbb{R} \to \mathbb{R}^n$ is the governing nonlinear function.
The autonomous system $\mb{r}(\mb{w}, t)=\mb{r}(\mb{w})$ is a special case.

\subsection{Resolvent Operator}

A steady state solution can be found by finding a point $\bar{\mb{w}}$ where $\dot{\bar{\mb{w}}} = \mb{0}$, i.e., $\mb{r}(\bar{\mb{w}}) = \mb{0}$.
Linearizing about this steady point and subtracting the base flow gives dynamics of the form
\begin{equation}
\dot{\boldsymbol{\eta}} = \mb{J}\boldsymbol{\eta} + \mb{B}\mb{f} ,
\end{equation}
where $\boldsymbol{\eta} \in \mathbb{R}^n$ is the perturbation from the base state, $\mb{J} \in \mathbb{R}^{n \times n}$ is the Jacobian, $\mb{B} \in \mathbb{R}^{n \times m}$ is an input matrix that applies forcing to specific states with $m$ denoting the forcing dimension, and $\mb{f}=\hat{\mb{f}}e^{j\omega t}$ is a harmonic forcing term with amplitude $\hat{\mb{f}} \in \mathbb{C}^m$.
For fully actuated systems, $\mb{B}=\mb{I}$ and $m = n$.

For a linear system with harmonic forcing, the state will oscillate at the same frequency as the forcing with a phase shift.
The response can therefore be given as $\boldsymbol{\eta} = \hat{\boldsymbol{\eta}}e^{j\omega t}$.
Substituting and simplifying gives the dynamics
\begin{equation}
\label{eq:linear_freq}
j\omega\hat{\boldsymbol{\eta}} = \mb{J}\hat{\boldsymbol{\eta}} + \mb{B}\hat{\mb{f}},
\end{equation}
and the classic resolvent operator can then be found by rearranging the equation
\begin{equation}
\hat{\boldsymbol{\eta}} = \mb{R}\hat{\mb{f}}, \quad\quad \mathbf{R} = (j\omega \mb{I} - \mb{J})^{-1}\mb{B},
\end{equation}
where $\mb{R}\in\mathbb{R}^{n\times m}$ is the classical resolvent operator.

The resolvent gain $G(\omega)$, defined as the maximum energy amplification in the steady-state response after initial transients have decayed, is
\begin{equation}
G(\omega) = \max_{\hat{\mb{f}} \neq \mb{0}} \frac{\|\hat{\boldsymbol{\eta}}\|_2}{\|\hat{\mb{f}}\|_2} = \sigma_{\max}(\mb{R}).
\end{equation} 
The optimal forcing and response modes are the leading right and left singular vectors $\mb{v}_1$ and $\mb{u}_1$ from the SVD 
\begin{equation}
\mb{R} = \mb{U}\boldsymbol{\Sigma}\mb{V}^*,
\end{equation} 
where $\mb{U} \in \mathbb{C}^{n \times n}$ and $\mb{V} \in \mathbb{C}^{m \times m}$ are unitary matrices with columns $\mb{u}_i$ (response modes) and $\mb{v}_i$ (forcing modes), and $\boldsymbol{\Sigma} \in \mathbb{R}^{n \times m}$ is diagonal with singular values $\sigma_1 \geq \sigma_2 \geq \cdots \geq 0$.

\section{Harmonic Resolvent (HR)}
\label{sec:harmonic_resolvent_overview}

For systems with time-periodic base states, the resolvent framework must be extended to account for the time-varying Jacobian and the quasi-periodic response that arises when forcing and base flow frequencies are incommensurate.
In this section, we review the harmonic resolvent operator developed by~\citet{Padovan2020,Padovan2022}, which operates in the frequency domain using the harmonic balance method.

\subsection{Linear Time-Periodic Systems}

Consider a nonlinear system of form Eq.~\ref{eq:nonlinear_system} with a $T_0$-periodic orbit $\bar{\mb{w}}(t)$ with base frequency $\omega_0 = 2\pi/T_0$:
\begin{equation}
\bar{\mb{w}}(t) = \bar{\mb{w}}(t+T_0), \quad\quad \dot{\bar{\mb{w}}}(t) = \mb{r}(\bar{\mb{w}}(t),t).
\end{equation}
Linearizing about this orbit yields the forced linear time-periodic (LTP) system
\begin{equation}
\dot{\boldsymbol{\eta}}(t) = \mb{J}(t)\boldsymbol{\eta}(t) + \mb{B}\mb{f}(t),
\quad\quad \mb{J}(t) = \frac{\partial \mb{r}}{\partial \mb{w}}\bigg|_{\bar{\mb{w}}(t)},
\label{eq:periodic_resolvent}
\end{equation}
where $\mb{J}(t)$ is the $T_0$-periodic Jacobian and $\mb{B}$ is the input matrix, assumed constant here though the formulation extends to periodic $\mb{B}(t)$.
Quasi-periodic forcing $\mb{f}(t) = \sum_{k\in\mathbb{Z}}\hat{\mb{f}}_k e^{j(\omega_f+k\omega_0) t}$ with spectral content of both the base and forcing frequencies $\omega_f$ couples with the periodic base flow, producing a quasi-periodic response (see Appendix~\ref{app:quasi_periodic} for derivation).
Purely harmonic forcing at frequency $\omega_f$ is the special case where $\hat{\mb{f}}_k = \mb{0}$ for $k \neq 0$.
\begin{equation}
\boldsymbol{\eta}(t) = \sum_{k\in\mathbb{Z}} \hat{\boldsymbol{\eta}}_k e^{j(\omega_f + k\omega_0)t},
\label{eq:periodic_response}
\end{equation}
containing sidebands at frequencies $\omega_f + k\omega_0$.
The harmonic resolvent gain, accounting for this full quasi-periodic response after transients have decayed, is defined as
\begin{equation}
G(\omega_f) = \max_{\hat{\mb{f}} \neq \mb{0}} \frac{\|\hat{\boldsymbol{\eta}}\|_2}{\|\hat{\mb{f}}\|_2},
\label{eq:resolvent_def}
\end{equation}
where $\hat{\boldsymbol{\eta}} = [\ldots, \hat{\boldsymbol{\eta}}_{-1}^\top, \hat{\boldsymbol{\eta}}_0^\top, \hat{\boldsymbol{\eta}}_1^\top, \ldots]^\top$ and $\hat{\mb{f}} = [\ldots, \hat{\mb{f}}_{-1}^\top, \hat{\mb{f}}_0^\top, \hat{\mb{f}}_1^\top, \ldots]^\top$ stack the Fourier coefficients of the response and forcing.

\subsection{Harmonic Balance Method}
\label{sec:harmonic_balance}

The harmonic balance method provides a frequency-domain framework for analyzing periodic and quasi-periodic systems by expanding all time-varying quantities as Fourier series.
For a $T_0$-periodic function $g(t)$, the Fourier expansion is
\begin{equation}
g(t) = \sum_{k=-\infty}^{\infty} \hat{g}_k e^{jk\omega_0 t}, \quad\quad \hat{g}_k = \frac{1}{T_0}\int_0^{T_0} g(t) e^{-jk\omega_0 t} dt.
\end{equation}
The key property exploited by harmonic balance is that time differentiation becomes algebraic multiplication in frequency space: $\dot{g}(t) \leftrightarrow jk\omega_0 \hat{g}_k$.
Similarly, the product of two periodic functions corresponds to a convolution of their Fourier coefficients.

For the nonlinear system $\dot{\mb{w}} = \mb{r}(\mb{w}, t)$, periodic trajectories satisfy the harmonic balance governing equation
\begin{equation}
\label{eq:hb_governing}
jk\omega_0 \hat{\bar{\mb{w}}}_k = \hat{\mb{r}}_k(\hat{\bar{\mb{w}}}), \quad k = 0, \pm 1, \pm 2, \ldots,
\end{equation}
where $\hat{\bar{\mb{w}}}_k$ are the Fourier coefficients of the periodic orbit $\bar{\mb{w}}(t)$ and $\hat{\mb{r}}_k$ denotes the $k$-th Fourier coefficient of $\mb{r}(\bar{\mb{w}}(t),t)$.
This converts the continuous periodic orbit problem into a system of nonlinear algebraic equations in the frequency domain.

Applying harmonic balance to the LTP system Eq.~\ref{eq:periodic_resolvent}, we expand the periodic Jacobian as $\mb{J}(t) = \sum_{\ell} \hat{\mb{J}}_\ell e^{j\ell\omega_0 t}$.
Substituting the quasi-periodic response Eq.~\ref{eq:periodic_response} and collecting terms at each frequency $\omega_f + k\omega_0$ yields (see Appendix~\ref{app:quasi_periodic})
\begin{equation}
\label{eq:harmonic_balance}
j(\omega_f + k\omega_0)\hat{\boldsymbol{\eta}}_k = \sum_{\ell=-\infty}^{\infty} \hat{\mb{J}}_\ell\hat{\boldsymbol{\eta}}_{k-\ell} + \hat{\mb{B}}\hat{\mb{f}}_k.
\end{equation}
The convolution $\sum_\ell \hat{\mb{J}}_\ell \hat{\boldsymbol{\eta}}_{k-\ell}$ couples different response harmonics through the time-varying Jacobian coefficients.
Purely harmonic forcing at frequency $\omega_f$ is a special case where $\hat{\mb{f}}_k = \mb{0}$ for $k \neq 0$; the coupling through $\mb{J}(t)$ still produces a quasi-periodic response with nonzero $\hat{\boldsymbol{\eta}}_k$ at all sidebands.

\subsection{HR Operator Formulation}
\label{sec:hr_operator}

Truncating to sidebands $k = -n_{\text{har}}, \ldots, n_{\text{har}}$ gives the finite system $\mb{L}_{\text{HR}}^{(n_{\text{har}})}(\omega_f)\hat{\boldsymbol{\eta}} = \mb{B}_{\text{HR}}^{(n_{\text{har}})}\hat{\mb{f}}$, where $\hat{\mb{f}}$ stacks the Fourier coefficients of the forcing for all resolved harmonics and $\mb{B}_\text{HR}^{(n_\text{har})} = \mathrm{blkdiag}(\mb{B}, \dots, \mb{B})$ is a block diagonal matrix that applies the input matrix $\mb{B}$ to each harmonic forcing component.
When a truncation level $n_{\text{har}}$ is fixed, we abbreviate $\mb{L}_{\text{HR}}$, $\mb{B}_{\text{HR}}$, and $\mb{R}_{\text{HR}}$ for readability.
The operator $\mb{L}_{\text{HR}}^{(n_{\text{har}})}(\omega_f)$ has the block-Toeplitz structure
\begin{equation}
\label{eq:L_operator_matrix}
\mb{L}_{\text{HR}}^{(n_{\text{har}})}(\omega_f)
=
\begin{bmatrix}
\ddots & \vdots & \vdots & \vdots & \reflectbox{$\ddots$} \\[0.4em]
\cdots & \mb{A}_{-1} & -\hat{\mathbf{J}}_{-1} & -\hat{\mathbf{J}}_{-2} & \cdots \\[0.4em]
\cdots & -\hat{\mathbf{J}}_{1} & \mb{A}_{0} & -\hat{\mathbf{J}}_{-1} & \cdots \\[0.4em]
\cdots & -\hat{\mathbf{J}}_{2} & -\hat{\mathbf{J}}_{1} & \mb{A}_{1} & \cdots \\[0.4em]
\reflectbox{$\ddots$} & \vdots & \vdots & \vdots & \ddots
\end{bmatrix}, \quad\quad
\mb{A}_k = j(\omega_f + k\omega_0)\mathbf{I} - \hat{\mathbf{J}}_0.
\end{equation}
The diagonal blocks $\mb{A}_k$ correspond to the shifted frequency $\omega_f + k\omega_0$, while the off-diagonal blocks $-\hat{\mb{J}}_\ell$ couple different harmonics through the Fourier coefficients of the time-varying Jacobian.
The block-Toeplitz structure arises from the convolution in Eq.~\ref{eq:harmonic_balance} as each $(i,j)$ block depends only on the difference $i-j$.

In the case of non-autonomous systems, $\mb{L}_{\text{HR}}^{(n_{\text{har}})}$ is nonsingular and the truncated harmonic resolvent is \begin{equation}
\mb{R}_{\text{HR}}^{(n_{\text{har}})} = \left(\mb{L}_{\text{HR}}^{(n_{\text{har}})}\right)^{-1}\mb{B}_{\text{HR}}^{(n_{\text{har}})},
\end{equation}
with gain $G(\omega_f) = \sigma_{\text{max}}(\mb{R}_{\text{HR}}^{(n_{\text{har}})})$ that converges monotonically to the bi-infinite limit as $n_{\text{har}} \to \infty$.

When the underlying system is autonomous and the base flow is an exact periodic orbit (such as a limit cycle), the phase-shift invariance introduces a neutral Floquet mode that makes $\mb{L}_{\text{HR}}^{(n_{\text{har}})}$ singular at forcing frequencies which are integer multiples of the base frequency $\omega_f = k\omega_0$; in that case the projected resolvent described by~\citet{Padovan2022} is used.
In practice, truncating to $n_{\text{har}}$ harmonics yields an operator of size $(n \cdot n_{\text{har}})\times(n \cdot n_{\text{har}})$, where $n$ is the number of state variables.
Note that $n_{\text{har}}$ harmonics corresponds to $n_{\text{TS}} = 2n_{\text{har}} + 1$ collocation points in the time-spectral discretization.

\section{Time-Spectral Resolvent (TSR)}

\label{sec:disc_time_op}

In this section, we derive the time-spectral resolvent (TSR) operator using the time-spectral method.
Rather than mapping between the Fourier coefficients of the forcing and response, the TSR operator maps between the forcing and response envelope defined on a discrete grid with $n_{\text{TS}}$ collocation points.
This TSR operator is suitable for non-autonomous systems.
In the case of autonomous systems, singularities arise due to the phase ambiguity of the neutral Floquet mode and the TSR operator must be handled as described in Section~\ref{sec:auto_sysms}.

\subsection{Time-Spectral Method}
\label{sec:ts_method}

The time-spectral method discretizes periodic functions onto equally spaced collocation points and represents time derivatives using spectral differentiation.
For a $T_0$-periodic function, we define $n_{\text{TS}}$ collocation points at phases
\begin{equation}
\theta_j = \frac{2\pi j}{n_{\text{TS}}}, \quad j=0,\dots,n_{\text{TS}}-1,
\end{equation}
which converts the $T_0$-periodic continuous system into a $2\pi$-periodic discrete system.
We require $n_{\text{TS}}$ to be odd, with $n_{\text{TS}} = 2n_{\text{har}} + 1$, so that the discrete Fourier transform (DFT) resolves harmonics $k = -n_{\text{har}}, \ldots, n_{\text{har}}$ symmetrically about $k=0$.

The key ingredient is the Fourier spectral differentiation matrix $\mb{D} \in \mathbb{R}^{n_{\text{TS}} \times n_{\text{TS}}}$, which exactly differentiates trigonometric polynomials of degree up to $n_{\text{har}} = (n_{\text{TS}}-1)/2$.
The entries are given by~\cite{Trefethen2000}
\begin{equation}
\label{eq:diff_matrix}
\mb{D}_{jk} = \begin{cases}
\displaystyle \frac{1}{2}(-1)^{j-k}\csc\left(\frac{\pi(j-k)}{n_{\text{TS}}}\right), & j \neq k, \\[1em]
0, & j = k.
\end{cases}
\end{equation}
This matrix has the property that for a periodic function $g(\theta)$ sampled at $\theta_j$, the derivative at collocation points is $\dot{g}(\theta_j) = \sum_k \mb{D}_{jk} g(\theta_k)$.
The matrix $\mb{D}$ is skew-symmetric and circulant, with eigenvalues $\{0, \pm j, \pm 2j, \ldots, \pm \frac{n_{\text{TS}}-1}{2}j\}$ corresponding to the representable frequencies.

For the nonlinear system $\dot{\mb{w}} = \mb{r}(\mb{w}, t)$, periodic trajectories satisfy the time-spectral governing equation
\begin{equation}
\label{eq:ts_governing}
\omega_0 (\mb{D} \otimes \mb{I}_n) \bar{{\mb{w}}} = \mb{r}(\bar{{\mb{w}}},t),
\end{equation}
where $\bar{{\mb{w}}} = [\bar{\mb{w}}(\theta_0)^\top, \ldots, \bar{\mb{w}}(\theta_{n_{\text{TS}}-1})^\top]^\top$ stacks the periodic orbit $\bar{\mb{w}}(t)$ at all collocation points and $\mb{r}({\bar{\mb{w}}},t)$ is evaluated point-wise.
This converts the continuous periodic orbit problem into a system of $n \cdot n_{\text{TS}}$ nonlinear algebraic equations.

\subsection{TSR Operator Formulation}

The quasi-periodic forcing and response Eq.~\ref{eq:periodic_response} can equivalently be written using the carrier-envelope decomposition
\begin{equation}
\boldsymbol{\eta}(t) = \hat{\boldsymbol{\eta}}(t)e^{j\omega_f t},\quad \mb{f}(t) = \hat{\mb{f}}(t)e^{j\omega_f t}
\end{equation}
where $\hat{\boldsymbol{\eta}}(t) = \sum_{k=-n_{\text{har}}}^{n_{\text{har}}} \hat{\boldsymbol{\eta}}_k e^{jk\omega_0 t}$ (and equivalently for $\hat{\mb{f}}(t)$) is a $T_0$-periodic envelope.
Substituting into the LTP system Eq.~\ref{eq:periodic_resolvent} gives
\begin{equation}
\label{eq:envelope_ode}
\dot{\hat{\boldsymbol{\eta}}}(t) = \big(\mb{J}(t) - j\omega_f\mb{I}\big)\hat{\boldsymbol{\eta}}(t) + \mb{B}\hat{\mb{f}}.
\end{equation}
This is a linear time-periodic ordinary differential equation (ODE) for the envelope $\hat{\boldsymbol{\eta}}(t)$, which can be discretized using the time-spectral method.

Evaluating at $n_{\text{TS}}$ collocation points yields the stacked forcing and state vectors
\begin{equation}
\tilde{\boldsymbol{\eta}} = [\hat{\boldsymbol{\eta}}(\theta_0)^\top, \ldots, \hat{\boldsymbol{\eta}}(\theta_{n_{\text{TS}}-1})^\top]^\top, \quad \tilde{\mb{f}} = [\hat{\mb{f}}(\theta_0)^\top, \ldots, \hat{\mb{f}}(\theta_{n_{\text{TS}}-1})^\top]^\top.
\end{equation}
Using the differentiation matrix Eq.~\ref{eq:diff_matrix}, the derivative becomes $\dot{\tilde{\boldsymbol{\eta}}} = \omega_0(\mb{D}\otimes\mb{I}_n)\tilde{\boldsymbol{\eta}}$.
The time-spectral system is $\mb{L}_{\text{TS}}\tilde{\boldsymbol{\eta}} = \mb{B}_{\text{TS}}\hat{\mb{f}}$, where
\begin{equation}
\label{eq:input_mat}
\mb{B}_{\text{TS}}^{(n_{\text{TS}})} = \mathrm{blkdiag}(\mb{B},\dots,\mb{B})
\end{equation}
applies the forcing to each discretized point in $\tilde{\mb{f}}$.
Note that if it is desired to restrict the forcing to be harmonic where $\hat{\mb{f}}$ is constant, the time spectral input matrix is $\mb{B}_{\text{TS}}^{(n_{\text{TS}})} = 1/\sqrt{n_{\text{TS}}} \otimes \mb{B}$. For more details on this formulation, see Appendix~\ref{app:parseval}.

The matrices $\mb{L}_\text{TS}$ and $\mb{J}_\text{TS}$ are defined as
\begin{equation}
\mb{L}_{\text{TS}}^{(n_{\text{TS}})}(\omega_f) = j\omega_f\mb{I} - \mb{J}_{\text{TS}}, \quad\quad \mb{J}_{\text{TS}} = \mb{J}_{\text{block}} - \omega_0\big(\mb{D}\otimes\mb{I}_n\big) ,
\end{equation}
with $\mb{J}_{\text{block}} = \mathrm{blkdiag}(\mb{J}_0, \ldots, \mb{J}_{n_{\text{TS}}-1})$, where $\mb{J}_k = \partial\mb{r}/\partial\mb{w}|_{\bar{\mb{w}}(\theta_k)}$ denotes the Jacobian evaluated at collocation point $\theta_k$.
Analogous to the block-Toeplitz structure of $\mb{L}_{\text{HR}}$ in Eq.~\ref{eq:L_operator_matrix}, the time-spectral operator has the explicit structure
\begin{equation}
\label{eq:L_TS_matrix}
\resizebox{0.9\columnwidth}{!}{$\displaystyle
\mb{L}_{\text{TS}}^{(n_{\text{TS}})} =
\begin{bmatrix}
j\omega_f\mb{I} - \mb{J}_0 & & \\
& \ddots & \\
& & j\omega_f\mb{I} - \mb{J}_{n_{\text{TS}}-1}
\end{bmatrix}
+ \omega_0
\begin{bmatrix}
\mb{D}_{00}\mb{I} & \cdots & \mb{D}_{0,n_{\text{TS}}-1}\mb{I} \\
\vdots & \ddots & \vdots \\
\mb{D}_{n_{\text{TS}}-1,0}\mb{I} & \cdots & \mb{D}_{n_{\text{TS}}-1,n_{\text{TS}}-1}\mb{I}
\end{bmatrix}
$}.
\end{equation}
The first term is block-diagonal containing $(j\omega_f\mb{I} - \mb{J}_k)$ at each collocation point $\theta_k$, while the second term is block-circulant representing the spectral differentiation.
In contrast to the block-Toeplitz structure of $\mb{L}_{\text{HR}}^{(n_{\text{har}})}$, which couples harmonics through off-diagonal Fourier coefficients $\hat{\mb{J}}_\ell$ ($\ell \neq 0$), the TSR operator $\mb{L}_{\text{TS}}^{(n_{\text{TS}})}$ couples collocation points through the differentiation matrix $\mb{D}$.

While both the HR and TSR operators result in systems of similar dimension, the main computational advantage in the TSR operator is the avoidance of explicitly calculating the Fourier coefficients $\hat{\mb{J}}_n$.
The HR operator requires precalculation of these Fourier coefficients via integration, which can be computationally expensive for high dimensional systems.
In contrast, the TSR operator relies on the direct evaluation of the Jacobian at discrete collocation points, avoiding the extra step of decomposing the linearized flow into Fourier modes.

The truncated full TSR operator can then be written as
\begin{equation}
\label{eq:discrete_op}
\tilde{\boldsymbol{\eta}} = \mb{R}_{\text{TS}}^{(n_{\text{TS}})}\tilde{\mb{f}}, \quad\quad \mb{R}_{\text{TS}}^{(n_{\text{TS}})} = \left(\mb{L}_{\text{TS}}^{(n_{\text{TS}})}\right)^{-1}\mb{B}_{\text{TS}}^{(n_{\text{TS}})} ,
\end{equation}
which maps from an optimal forcing $\tilde{\mb{f}}$ to the resulting response $\tilde{\boldsymbol{\eta}}$.
In this formulation, the full harmonic resolvent gain Eq.~\ref{eq:resolvent_def} at resolution $n_{\text{TS}}$ is the maximum singular value $\sigma_{\text{max}}(\mb{R}_{\text{TS}}^{(n_{\text{TS}})})$.
When a particular $n_{\text{TS}}$ is fixed, we abbreviate $\mb{L}_{\text{TS}}$, $\mb{B}_{\text{TS}}$, and $\mb{R}_{\text{TS}}$ for readability.
The left singular vector $\mb{u}_1$ corresponds to the optimal response envelope, while the right singular vector $\mb{v}_1$ corresponds to the optimal forcing.
The response can be reconstructed by interpolating $\mb{u}_1$ and modulating with the carrier wave $e^{j\omega_f t}$.

\subsection{Normalization of the Forcing and Response Modes}
\label{sec:normalization}

The singular vectors from the SVD are non-unique up to an arbitrary phase $e^{j\phi}$.
To fix this ambiguity, we normalize the response mode $\mb{u}_1$ and rotate it such that it is aligned with the base flow, then recover the forcing mode via
\begin{equation}
\label{eq:norm_forcing}
\mb{v}_1 = \frac{1}{\sigma_1}\mb{R}_{\text{TS}}^* \mb{u}_1.
\end{equation}
This preserves the physical phase relationship between forcing and response, ensuring that driving the system with $\mb{v}_1$ produces the response $\sigma_1 \mb{u}_1$.
For more details, see \citet{Kanchi2025}.

\subsection{Mapping Modes to the Time Domain}
\label{sec:time_domain_mapping}

The optimal response $\mb{u}_1$ and forcing $\mb{v}_1$ envelopes from the SVD are sampled at discrete collocation points.
To reconstruct the continuous quasi-periodic signals, standard trigonometric interpolation~\cite{Trefethen2000} is used to obtain the envelopes $\hat{\boldsymbol{\eta}}(\theta)$ and $\hat{\mb{f}}(\theta)$ at any phase $\theta \in [0, 2\pi)$.
The envelopes are then modulated with the carrier wave to obtain
\begin{equation}
\boldsymbol{\eta}(t) = \hat{\boldsymbol{\eta}}(\omega_0 t) e^{j\omega_f t}, \quad \mb{f}(t) = \hat{\mb{f}}(\omega_0 t) e^{j\omega_f t}
\label{eq:forcing_reconstruction}
\end{equation}
where the physical response $\mathrm{Re}(\boldsymbol{\eta}(t))$ and forcing $\mathrm{Re}(\mb{f}(t))$ are the real parts of the signals.

\section{TSR for Autonomous Systems}

\label{sec:auto_sysms}

In contrast to non-autonomous systems where periodicity arises from an external driver, autonomous systems that exhibit limit cycle oscillations (LCOs) possess phase invariance associated with a neutral Floquet mode.
This means that for any base flow solution $\overbar{\mb{w}}(t)$, a time-shifted solution $\overbar{\mb{w}}(t + \phi)$ is also valid for any phase shift $\phi$.
This results in the time-spectral operator $\mb{L}_{\text{TS}}$ becoming singular at resonant frequencies for autonomous systems.
In this section, we follow the approach by~\citet{Padovan2022} to project $\mb{L}_\text{TS}$ onto a subspace $\Sigma_\text{TS}$ in which it is bijective and invariant.
In this subspace, the governing equation $\mb{L}_\text{TS}\tilde{\boldsymbol{\eta}} = \mb{B}_\text{TS}\tilde{\mb{f}}$ has a unique solution $\tilde{\boldsymbol{\eta}}_\Sigma \in \Sigma_\text{TS}$ when the forcing $\tilde{\mb{f}} \in \Sigma_\text{TS}$, denoted $\tilde{\mb{f}}_\Sigma$.

This section is outlined as follows.
First, we establish that for a neutrally stable linear system, the null space of the time-spectral Jacobian $\mathcal{J}_\text{TS}$ is one dimensional.
We also derive that the right and left null vectors which span this null space are the neutral Floquet mode $\tilde{\mb{p}}_0$ and its adjoint $\tilde{\mb{q}}_0$.
Next, we show that the null space of $\mb{L}_\text{TS}$ is also one dimensional and its null space is spanned by the modulated neutral $\tilde{\mb{p}}_k$ and adjoint $\tilde{\mb{q}}_k$ Floquet modes.
Using this information, we define the subspace $\Sigma_\text{TS}$ to be the range of the oblique projector matrix $\mb{P}_{\omega} = \mb{I} - \tilde{\mb{p}}_{\omega} \tilde{\mb{q}}_{\omega}^*$.
Finally, we prove that the linear system $\mb{L}_\text{TS}\tilde{\boldsymbol{\eta}}_\Sigma = \mb{B}_\text{TS}\tilde{\mb{f}}_\Sigma$ has exactly one unique solution.

\subsection{Null Space of the Time-Spectral Jacobian}
\label{sec:singularity_jacobian}

For an autonomous system $\dot{\mb{w}} = \mb{r}(\mb{w})$, the orbit tangent vector to the base flow $\tilde{\mb{p}}_0(t) \coloneqq \d \overbar{\mb{w}}/\d t$ satisfies the linearized equation 
\begin{equation}
\dot{\tilde{\mb{p}}}_0 = \frac{\partial \mb{r}}{\partial \mb{w}}\bigg|_{\overbar{\mb{w}}} \tilde{\mb{p}}_0 = \mb{J}(t)\tilde{\mb{p}}_0(t).
\end{equation}
This follows directly from the fact that the base flow satisfies $\dot{\overbar{\mb{w}}}(t) = \mb{r}(\overbar{\mb{w}}(t))$.
Differentiating in time and using the chain rule gives $\ddot{\overbar{\mb{w}}}(t) = \left.\frac{\partial \mb{r}}{\partial \mb{w}}\right|_{\overbar{\mb{w}}(t)}\dot{\overbar{\mb{w}}}(t)$.
With the definition $\tilde{\mb{p}}_0(t) \coloneqq \dot{\overbar{\mb{w}}}(t)$, this is exactly the variational equation $\dot{\tilde{\mb{p}}}_0(t) = \mb{J}(t)\tilde{\mb{p}}_0(t)$.
This vector is $T_0$-periodic with the base flow and is normalized to define a unique scaling $\tilde{\mb{p}}_0 \gets \tilde{\mb{p}}_0 / \|\tilde{\mb{p}}_0\|$.
In Floquet analysis which accepts solutions of the form $e^{\lambda t}\mb{p}(t)$, a periodic mode which satisfies $\mb{p}(t) = \mb{p}(t + T)$ corresponds to a neutral mode with Floquet exponent $\lambda = 0$.
Thus, the orbit tangent vector $\tilde{\mb{p}}_0$ is a neutral Floquet mode corresponding to the neutral stability of the linearized dynamics.

In the time-spectral formulation, the time-spectral Jacobian operator $\mathcal{J}_\text{TS}$ is defined by its action on an $n$-vector-valued periodic mode $\mb{y}(\theta)$ via
\begin{equation}
\mathcal{J}_\text{TS}\mb{y} \coloneqq \mb{J}(t)\mb{y} - \omega_0\f{\d\mb{y}}{\d\theta},
\end{equation}
where $\mb{y}(\theta)\in\mathbb{C}^n$ and $\mb{J}(t)\in\mathbb{C}^{n\times n}$, so both terms on the right-hand side have the same dimension as $\mb{y}$.
Here $\theta = \omega_0 t$ is a dimensionless phase, so $\omega_0\,\d\mb{y}/\d\theta=\d\mb{y}/\d t$ and the definition can equivalently be written as $\mathcal{J}_\text{TS}\mb{y} = \mb{J}(t)\mb{y} - \dot{\mb{y}}$.
In the discretized collocation setting with $n_\text{TS}$ points, we represent $\mb{y}$ by the stacked vector $\tilde{\mb{y}}\in\mathbb{C}^{n n_\text{TS}}$, represent $\mb{J}(t)$ by $\mb{J}_\text{block}$, and represent $\mb{J}_\text{TS}$ by its matrix realization in $\mathbb{C}^{(n n_\text{TS})\times(n n_\text{TS})}$ (see Eq.~\ref{eq:L_TS_matrix}).
We denote this collocation matrix realization by $\mb{J}_\text{TS}$.
The neutral Floquet mode $\tilde{\mb{p}}_0$ satisfies $\dot{\tilde{\mb{p}}}_0 = \mb{J}(t)\tilde{\mb{p}}_0$, which can be rearranged as $\mb{J}(t)\tilde{\mb{p}}_0 - \dot{\tilde{\mb{p}}}_0 = 0$.
From the definition of the time-spectral Jacobian, this is equivalent to
\begin{equation}
\mathcal{J}_\text{TS}\tilde{\mb{p}}_0 = 0.
\end{equation}
Therefore, the neutral Floquet mode falls in the null space of $\mathcal{J}_\text{TS}$ and implies that the collocation matrix $\mb{J}_\text{TS}$ has a corresponding zero eigenvalue.
Since Floquet pairs are non-unique under the shift $\lambda \mapsto \lambda + jk\omega_0$ with a corresponding rephasing of the periodic part $\mb{p}(t)\mapsto \mb{p}(t)e^{-jk\omega_0 t}$, the neutral eigenpair $(0,\tilde{\mb{p}}_0)$ generates a vertical ladder of shifted copies $(jk\omega_0, \tilde{\mb{p}}_0(\theta)e^{-jk\theta})$ for any integer $k$.
This can be shown by applying the operator on the shifted modes
\begin{equation}
\begin{aligned}
\mathcal{J}_\text{TS} (\tilde{\mb{p}}_0 e^{-jk\theta}) &= \mb{J}(t)(\tilde{\mb{p}}_0 e^{-jk\theta}) - \omega_0 \frac{\d}{\d\theta}(\tilde{\mb{p}}_0 e^{-jk\theta}) \\
&= \left( \mb{J}(t)\tilde{\mb{p}}_0 - \omega_0 \frac{\d\tilde{\mb{p}}_0}{\d\theta} \right) e^{-jk\theta} + j k \omega_0 \tilde{\mb{p}}_0 e^{-jk\theta} ,
\end{aligned}
\end{equation}
since the first parenthesis equals $\mathcal{J}_\text{TS}\tilde{\mb{p}}_0 = 0$.
This confirms the modulated neutral Floquet modes $\tilde{\mb{p}}_k = \tilde{\mb{p}}_0 e^{-jk\theta}$ are right eigenvectors of $\mathcal{J}_\text{TS}$ with a corresponding eigenvalue $jk\omega_0$.
In the collocation discretization, the matrix $\mb{J}_\text{TS}$ represents $\mathcal{J}_\text{TS}$ on the $n_\text{TS}$-point trigonometric polynomial subspace.
Therefore, for the unique DFT modes $|k| \leq (n_\text{TS}-1)/2$, the corresponding eigenvalues $jk\omega_0$ appear in the spectrum of $\mb{J}_\text{TS}$ (with aliasing outside this range).
A short discrete proof is provided in Appendix~\ref{app:shifted_eigenpairs}.

Similarly, there exists an adjoint neutral mode $\tilde{\mb{q}}_0 = \mathrm{Adj}(\tilde{\mb{p}}_0)$ which is the corresponding left null vector of $\mathcal{J}_\text{TS}$ satisfying $\mathcal{J}_\text{TS}^*\tilde{\mb{q}}_0 = 0$.
Analogous to the neutral Floquet mode, the modulated adjoint modes $\tilde{\mb{q}}_k = \tilde{\mb{q}}_0e^{-jk\theta}$ also appear as left eigenvectors of $\mathcal{J}_\text{TS}$ with eigenvalues $-jk\omega_0$ for the resolved modes $|k| \leq (n_\text{TS}-1)/2$.
We note that the left and right eigenvectors of $\mathcal{J}_\text{TS}$ for a bi-orthogonal bases leading to the normalization convection
\begin{equation}
\langle \tilde{\mb{p}}_0, \tilde{\mb{q}}_0 \rangle = \int_{0}^{T_0} \tilde{\mb{p}}_0^*(t)\tilde{\mb{q}}_0(t)\,\mathrm{d}t = 1,
\end{equation}
where $\langle \tilde{\mb{p}}_0, \tilde{\mb{q}}_0 \rangle$ is the inner product of the adjoint and neutral modes.
It follows that for the modulated modes $\langle \tilde{\mb{p}}_k, \tilde{\mb{q}}_k \rangle = 1$
For information on how to efficiently compute the neutral and adjoint modes even in the case of a under-resolved base flow, we refer the reader to Appendix~\ref{app:modes}.

\subsection{Null Space of the Operator $\mb{L}_\text{TS}$}

At resonance, the time-spectral resolvent operator $\mb{L}_\text{TS} = j\omega_f \mb{I} - \mb{J}_\text{TS}$ becomes singular whenever $j\omega_f$ coincides with an eigenvalue of $\mb{J}_\text{TS}$.
A left null vector of $\mb{L}_\text{TS}$ satisfies
\begin{equation}
\begin{aligned}
\mb{L}_\text{TS}^*\mb{v} &= (-j\omega_f\mb{I} - \mb{J}^*_\text{TS})\mb{v} = 0\\
\mb{J}_\text{TS}^*\mb{v} &= -j\omega_f\mb{v},
\end{aligned}
\end{equation}
indicating that any left null vector of $\mb{L}_\text{TS}$ must be a left eigenvector of $\mb{J}_\text{TS}$ with corresponding eigenvalue $-j\omega_f$.
Using the resonance condition $\omega_f = k\omega_0$ it can be seen that $\mb{J}_\text{TS}$ has exactly one eigenvalue $-jk\omega_0$ with corresponding eigenvalue $\tilde{\mb{q}}_k$.
It can similarly be shown that the corresponding right null vector of $\mb{L}_\text{TS}$ is the modulated neutral mode $\tilde{\mb{p}}_0$.
We can then conclude that the null space of $\mb{L}_\text{TS}$ is one-dimensional and spanned by right and left null vectors $\tilde{\mb{p}}_k$ and $\tilde{\mb{q}}_k$.

\subsection{Solvability Condition at Resonance}

Due to the singularity of the operator $\mb{L}_\text{TS}$ at resonance, the linear system $\mb{L}_\text{TS}\tilde{\boldsymbol{\eta}} = \mb{B}_\text{TS}\tilde{\mb{f}}$ only has a unique solution when the solvability condition presented in this section holds.
For a solution $\tilde{\boldsymbol{\eta}}$ to exist, $\tilde{\mb{f}} \in \mathrm{Range}(\mb{L}_\text{TS})$ must be true.
The fundamental theorem of linear algebra states $\mathrm{Range}(\mb{L}_\text{TS}) = \mathrm{Null}(\mb{L}_\text{TS}^*)^\perp$, where $\mathrm{Null}$ denotes the null space and $\Box^\perp$ denotes the orthogonal complement.
Since the null space of $\mb{L}_\text{TS}^*$ is one-dimensional and spanned by $\tilde{\mb{q}}_k$, the solvability condition can be given as
\begin{equation}
\langle \tilde{\mb{q}}_k, \mb{B}_\text{TS}\hat{\mb{f}} \rangle = \tilde{\mb{q}}_k^* \mb{B}_\text{TS}\hat{\mb{f}} = 0.
\label{eq:solvability}
\end{equation}
When this condition holds, solutions $\tilde{\boldsymbol{\eta}}$ exist but are not unique and can contain an arbitrary component in the direction of $\mathrm{Null}(\mb{L}_\text{TS}) = \tilde{\mb{p}}_k$.
The transverse resolvent operator presented in Section~\ref{sec:transverse} selects a unique transverse response by projecting into a subspace in which $\mb{L}_\text{TS}$ is bijective.

It can also be seen that the singularity of $\mb{L}_\text{TS}$ is not a numerical artifact but a physical phenomena corresponding to secular growth of the linearized dynamics.
% TODO:
It is later shown in Section~\ref{sec:physical_response} that a physical forcing which is aligned with the modulated adjoint direction at resonance leads to unbounded phase growth of the linear dynamics.
When this occurs, this violates the steady state assumption of the resolvent framework.
Therefore, the full resolvent gain
\begin{equation}
G(\omega_f) = \sigma_\text{max}(\mb{L}_\text{TS}^{-1}\mb{B}_\text{TS}) \to \infty \quad \text{as } \omega_f \to k\omega_0
\end{equation}
diverges at resonance, which is the physically correct prediction for this case.

\subsection{Transverse Resolvent Analysis}
\label{sec:transverse}

The full TSR operator (Eq.~\ref{eq:discrete_op}) maps from an arbitrary forcing $\tilde{\mb{f}}$ to the resulting response $\tilde{\boldsymbol{\eta}}$ without considering the excitation of the neutral Floquet mode which can lead to unbounded phase drift.
In this section, we follow the approach by~\citet{Padovan2022} to define a subspace $\Sigma_\text{TS}$ in which the TSR operator $\mb{L}_\text{TS}$ is both bijective and invariant.
When the TSR operator is restricted to this subspace, the projected forcing $\tilde{\mb{f}}_{\Sigma_\text{TS}}$ is required to be orthogonal to the modulated adjoint neutral mode.
We later show that this specific forcing excites a response $\tilde{\boldsymbol{\eta}}_\Sigma$ that contains a zero mean phase drift component.
This projection isolates the unique solution $\tilde{\boldsymbol{\eta}}_\Sigma$ which contains zero phase drift at every forcing frequency, including resonance.

In many cases, the shape deformation of the limit cycle (transverse response) is of greater physical interest than the phase drift.
For example, in aeroelastic flutter, transverse perturbations change the oscillation amplitude and structural loads, while phase perturbations merely shift the timing.
To analyze only the transverse response, we define the transverse resolvent gain%\emph{transverse resolvent gain}
\begin{equation}
G_\Sigma(\omega_f) = \max_{\hat{\mb{f}}_{\Sigma} \neq 0} \frac{\|\tilde{\boldsymbol{\eta}}_{\Sigma}\|_2}{\|\hat{\mb{f}}_{\Sigma}\|_2},
\end{equation}
which measures the maximum amplification from forcing to response in the subspace $\Sigma_\text{TS}$.

To define this subspace, we introduce the oblique projector
\begin{equation}
\mb{P}_\omega = \mb{I} - \tilde{\mb{p}}_\omega \tilde{\mb{q}}_\omega^* ,
\end{equation}
where $\tilde{\mb{p}}_\omega = \tilde{\mb{p}}_0 e^{-j\f{\omega_f}{\omega_0}\theta}$ corresponds to modulating the neutral mode with the phase $\omega_f t$.
Similarly for the adjoint mode, $\tilde{\mb{q}}_\omega = \tilde{\mb{q}}_0 e^{-j\f{\omega_f}{\omega_0}\theta}$.
We note that at resonance ($\omega_f = k\omega_0$) these modulated modes become $\tilde{\mb{p}}_k$ and $\tilde{\mb{q}}_k$ for the specific integer $k$.
We then formally define the subspace $\Sigma_\text{TS}$ as the range of the projector $\mb{P}_\omega$.
It can be shown that this subspace is orthogonal to the modulated adjoint $\tilde{\mb{q}}_\omega$
\begin{equation}
\tilde{\mb{q}}_\omega^*{\mb{P}}_\omega = \tilde{\mb{q}}_\omega^*(\mb{I} - \tilde{\mb{p}}_\omega \tilde{\mb{q}}_\omega^*) = \tilde{\mb{q}}_\omega^* - \tilde{\mb{q}}_\omega^* = 0,
\end{equation}
by recognizing that $\tilde{\mb{q}}_\omega^*\tilde{\mb{p}}_\omega = 1$.
It can also be shown that this subspace filters out any component in the direction of $\tilde{\mb{p}}_\omega$ since $\mb{P}_k\tilde{\mb{p}}_\omega = 0$.

Using this projector to enforce $\tilde{\mb{f}}_\Sigma = \mb{P}_\omega\tilde{\mb{f}}$ and $\tilde{\boldsymbol{\eta}}_\Sigma = \mb{P}_\omega\tilde{\boldsymbol{\eta}}$ defines the transverse resolvent operator
\begin{equation}
\mb{R}_\Sigma \coloneqq \mb{P}_\omega \mb{L}_\text{TS}^{-1} \mb{P}_\omega \mb{B}_\text{TS},
\end{equation}
which maps from $\tilde{\mb{f}}_\Sigma \rightarrow \tilde{\boldsymbol{\eta}}_\Sigma$.
We then have that $G_\Sigma(\omega_f) = \sigma_{\max}(\mb{R}_\Sigma)$ is the transverse resolvent gain.
The optimal forcing and response modes $\tilde{\mb{f}}_\Sigma$ and $\tilde{\boldsymbol{\eta}}_\Sigma$ can then be given as the right $\mb{v}_1$ and left $\mb{u}_1$ singular vectors of $\mb{R}_\Sigma$.
We note that since $\mb{L}_\text{TS}$ is non-invertible this equation cannot be directly solved.
Instead we efficiently solve the system and take the SVD using the process described in Appendix~\ref{app:efficient_resolvent}.

This formulation requires that a unique solution $\tilde{\boldsymbol{\eta}}_\Sigma \in \Sigma_\text{TS}$ exist for a forcing $\tilde{\mb{f}}_\Sigma \in \Sigma_\text{TS}$ (i.e the operator $\mb{L}_\text{TS}$ must be bijective in $\Sigma$).
A brief proof of this as well as the invariance of $\mb{L}_\text{TS}$ is given in Appendix~\ref{app:invariance_bijectivity}.
% TODO:

\subsection{Reconstruction of the Physical Response}
\label{sec:reconstruction}

The SVD of the transverse TSR operator $\mb{R}_{\Sigma}(\omega_f)$ yields the forcing $\tilde{\mb{f}}_{\Sigma}$ which does not excite any phase drift in the response $\tilde{\boldsymbol{\eta}}_{\Sigma}$.
However, when applying this forcing to the linearized dynamics to compare the predicted response to time-accurate integration as we do in Section~\ref{sec:num_ex}, the initial condition of the simulation can cause the neutral mode to show up in the response as a constant phase shift.
Specifically, the projected TSR operator predicts the particular solution of the forced linear equations and neglects the homogeneous component spanned by the neutral Floquet mode $\tilde{\mb{p}}_0(t)$.
In direct numerical simulations, the initial condition starting from the base flow $\tilde{\boldsymbol{\eta}}(0) = \mb{0}$ excites this homogeneous (neutral and non-decaying) mode.
% MH-HS: I changed this back to zero isntead of \eta_0 since if you do not start on the base flow, the neutral mode is not excited and this is not relevent.
Consequently, the full physical response observed in direct simulation is $\tilde{\boldsymbol{\eta}}(t) = \tilde{\boldsymbol{\eta}}_{\Sigma}(t) + c\tilde{\mb{p}}_0(t)$, where the drift coefficient $c \in \mathbb{R}$ represents a permanent phase shift.
To determine $c$, we enforce the initial condition by projecting onto the adjoint neutral mode direction
\begin{equation}
c = -\frac{\langle  \text{Re}(\tilde{\boldsymbol{\eta}}_{\Sigma}(0)), \tilde{\mb{q}}_0(0) \rangle}{\langle \tilde{\mb{p}}_0(0), \tilde{\mb{q}}_0(0) \rangle}.
\label{eq:drift_coeff}
\end{equation}
% MH-HS: same comment as above
This reconstruction bridges the gap between the frequency-domain TSR prediction and time-domain simulation.

Finally, the addition of the phase component to the response increases the resolvent gain.
In many practical cases, the amplification of the transverse dynamics significantly exceeds the magnitude of the permanent phase shift.
Consequently, the contribution of the phase shift to the total energy is often sub-dominant, resulting in only a minor difference between the projected $G_\Sigma(\omega_f)$ gain and the gain computed from time-accurate integration starting from a zero initial condition.
The reconstructed resolvent gain $G_\text{rec}(\omega_f)$ which includes both the transverse particular response and the homogeneous neutral mode $\tilde{\boldsymbol{\eta}}(t) = \tilde{\boldsymbol{\eta}}_{\Sigma}(t) + c\tilde{\mb{p}}_0(t)$ can be constructed as
\begin{equation}
\begin{split}
G_\text{rec}(\omega_f) &= \max_{\tilde{\mb{f}}_{\Sigma} \neq \mb{0}} \frac{\|\tilde{\boldsymbol{\eta}}_{\Sigma} (t) + c(\tilde{\mb{f}}_{\Sigma})\tilde{\mb{p}}_0(t)\|_2}{\|\tilde{\mb{f}}_{\Sigma}\|_2} \\
&= \max_{\tilde{\mb{f}}_{\Sigma} \neq \mb{0}} \sqrt{ \left( \frac{\|\tilde{\boldsymbol{\eta}}_{\Sigma}\|_2}{\|\tilde{\mb{f}}_{\Sigma}\|_2} \right)^2 + \left( \frac{|c(\tilde{\mb{f}}_{\Sigma})|\|\tilde{\mb{p}}_0\|_2}{\|\tilde{\mb{f}}_{\Sigma}\|_2} \right)^2 + \frac{2\text{Re}\langle \tilde{\boldsymbol{\eta}}_{\Sigma}, c(\tilde{\mb{f}}_{\Sigma})\tilde{\mb{p}}_0 \rangle}{\|\tilde{\mb{f}}_{\Sigma}\|_2^2} },
\end{split}
\label{eq:full_gain}
\end{equation}
where the second equality follows from the expansion of the norm $\|\mathbf{u}+\mathbf{v}\|_2$, which includes a cross-term $2\text{Re}\langle \tilde{\boldsymbol{\eta}}_{\Sigma}, c\tilde{\mb{p}}_0 \rangle$ because the transverse response and neutral modes are not orthogonal in non-normal systems.
If $\tilde{\mb{f}}_{\Sigma}$ is chosen as the optimal forcing for $\mb{R}_{\Sigma}(\omega_f)$, then $\|\tilde{\boldsymbol{\eta}}_{\Sigma}\|_2/\|\tilde{\mb{f}}_{\Sigma}\|_2 = G_\Sigma(\omega_f) = \sigma_\text{max}(\mb{R}_{\Sigma})$.

\subsection{Physical Requirement for Zero Mean Phase Drift}
\label{sec:physical_response}

It has been shown that with the transverse resolvent operator a unique transverse solution $\tilde{\boldsymbol{\eta}}_\Sigma$ exists for forcing orthogonal to the adjoint mode (defined by the subspace $\Sigma$).
However, until now the link to the physics of the linearized system has not been addressed.
In this section, we show that the forcing which satisfies $\langle \tilde{\mb{q}}_k, \tilde{\mb{f}}\rangle = 0$ excites only the shape deformation of the physical response and zero mean phase drift.

Given the linear system
\begin{equation}
\label{eq:linear_dynamics}
\dot{\boldsymbol{\eta}}(t) = \mb{J}(t){\boldsymbol{\eta}}(t) + \hat{\mb{f}}(\omega_0 t)e^{j\omega_f t},
\end{equation}
where $\hat{\mb{f}}(\omega_0 t)$ is a $2\pi$-periodic envelope which corresponds to the quasi-periodic forcing $\mb{f}(t) = \hat{\mb{f}}(\omega_0 t)e^{j\omega_f t}$.
The physical response can be decomposed into a phase drift component and a transverse component $\boldsymbol{\eta} = c(t)\tilde{\mb{p}}_0(t) + \mb{v}(t)$.
Here $\tilde{\mb{p}}_0(t) = \dot{\overbar{\mb{w}}}$ corresponds to the direction of phase shift in the linearized system.
This can be shown via first order Taylor series expansion
\begin{equation}
\overbar{\mb{w}}(t + \epsilon) \approx \overbar{\mb{w}}(t) + \epsilon\dot{\overbar{\mb{w}}}(t) = \overbar{\mb{w}}(t) + \epsilon\tilde{\mb{p}}_0(t),
\end{equation}
for a phase perturbation $\epsilon$.
The phase drift amplitude $c(t)$ is a scalar function that represents the instantaneous amplitude of this phase perturbation at any time $t$.

The vector $\mathbf{v}(t)$ represents the transverse component of the perturbation from the base flow, capturing the shape deformation of the limit cycle.
According to Floquet theory, $\mathbf{v}(t)$ is spanned entirely by the strictly stable right Floquet modes ($\lambda_i < 0$).
By the property of bi-orthogonality, the adjoint mode $\tilde{\mb{q}}_0(t)$ which is the left Floquet mode corresponding to $\lambda_0 = 0$ must be strictly orthogonal to all right Floquet modes associated with different eigenvalues ($\lambda_i \neq 0$).
Consequently, the adjoint mode is mathematically guaranteed to be orthogonal to the transverse dynamics at all times, such that $\tilde{\mb{{q}}}(t)^* \mathbf{v}(t) = 0$.

It is shown in Appendix~\ref{app:phase_drift_derivation} the time derivative of the phase perturbation magnitude $\dot{c}(t)$ can be isolated by substituting in $\boldsymbol{\eta} = c(t)\tilde{\mb{p}}_0(t) + \mb{v}(t)$ to Eq~\ref{eq:linear_dynamics}.
The result of this is
\begin{equation}
\dot{c}(t) = \tilde{\mb{q}}_0(t)^*\hat{\mb{f}}(\omega_0t)e^{j\omega_f t},
\end{equation}
which can be re-arranged to reveal the modulated adjoint mode
\begin{equation}
\dot{c}(t) = (\tilde{\mb{q}}_0(t)e^{-j\f{\omega_f}{\omega_0}\theta})^*\hat{\mb{f}}(\omega_0t) = \mb{q}_\omega(t)^*\hat{\mb{f}}(\omega_0t),
\end{equation}
since $\theta = \omega_0t$ and $(e^{j\omega_f t})^* = e^{-j\omega_f t}$.

We can determine the phase shift accumulated over one full period of the limit cycle $\Delta c$ with the integral
\begin{equation}
\Delta c = \int_0^{T_0} \dot{c}(t)dt = \int_0^{T_0} \tilde{\mb{q}}_\omega(t)^*\hat{\mb{f}}(t)dt.
\end{equation}
Changing to the frequency domain with the phase variable $\theta = \omega_0t$ gives
\begin{equation}
    \Delta c = \f{1}{\omega_0}\int_0^{2\pi} \tilde{\mb{q}}_\omega(\theta)^*\hat{\mb{f}}(\theta)d\theta \propto \langle\mb{q}_k,\hat{\mb{f}}\rangle,
\end{equation}
where the proportionality arises from $\int_0^{2\pi} \mb{q}_k(\theta)^*\hat{\mb{f}}(\theta)d\theta = \langle\mb{q}_k,\hat{\mb{f}}\rangle$.

Finally we can conclude that for a forcing $\hat{\mb{f}}_\Sigma \in \Sigma$ which satisfies $\langle\mb{q}_k,\hat{\mb{f}}_\Sigma\rangle = 0$, there is zero phase drift accumulation over any period of the orbit.
This links the TSR analysis where we require forcing orthogonal to the modulated adjoint mode to the physics of the system, proving that with our formulation no phase drift will accumulate in the linear dynamics.

\subsection{Summary of TSR Formulations}
\label{sec:formulation_summary}
% not sure if this needs its own section

The previous sections present three different TSR gains: 1. The full TSR gain $G(\omega_f)$ which is suitable for non-autonomous systems.
When this formulation is used in the context of autonomous systems it does not restrict the forcing to be orthogonal to the adjoint neutral mode, which can cause the response $\tilde{\boldsymbol{\eta}}(t)$ to exhibit phase drift.
At resonance, this phase drift is infinite which leads to singularities in the operator $\mb{L}_\text{TS}$, corresponding to secular growth of the response.
2. The transverse TSR gain $G_\Sigma(\omega_f)$ which calculates the amplification of a purely transverse response to a forcing orthogonal to the adjoint neutral mode.
3. The reconstructed TSR gain $G_\text{rec}(\omega_f)$ which considers the particular response from $G_\Sigma(\omega_f)$ as well as the homogeneous response spanned by the neutral mode.
This reconstructed gain can be compared directly to time-accurate integration of the optimally forced (with $\tilde{\mb{f}}_{\Sigma_\text{TS}}$) linear dynamics starting from a initial condition on the base flow.
All gain formulations are summarized in Table~\ref{tab:tsr_summary} with their key properties.

\begin{table}[h!]
\centering
\caption{TSR Gain Summary}
\label{tab:tsr_summary}
\ra{1.3}
\resizebox{\columnwidth}{!}{%
\begin{tabular}{@{}lcccc@{}}
\toprule
\textbf{Formulation} & \textbf{Gain Definition} & \textbf{Linear System} & \textbf{System Type} & \textbf{Singularity?} \\
\midrule
Full & $G(\omega_f) = \sigma_{\text{max}}(\mb{R}_{\text{TS}})$ & $\mb{L}_{\text{TS}}\tilde{\boldsymbol{\eta}} = \mb{B}_{\text{TS}}\tilde{\mb{f}}$ & Non-autonomous & Yes \\
Transverse & $G_\Sigma(\omega_f) = \sigma_{\text{max}}(\mb{R}_{\Sigma_\text{TS}})$ & $\mb{L}_{\text{TS}}\mb{P}_{\text{TS}}\tilde{\boldsymbol{\eta}} = \mb{P}_{\text{TS}}\mb{B}_{\text{TS}}\tilde{\mb{f}}$ & Autonomous & No \\
Reconstructed & $G_\text{rec}(\omega_f)$ (Eq.~\ref{eq:full_gain}) & $\mb{L}_{\text{TS}}\mb{P}_{\text{TS}}\tilde{\boldsymbol{\eta}} = \mb{P}_{\text{TS}}\mb{B}_{\text{TS}}\tilde{\mb{f}}$ & Autonomous & No \\
\bottomrule
\end{tabular}
}
\end{table}

\section{Convergence Analysis}
\label{sec:convergence}

In this section, we establish spectral convergence of both the harmonic resolvent (HR) and the time-spectral resolvent (TSR) operators under standard analyticity and stability assumptions. 
Throughout this section we consider the regular case (non-autonomous base flow or $\omega_f \neq 0$), so that all operators are invertible.
Let $\sigma^{(n_{\text{har}})}$ denote the maximum singular value of the truncated harmonic resolvent $\mb{R}_{\text{HR}}^{(n_{\text{har}})}$. Define the limiting gain
\begin{equation}
\sigma_\infty \coloneqq \lim_{n_{\text{har}}\to\infty} \sigma^{(n_{\text{har}})}.
\label{eq:sigma_limit_def}
\end{equation}
Existence of this limit and the exponential convergence rate are proved in Proposition~\ref{prop:cont_conv}.
% The autonomous case ($\omega_f=0$) follows by restriction to the projected subspace as described by Padovan and Rowley~\cite{Padovan2022}.

\subsection{Assumptions}
\label{sec:conv_assumptions}

Let $\mb{J}(t)$ be $T_0$-periodic and analytic in $t$, and let the underlying linear time-periodic system be exponentially stable. We additionally assume finite-section stability of the truncated harmonic resolvent operators,
\begin{equation}
\sup_{n_{\text{har}}} \left\|\left(\mb{L}_{\text{HR}}^{(n_{\text{har}})}\right)^{-1}\right\|_2 < \infty,
\label{eq:uniform_resolvent_bound}
\end{equation}
which holds when the Floquet exponents are uniformly separated from the imaginary axis and the Fourier coefficients of $\mb{J}$ decay exponentially (a standard consequence of finite-section results for exponentially decaying block-Toeplitz operators). Under these conditions, $\mb{R}_\infty$ is bounded and the gains of the truncated systems are uniformly controlled.
Throughout, $\|\cdot\|_2$ denotes the induced 2-norm; when the subscript is omitted, this norm is intended.
These hypotheses correspond to the non-autonomous setting emphasized earlier (no neutral Floquet mode); in the autonomous case with $\omega_f = 0$, the projected resolvent of~\citet{Padovan2022} would replace the regular inverses used here.

\begin{lemma}[Fourier Coefficient Decay]\label{lem:fourier_decay}
If $g(t)$ is analytic in a strip $|\operatorname{Im} t| < \rho$, then its Fourier coefficients satisfy $\|\hat{g}_k\| \leq C e^{-\alpha|k|}$ for $\alpha < \rho$~\cite{Trefethen2000, boyd2013chebyshev}.
\end{lemma}

When the base flow is computed via harmonic balance with $n$ harmonics, the Jacobian $\mb{J}^{(n)}(t)$ is evaluated at the truncated base flow $\bar{\mb{q}}^{(n)}(t)$, yielding Fourier coefficients $\hat{\mb{J}}_k^{(n)}$. We assume:
\begin{enumerate}
\item[(J1)] \textit{Coefficient decay:} The exact Jacobian coefficients satisfy $\|\hat{\mb{J}}_k^{(\infty)}\| \leq C e^{-\alpha|k|}$.
\item[(J2)] \textit{Coefficient convergence:} For $|k| \leq n$, $\|\hat{\mb{J}}_k^{(n)} - \hat{\mb{J}}_k^{(\infty)}\| \leq C e^{-\alpha n}$.
\item[(J3)] \textit{Uniform invertibility:} $\|(\mb{L}_{\text{HR}}^{(n)})^{-1}\| \leq M$ for all $n$ and some $M > 0$.
\end{enumerate}
Conditions (J1)--(J2) hold when the base flow and right-hand side are analytic: (J1) follows from Lemma~\ref{lem:fourier_decay}, and (J2) follows because the base flow error $\|\bar{\mb{q}}^{(n)} - \bar{\mb{q}}^{(\infty)}\|$ decays exponentially for analytic problems, and the Jacobian depends smoothly on the base flow. Condition (J3) holds when the forcing frequency $\omega_f$ is bounded away from the Floquet exponents of the linearized system; this is the regime where resolvent analysis is meaningful.

\begin{proposition}[Spectral Convergence of HR]\label{prop:cont_conv}
Let $\sigma^{(n_{\text{har}})}$ be the maximal singular value of the harmonic resolvent with $n_{\text{har}}$ harmonics, and $\sigma_\infty$ as in Eq.~\ref{eq:sigma_limit_def}. Then
\begin{equation}
|\sigma^{(n_{\text{har}})} - \sigma_\infty| \leq C e^{-\alpha n_{\text{har}}}.
\label{eq:HR_convergence}
\end{equation}
\end{proposition}

\begin{proof}
The strategy is to show that $\{\sigma^{(n)}\}$ is a Cauchy sequence with exponential rate by embedding coarse truncations into finer harmonic spaces.
For $m>n$, embed $\mb{L}_{\text{HR}}^{(n)}$ into the $m$-harmonic space by zero-padding off-diagonal blocks (Appendix~\ref{app:padding}); call the result $\widetilde{\mb{L}}_{\text{HR}}^{(n\to m)}$.
The added diagonal blocks are $\mb{A}_k^{(n)} = j(\omega_f + k\omega_0)\mb{I} - \hat{\mb{J}}_0^{(n)}$ for $|k|>n$, using the Jacobian from the $n$-truncated base flow.

Similarly pad $\mb{B}_{\text{HR}}^{(n)}$ to $\widetilde{\mb{B}}_{\text{HR}}^{(n\to m)}$; since $\mb{B}$ is nonzero only in the $k=0$ block, $\widetilde{\mb{B}}_{\text{HR}}^{(n\to m)} = \mb{B}_{\text{HR}}^{(m)}$ (Appendix~\ref{app:padding}).
Since the added blocks are decoupled and $\widetilde{\mb{B}}_{\text{HR}}^{(n\to m)}$ is zero there, the lifted resolvent $\widetilde{\mb{R}}_{\text{HR}}^{(n\to m)} = (\widetilde{\mb{L}}_{\text{HR}}^{(n\to m)})^{-1}\widetilde{\mb{B}}_{\text{HR}}^{(n\to m)}$ shares the same nonzero singular values as $\mb{R}_{\text{HR}}^{(n)}$. The full truncation is $\mb{R}_{\text{HR}}^{(m)} = (\mb{L}_{\text{HR}}^{(m)})^{-1}\mb{B}_{\text{HR}}^{(m)}$, built from the $m$-truncated base flow. The difference $\mb{L}_{\text{HR}}^{(m)} - \widetilde{\mb{L}}_{\text{HR}}^{(n\to m)}$ has two contributions (Appendix~\ref{app:padding}): (i) neglected couplings to/from outer harmonics $|k|>n$, and (ii) Jacobian coefficient differences $\hat{\mb{J}}_k^{(m)} - \hat{\mb{J}}_k^{(n)}$. By the Wiener norm bound for block-Toeplitz matrices ($\|\mb{L}\|_2 \leq \sum_k \|\hat{\mb{L}}_k\|$, see~\citet{bottcher2012spectral}) and assumptions (J1)--(J2),
\begin{equation}
\|\mb{L}_{\text{HR}}^{(m)} - \widetilde{\mb{L}}_{\text{HR}}^{(n\to m)}\|_2 \leq \underbrace{\sum_{|k|>n} \|\hat{\mb{J}}_k^{(m)}\|}_{\text{tail: (J1)}} + \underbrace{\sum_{|k|\leq n} \|\hat{\mb{J}}_k^{(m)} - \hat{\mb{J}}_k^{(n)}\|}_{\text{core: (J2)}} \leq C_1 e^{-\alpha n}.
\label{eq:L_diff_bound}
\end{equation}
Using the identity $\mb{A}^{-1}-\mb{B}^{-1}=\mb{A}^{-1}(\mb{B}-\mb{A})\mb{B}^{-1}$,
\begin{equation}
\widetilde{\mb{R}}_{\text{HR}}^{(n\to m)} - \mb{R}_{\text{HR}}^{(m)} = (\widetilde{\mb{L}}_{\text{HR}}^{(n\to m)})^{-1} (\mb{L}_{\text{HR}}^{(m)} - \widetilde{\mb{L}}_{\text{HR}}^{(n\to m)}) (\mb{L}_{\text{HR}}^{(m)})^{-1} \mb{B}_{\text{HR}}^{(m)},
\end{equation}
and by (J3), $\|\widetilde{\mb{R}}_{\text{HR}}^{(n\to m)} - \mb{R}_{\text{HR}}^{(m)}\|_2 \leq C_2 e^{-\alpha n}$ where $C_2 = C_1 M^2 \|\mb{B}_{\text{HR}}\|$.
Since $\sigma^{(n)} = \sigma_1(\widetilde{\mb{R}}_{\text{HR}}^{(n\to m)})$ (Appendix~\ref{app:padding}), Weyl's inequality gives $|\sigma^{(m)} - \sigma^{(n)}| \leq C_2 e^{-\alpha n}$, so $\{\sigma^{(m)}\}$ is Cauchy and converges to $\sigma_\infty$; this proves Eq.~\ref{eq:HR_convergence}.
\end{proof}

\begin{theorem}[Spectral Convergence of TSR]\label{thm:tsr_conv}
Let $\sigma_{\text{TS}}^{(n_{\text{TS}})}$ be the maximal singular value of the TSR operator with odd $n_{\text{TS}}$ points. Then
\begin{equation}
|\sigma_{\text{TS}}^{(n_{\text{TS}})} - \sigma_\infty| \leq C' e^{-\alpha' n_{\text{TS}}}, \qquad \alpha' > 0.
\label{eq:TSR_convergence}
\end{equation}
\end{theorem}

\begin{proof}
With odd $n_{\text{TS}}$, TSR represents $n_{\text{har}} = (n_{\text{TS}}-1)/2$ harmonics.
The DFT relates the resolvents by $\mb{R}_{\text{TS}}^{(n_{\text{TS}})} = \mathcal{F}^{-1}\mb{R}_{\text{HR}}^{(n_{\text{har}})}\mathcal{F}$ (see Appendix~\ref{app:equivalence}), so $\sigma_{\text{TS}}^{(n_{\text{TS}})} = \sigma^{(n_{\text{har}})}$ since the DFT is unitary.
The result follows directly from Proposition~\ref{prop:cont_conv}.
\end{proof}

\section{Numerical Examples}
\label{sec:num_ex}

In this section, we validate our TSR formulation against time accurate integration and demonstrate the convergence properties of the TSR for three numerical examples.
First, we demonstrate the harmonic forcing case (Appendix~\ref{app:parseval}) and the convergence properties of the TSR formulation in the non-autonomous Mathieu oscillator.
We also validate the TSR formulation using the autonomous van der Pol oscillator by showing the TSR analysis can predict the optimal resolvent gain and response when the system is subject to an optimal quasi-periodic forcing.
We then demonstrate the ability of the TSR formulation to handle high dimensional systems in the 1-D complex Ginzburg--Landau partial differential equation, where we use the quasi-periodic formulation and predict the optimal gain and response.
In all examples, the TSR predictions align with time accurate integration which demonstrates that the method accurately predicts gains and response modes across a range of system complexities.

\subsection{Parametrically Forced Mathieu Oscillator}

The Mathieu oscillator is a second order non-autonomous linear time periodic ordinary differential equation with a periodic Jacobian.
The governing equation is given by
\begin{equation}
\ddot{y} + 2\zeta\dot{y} + \left(\omega_n^2 + \alpha\cos(\omega_0 t)\right)y = u(t),
\label{eq:mathieu_governing}
\end{equation}
where $\omega_n = 1.0$ is the natural frequency, $\zeta = 0.1$ is the damping ratio, $\alpha = 0.2$ is the parametric modulation amplitude, and $\omega_0 = \sqrt{2}\omega_n$ is the parametric modulation frequency.
The base frequency of this system is $\omega_0$, corresponding to a period $T_0 = 2\pi/\omega_0$.
The system is subject to a harmonic control input $u(t) = \mathrm{Re}(\hat{f} e^{j\omega_f t})$ with forcing frequency $\omega_f$.
Note that in this example, we restrict the forcing to be purely sinusiodal at frequency $\omega_f$ using the formulation described in Appendix~\ref{app:parseval}.

The system can be transformed into a first order state space system and written as
\begin{equation}
\frac{\d}{\d t}\begin{bmatrix} y \\ \dot{y} \end{bmatrix} =
\begin{bmatrix} 0 & 1 \\ -\omega_n^2 - \alpha\cos(\omega_0 t) & -2\zeta \end{bmatrix}
\begin{bmatrix} y \\ \dot{y} \end{bmatrix} +
\begin{bmatrix} 0 \\ 1 \end{bmatrix} u(t),
\end{equation}
where the state is given by $\mb{w} = [y, \dot{y}]^\top$.

The time periodic Jacobian and input matrix $\mb{B}$ can then be written as
\begin{equation}
\mb{J}(t) = \begin{bmatrix} 0 & 1 \\ -\omega_n^2 - \alpha\cos(\omega_0 t) & -2\zeta \end{bmatrix}, \qquad
\mb{B} = \begin{bmatrix} 0 \\ 1 \end{bmatrix},
\end{equation}
which comprises a linear time periodic system.
Note that in the case of the Mathieu oscillator, while the Jacobian is $T_0$-periodic, the base flow is a trivial zero vector $\bar{\mb{w}} = \mb{0}$.
This results in the perturbation $\boldsymbol{\eta}$ being exactly equal to the forced state $\mb{w}$.
Therefore, the perturbation can be written in the form $\dot{\boldsymbol{\eta}}(t) = \mb{J}(t)\boldsymbol{\eta} + \mb{B}\mb{u}(t)$.

The system can then be discretized with $2\pi$-periodic phase $\theta = \omega_0 t$ and we use $n_{\text{TS}} = 5$ collocation points.
The discretized periodic Jacobian can then be given as
\begin{equation}
\mb{J}(\theta_j) = \begin{bmatrix} 0 & 1 \\ -\omega_n^2 - \alpha\cos(\theta_j) & -2\zeta \end{bmatrix}, 
\end{equation}
and the full TSR operator can be formed
\begin{equation}
\mb{R}_{\text{TS}}(\omega_f) = (j\omega_f \mb{I} - \mb{J}_{\text{TS}})^{-1} \mb{B}_{\text{TS}},
\end{equation}
where $\mb{B}_\text{TS} = (1/\sqrt{n_\text{TS}})(\mb{1}_{n_\text{TS}} \otimes \mb{B})$ is the time-spectral input matrix to restrict $u(t)$ to be sinusiodal.
The maximum energy amplification $G(\omega_f)$ for a given forcing frequency $\omega_f$ is then given by the largest singular value of $\mb{R}_{\text{TS}}$.
Note that the Mathieu oscillator is a non-autonomous system and the periodicity of the base flow arises from the external driver $\mathrm{cos}(\omega_0 t)$ and not a neutral Floquet mode.
This fully defines the phase of the system and the projection method described in Section~\ref{sec:transverse} is not needed.
Also, due to the lack of a Neutral Floquet mode the operator $\mb{L}_\text{TS}$ is never singular, even at resonant forcing frequencies.

To validate that the TSR operator accurately predicts the maximum amplification we compare the resolvent gain calculated to the gain computed from time accurate integration for $120$ forcing frequencies $\omega_f$ between 0.2$\omega_0$ and $3\omega_0$.
To compute the simulated gain, we extract the harmonic forcing $\mb{f}(t) = \mathrm{Re}(\hat{\mb{f}}e^{j\omega_f t})$
from the leading right singular vector $\mb{v}_1$ and use time-accurate Runge--Kutta (RK45) integration.
Once the forced system reaches a steady orbit, we extract the maximum response amplitude $||\boldsymbol{\eta}||_{L^2}$ and divide by the forcing amplitude to compute the simulated gain.

\begin{figure}[h!]
\centering
\includegraphics[width=0.75\textwidth]{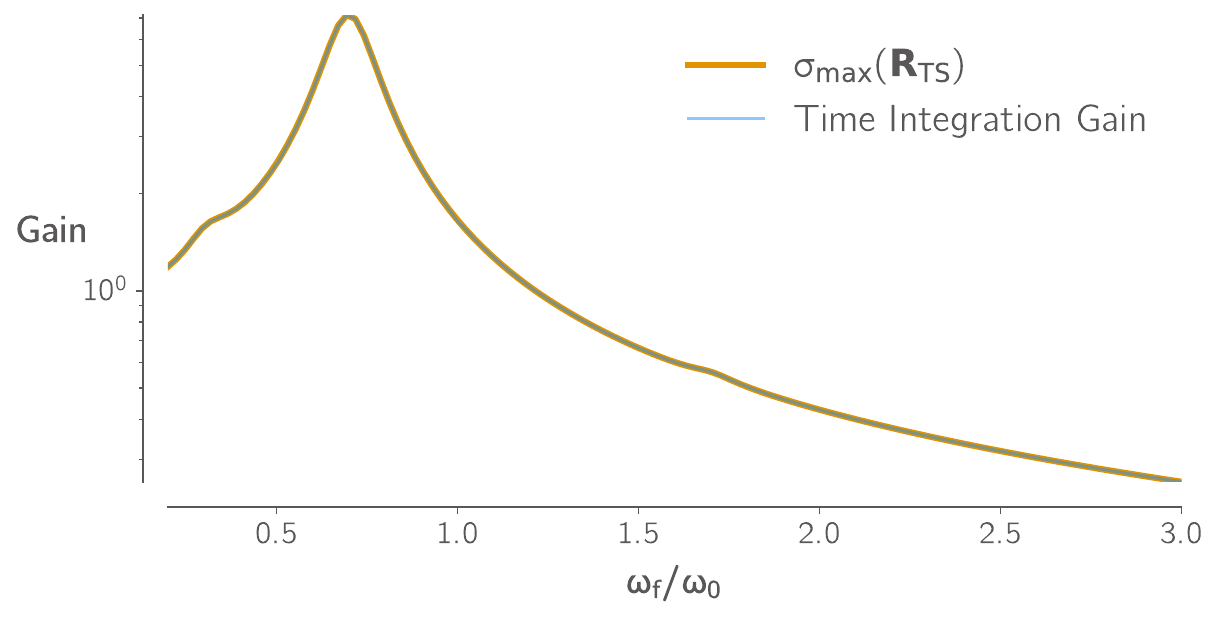}
\caption{The TSR gains perfectly matches the gains calculated from time accurate integration for all forcing frequencies.}
\label{fig:mathieu_resolvent}
\end{figure}
% [x] TODO HS-: plot 3 5 7 maybe? show the convergence
% MH-HS: for this problem, there is no visual difference in 3 or 5

\begin{figure}[h!]
\centering
\includegraphics[width=1\textwidth]{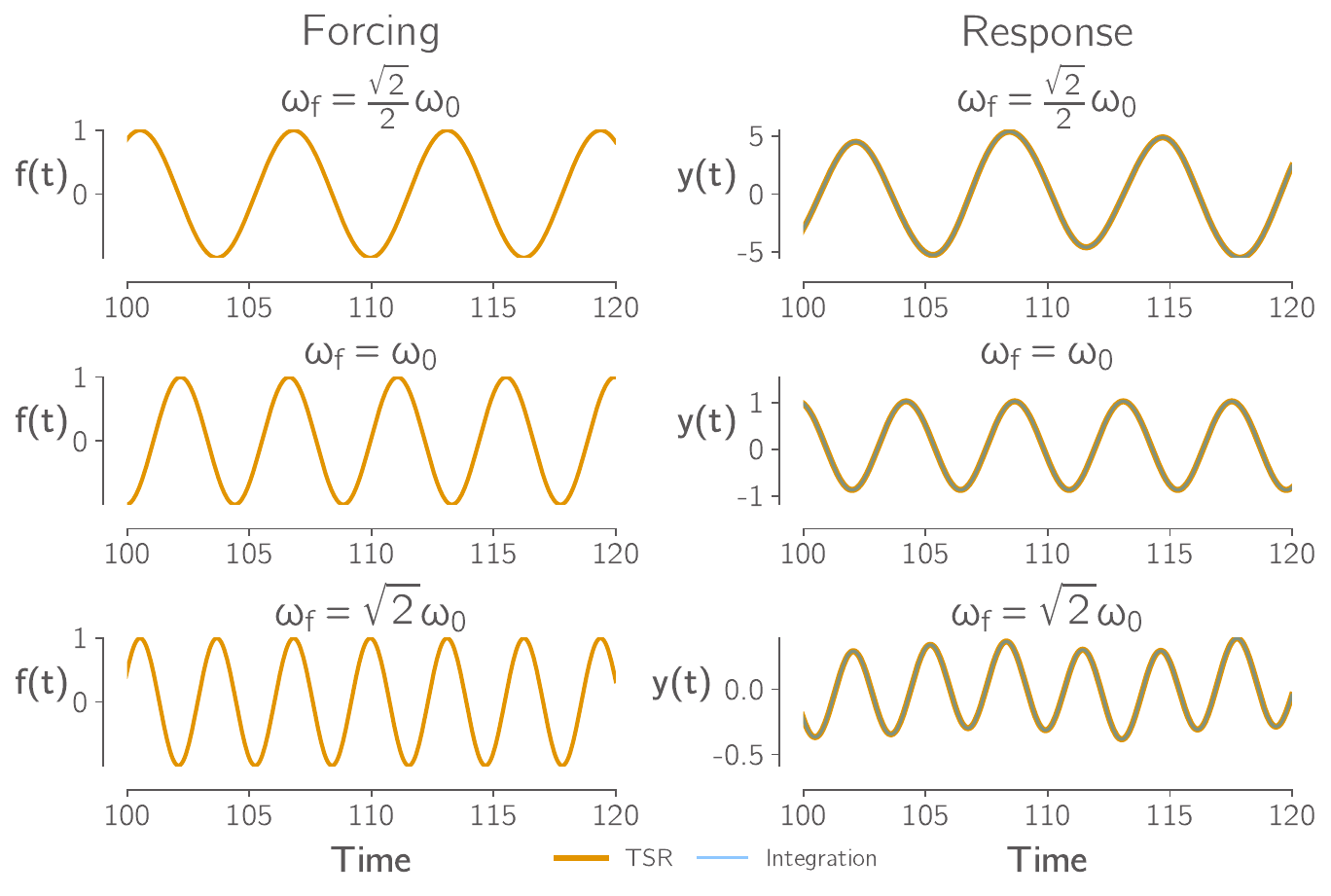}
\caption{The optimal forcing (left) and optimal response (right) from the TSR operator. The optimal response perfectly reconstructs the full quasi-periodic response when compared to the simulated result.}
\label{fig:mathieu_time_response}
\end{figure}

The comparison between the simulated and TSR gains are shown in Figure~\ref{fig:mathieu_resolvent}.
As can be seen, the full TSR operator accurately predicts the amplification for all frequency cases with only $n_{\text{TS}}=5$ collocation points.
The full TSR operator can not only accurately predict the resolvent gain via the maximum singular value, but also the response of the system through the left singular vector.
Shown in Figure~\ref{fig:mathieu_time_response} are the optimal forcing modes extracted from the right singular vector $\mb{v}_1$ as well as the predicted response of the system to this forcing extracted from the left singular vector $\mb{u}_1$ for a few selected forcing frequencies.
The optimal response mode is rotated to be aligned with the base flow as described in Section~\ref{sec:normalization} and scaled by $\sigma_\text{max}$ to have a physically accurate magnitude.
The simulated response of the linear system with optimal forcing from $\mb{v}_1$ is compared with the predicted response $\mb{u}_1$ in the figure.
As can be seen, the full TSR operator is able to perfectly reconstruct the quasi-periodic response with machine precision when compared to time accurate integration.

We also showcase the spectral convergence described in Section~\ref{sec:convergence} for the Mathieu Oscillator at the forcing point $\omega_f = \sqrt{2} \omega_0$.
This is done by computing the ground truth gain with a very high $n_\text{TS} = 501$ number of collocation points, then evaluating the convergence by inspecting the relative error between the gain computed using a coarse temporal grid and the ground truth.
This is shown in Figure~\ref{fig:spectral_convergence} as well as a trend line fitted to the spectral convergence curve.
As can be seen, the relative error decreases exponentially with increasing $n_\text{TS}$ points, numerically validating the spectral convergence property of the TSR.
% MH-HS: Highlighting that I added this paragraph and figure. 

\begin{figure}[!h]
\centering
\includegraphics[width=0.75\linewidth]{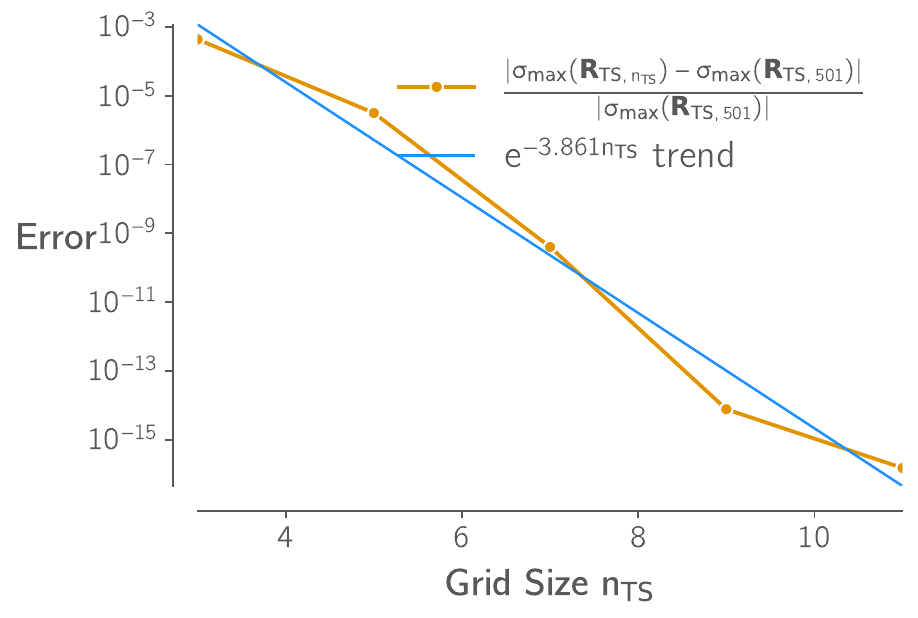}
\caption{Relative error of $\sigma_\text{max}(\mb{R}_\text{TS})$ with the ground truth gain computed with $501$ collocation points at $\omega_f = \sqrt{2}\omega_0$. The spectral convergence trend can be seen.}
\label{fig:spectral_convergence}
\end{figure}
% [x] TODO HS-: be explicit of what is plotted here, sigma? avoid dashed line
% MH-HS: removed dahed line and added formula

\subsection{Autonomous van der Pol Oscillator}
\label{sec:vdp}

Extending to an autonomous nonlinear dynamical system, we consider the van der Pol oscillator, a classical example of an autonomous system with nonlinear damping.
The governing equation is
\begin{equation}
\ddot{y} - \mu(1 - y^2)\dot{y} + y = u(t),
\label{eq:vdp}
\end{equation}
where the nonlinear damping is governed by $\mu = 1.0$.
In this example, we consider a quasi-periodic forcing $u(t)$ whose spectrum includes frequencies at $\omega_f + k\omega_0$.
Introducing the state $\mb{w} = [y, \dot{y}]^\top$, the system can be written in first-order form as
\begin{equation}
\frac{\d}{\d t}\begin{bmatrix} y \\ \dot{y} \end{bmatrix} =
\begin{bmatrix} \dot{y} \\ \mu(1 - y^2)\dot{y} - y \end{bmatrix} +
\begin{bmatrix} 0 \\ 1 \end{bmatrix} u(t).
\end{equation}
The base flow $\bar{\mb{w}}$ of this system is a LCO whose frequency depends on the choice of $\mu$.
In the case of $\mu = 1.0$, the base flow, shown in Figure~\ref{fig:vdp_base_flow}, has period $T_0 \approx 6.67$ and is computed using the time spectral method to ensure machine precision.

As can be seen in the right column of Figure~\ref{fig:vdp_base_flow}, different discretization levels $n_\text{TS}$ lead to different ranges of resolved eigenvalues $\lambda$ of the time spectral Jacobian $\mb{J}_\text{TS}$.
These eigenvalues correspond to the Floquet exponents of the linearized dynamics and the neutral $\mathrm{Re}(\lambda) = 0$ and stable $\mathrm{Re}(\lambda) \approx -1$ Floquet modes can clearly be seen in the figure.
As discussed in Section~\ref{sec:singularity_jacobian}, the time-spectral formulation generates a spectrum characterized by vertical ladders of eigenvalues where each physical Floquet exponent is repeated at integer shifts of the base frequency ($\lambda + jk\omega_0$).
The extent of this ladder corresponds to the temporal resolution, capturing frequency modes up to $|k| \leq (n_\text{TS}-1)/2$.
Consequently, increasing the number of collocation points from $n_\text{TS}=11$ to $n_\text{TS}=31$ extends the resolved spectrum, allowing for representation of higher frequency temporal dynamics for each Floquet mode.

\begin{figure}[h!]
\centering
\includegraphics[width=\textwidth]{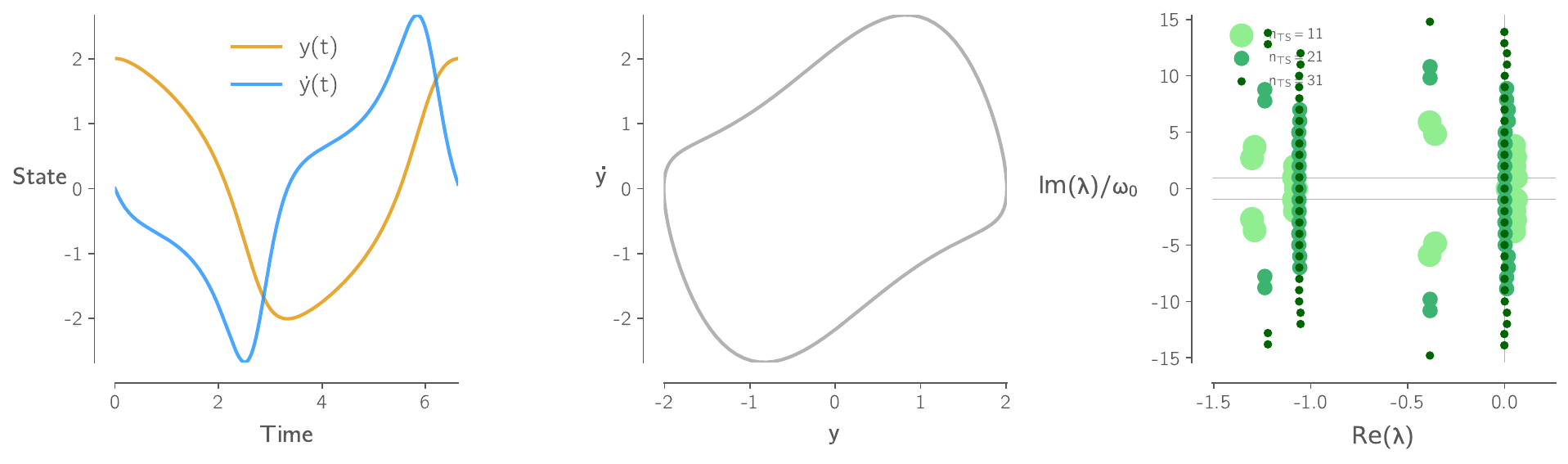}
\caption{Base flow of the van der Pol oscillator. The timeseries of the state variables (Left), LCO phase portrait (Middle), and resolved modes for different value of $n_\text{TS}$ (Right) are shown.}
\label{fig:vdp_base_flow}
\end{figure}

Linearizing the system around this base flow with the time-periodic Jacobian
\begin{equation}
\mb{J}(\bar{\mb{w}}) = 
\begin{bmatrix}
0 & 1\\
-2\mu y \dot{y} - 1 & \mu (1-y^2)
\end{bmatrix} ,
\end{equation}
gives the linearized system
\begin{equation}
\dot{\boldsymbol{\eta}}(t) = \mb{J}(\bar{\mb{w}})\boldsymbol{\eta}(t) + \mb{B}u(t) ,
\end{equation}
where $\mb{B}$ = $\begin{bmatrix}0 & 1\end{bmatrix}^\top$ and $\boldsymbol{\eta}(t)$ is the perturbation from the baseflow.

Discretizing the base flow into $n_\text{TS} = 31$ equally spaced collocation points with the phase variable $\theta_j = \omega_0 t$ gives the discretized periodic Jacobian
\begin{equation}
\mb{J}(\theta_j) = 
\begin{bmatrix}
0 & 1\\
-2\mu y_j \dot{y}_j - 1 & \mu (1 - y_j^2)
\end{bmatrix} ,
\end{equation}
where $y_j$ and $\dot{y}_j$ is the base flow evaluated at each collocation point.

Following the procedure for autonomous systems in Section~\ref{sec:auto_sysms}, we apply the oblique projector $\mb{P}_\omega$ implicitly through vector-vector products as described in Appendix~\ref{app:efficient_resolvent} to compute the SVD of the transverse resolvent operator $\mb{R}_{\Sigma}$ where the dominant singular value is the transverse gain $G_\Sigma(\omega_f)$.
This approach ensures that the resolvent gain is bounded by projecting the operator $\mb{L}_\text{TS}$ to the subspace $\Sigma_\text{TS}$ in which it is bijective and invariant.
This formulation projects both the forcing and response orthogonal to the modulated adjoint neutral mode $\tilde{\mb{q}}_\omega$, ensuring that the physical response $\tilde{\boldsymbol{\eta}}_{\Sigma}$ does not have any phase drift component.
Once the transverse gain is obtained, the reconstructed gain is computed as described in Section~\ref{sec:reconstruction}.
This reconstructed gain allows us to accurately compare our resolvent gain to time-accurate integration.
To highlight the difference between the reconstructed gain and the unprojected TSR gain, we also compute the unprojected gain $G(\omega_f)$ from the operator in Eq.~\ref{eq:discrete_op}.

In this example, we allow the forcing $u(t)$ to be quasi-periodic containing components of both the specified forcing frequency and the base frequency $\omega_f + k\omega_0$.
To do this, we construct the time-spectral input matrix in the standard way by discretizing it at each collocation point $\mb{B}_\text{TS} = \mathrm{blkdiag}(\mb{B},\dots,\mb{B})$.
This formulation removes the need for the $1/\sqrt{n_\text{TS}}$ Parseval scaling used in the Mathieu example.
Note that the predicted resolvent gain using the quasi-periodic forcing approach will always be equal to or larger than the gain predicted using the purely harmonic forcing formulation as the harmonic forcing is a special case which lies in the spectrum of the quasi-periodic forcing ($k=0$).

Both the reconstructed and unprotected TSR gains are validated against time-accurate integration for 300 forcing frequencies $\omega_f \in \left[0.2\omega_0, 2.3\omega_0\right]$ as shown in Figure~\ref{fig:vdp_gains}.
For the unprojected gain $G(\omega_f)$, the optimal forcing vector $\tilde{\mb{f}}$ is extracted from the SVD of the full TSR operator and used to drive the linearized system.
As expected, this gain is the largest as the full forcing is not orthogonal to the adjoint mode and excites phase drift in the system.
At resonance (integer multiples of the base frequency), this manifests as a singularity in the linear operator $\mb{L}_\text{TS}$ corresponding to secular growth in the integrated system.
The reconstructed gain includes both the particular response $\tilde{\boldsymbol{\eta}}_{\Sigma}$ predicted by the transverse operator as well as the homogeneous neutral mode component $c\tilde{\mb{p}}_0$.
The reconstructed gain perfectly matches the gain computed from time-accurate integration initialized from the base flow.
In both cases, the TSR predictions show excellent agreement with the time-accurate integration.

\begin{figure}[h!]
\centering
\includegraphics[width=\textwidth]{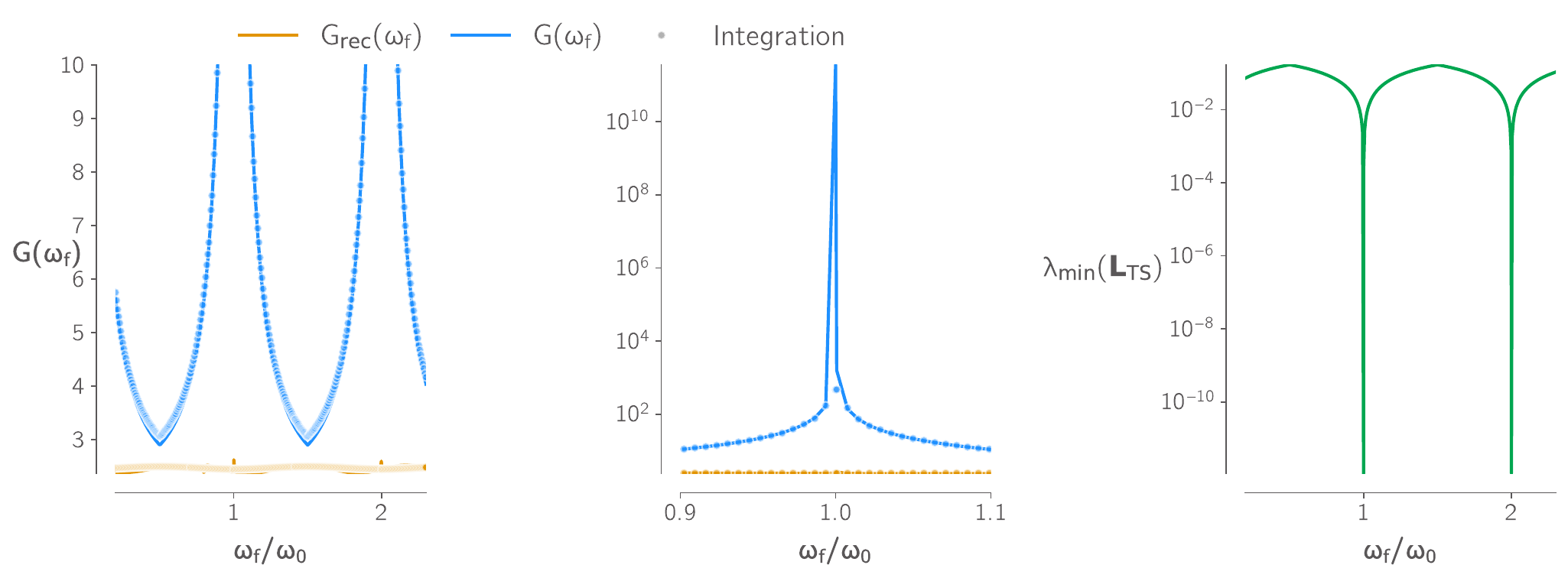}
\caption{Resolvent gain (Left), zoom near resonance (Middle), and minimum singular value of $\mb{L}_\text{TS}$ (Right) for the van der Pol oscillator.}
\label{fig:vdp_gains}
\end{figure}

The time-domain reconstruction of the optimal response modes is compared against time-accurate integration in Figure~\ref{fig:vdp_modes} for three selected forcing frequencies.
The middle column shows the transverse response $\tilde{\boldsymbol{\eta}}_{\Sigma}$ predicted by the left singular vector of the transverse resolvent operator $\mb{R}_{\Sigma}$.
To validate this against the integration of the linearized dynamics (which inherently includes phase drift from the initial condition), we isolate the transverse component of the integrated response by subtracting the scaled time derivative of the base flow.
The excellent agreement confirms that the transverse operator correctly identifies the shape deformation of the limit cycle.
The right column shows the full reconstructed response which adds the constant homogeneous phase shift $c\tilde{\mb{p}}_0$ required to satisfy the initial condition $\boldsymbol{\eta}(0) = \mb{0}$.
It can be seen that even in the resonant case, the response remains bounded, showing that the forcing $\tilde{\mb{f}}_\Sigma$ does not excite unbounded phase drift in the system.

In order to verify the physical accuracy of the transverse resolvent formulation, we perform a stroboscopic analysis of the linearized response as shown in Figure~\ref{fig:vdp_phase}.
The horizontal axis represents the instantaneous phase drift $c(t_k)$ as described in Section~\ref{sec:physical_response}, which is extracted at each base-flow period $t_k = k T_0$.
This instantaneous phase drift is computed as $c(t_k) = \langle \tilde{\mathbf{q}}_0(t_k), \boldsymbol{\eta}(t_k) \rangle / \langle \tilde{\mathbf{q}}_0(t_k), \tilde{\mathbf{p}}_0(t_k) \rangle$.
The vertical axis represents the transverse response magnitude $\|\mathbf{v}(t_k)\|$, where $\mathbf{v} = \boldsymbol{\eta} - c\tilde{\mathbf{p}}_0$ isolates the shape deformation of the limit cycle from the phase shift.
To generate these results, we extract the optimal forcing envelopes for both the projected $\tilde{\mathbf{f}}_{\Sigma}$ and unprojected $\tilde{\mathbf{f}}$ cases directly from the TSR operators.
The results clearly demonstrate that the projected forcing does not excite any phase drift, resulting in a response that remains a fixed point in the phase portrait.
In contrast, the unprojected forcing excites the neutral Floquet mode, causing the response to exhibit linear secular growth in the phase coordinate $c(t_k)$ over successive periods.
This numerical validation confirms that the transverse TSR analysis successfully isolates the stable shape deformation of autonomous systems by filtering out the physically unbounded phase sensitivity.

\begin{figure}
    \centering
    \includegraphics[width=1\linewidth]{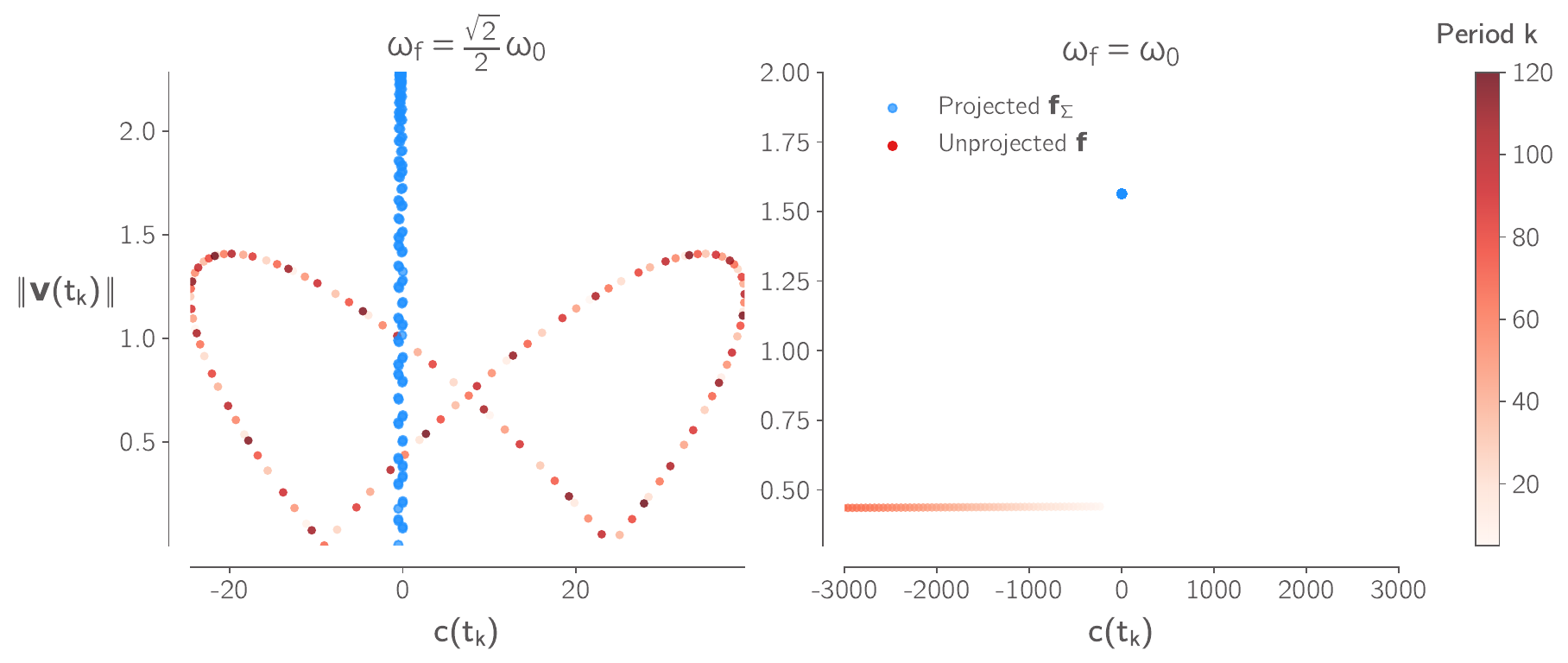}
    \caption{The projected forcing $\mb{f}_\Sigma$ excites near-zero phase drift in the system for both frequency cases.}
    \label{fig:vdp_phase}
\end{figure}

\begin{figure}[h!]
\centering
\includegraphics[width=\textwidth]{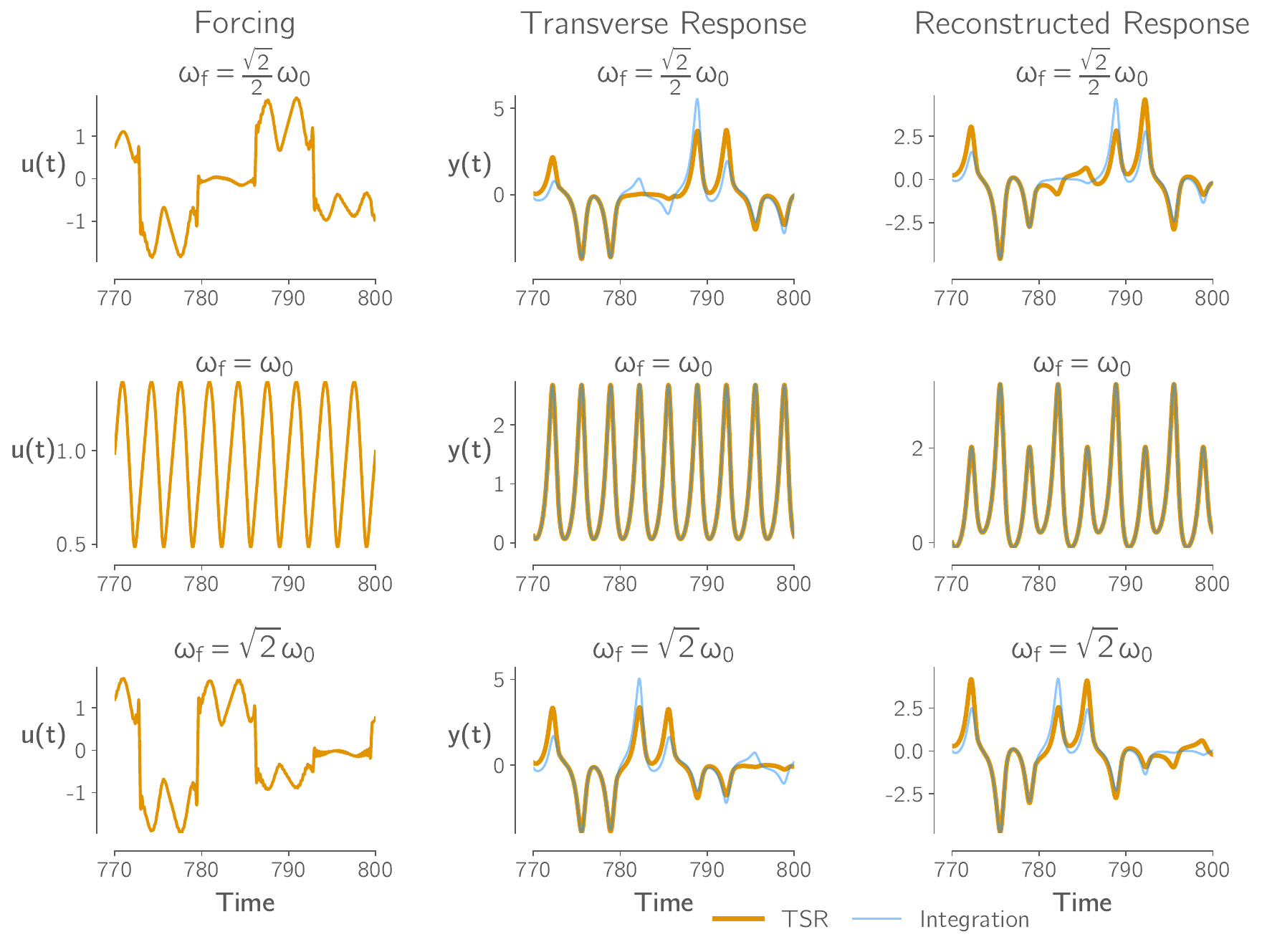}
\caption{TSR Timeseries response of the van der Pol oscillator compared with ODE integration. The forcing (left) is orthogonal to the adjoint mode and does not excite any phase drift in the full response (right). The transverse response (middle) does not contain the homogeneous phase component.}
\label{fig:vdp_modes}
\end{figure}

\subsection{Complex Ginzburg-Landau Equation}
\label{sec:cgle_validation}

The complex Ginzburg--Landau (CGL) equation provides a high dimensional example governed by a PDE.
A detailed overview of the CGL equation and its properties are given by~\citet{Aranson2002} and one form of the governing equation is
\begin{equation}
\partial_t {A} = {A} + (1 + j\alpha)\Delta {A} - (1 + j\beta)|{A}|^2{A} ,
\label{eq:cgl_complex}
\end{equation}
where $A(x,t) \in \mathbb{C}$ is the complex state, $\alpha = 0.1$ is the dispersion parameter and $\beta = 0.2$ is the nonlinear frequency parameter.
Periodic boundary conditions are used with domain $x \in [0, L]$ with $L=20$ discretized into $n_{\text{node}}=50$ nodes using the finite difference method.

The complex state can be decomposed into its real and imaginary parts, $A = w_1 + jw_2$, to form a purely real state space system of form Eq.~\ref{eq:nonlinear_system}
\begin{equation}
\begin{bmatrix}
\partial_t{w_1}\\
\partial_t{w_2}
\end{bmatrix}
=
\begin{bmatrix}
1 + \Delta & -\alpha\Delta\\
\alpha\Delta & 1 + \Delta
\end{bmatrix}
\begin{bmatrix}
{w_1}\\
{w_2}
\end{bmatrix}
- ({w}_1^2 + {w}_2^2)
\begin{bmatrix}
1 & -\beta\\
\beta & 1
\end{bmatrix}
\begin{bmatrix}
{w}_1\\
{w}_2
\end{bmatrix} ,
\label{eq:cgl_real}
\end{equation}
where $\Delta$ is the Laplacian operator.

We consider a plane wave (PW) base flow given by
\begin{equation}
A_{\text{PW}}(x,t) = \sqrt{1-k^2}\exp(j(kx-\omega_0 t)) ,
\label{eq:cgl_planewave}
\end{equation}
with temporal frequency $\omega_0 = \beta({1-k^2}) + \alpha k^2$.
%This can be verified by substitution: the Laplacian gives $\Delta A_{\text{PW}} = -k^2 A_{\text{PW}}$ and the nonlinear term gives $|A_{\text{PW}}|^2 A_{\text{PW}} = (1-k^2)A_{\text{PW}}$, yielding $\partial_t A_{\text{PW}} = -j\omega_0 A_{\text{PW}}$.
Using the wavenumber $k=2\pi/L$, this plane wave solution defines a $T_0 = 2\pi / \omega_0$ periodic base flow $\bar{\mb{w}}(t)$ in the sense of Eq.~\ref{eq:periodic_resolvent}.
Note that unlike a limit cycle, the plane wave is not an isolated periodic orbit; it belongs to a continuous family parameterized by $k$ and admits spatial translation symmetry.
The contour plot shown in Figure~\ref{fig:cgl_baseflow} shows the base flow for the chosen parameters.
\begin{figure}[!h]
\centering
\includegraphics[width=1\linewidth]{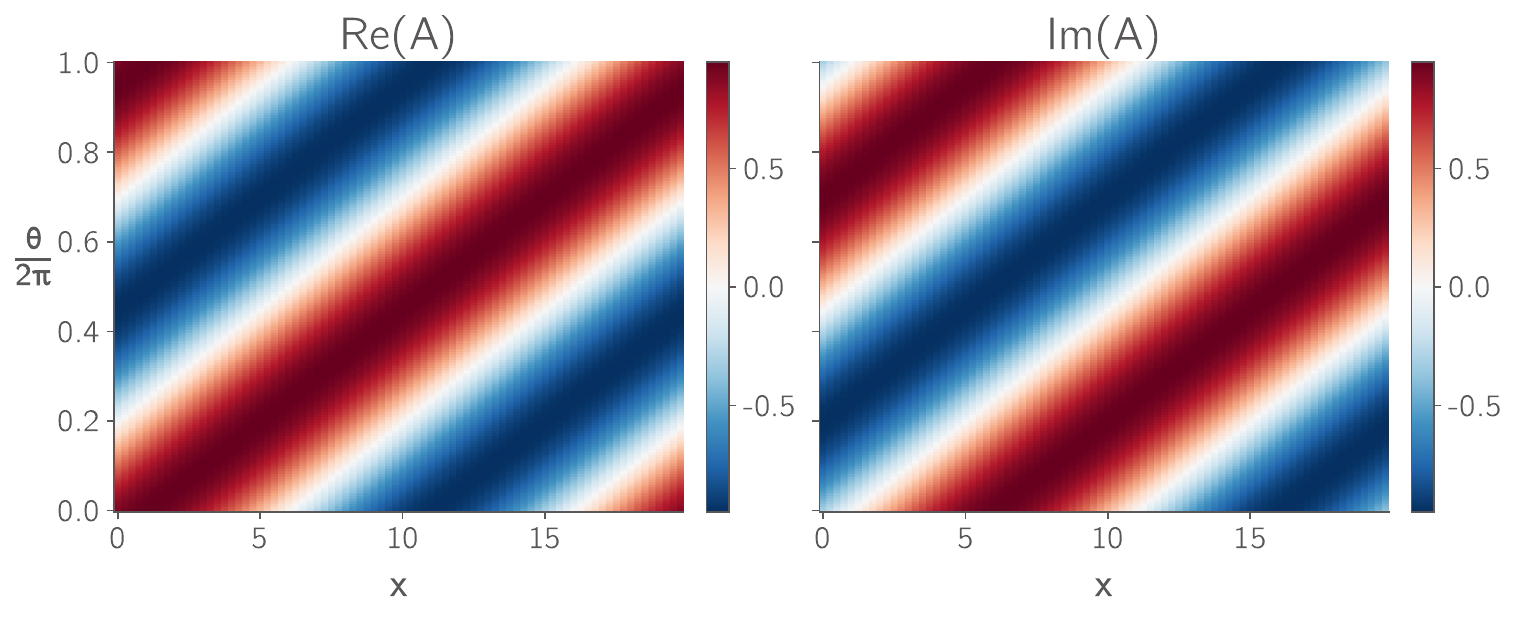}
\caption{The plane wave solution defines the periodic base flow for the CGL equation.}
\label{fig:cgl_baseflow}
\end{figure}

Linearizing Eq.~\ref{eq:cgl_real} around this base flow yields the time-periodic Jacobian $\mb{J}(t) = \partial \mb{r}/\partial \mb{w}|_{\bar{\mb{w}}(t)}$, which can be evaluated analytically at any point along the orbit.
This forms a linear time-periodic (LTP) system of the form Eq.~\ref{eq:periodic_resolvent} with spatially uniform forcing applied to all states through $\mb{B} = \mb{I}$.
The discrete Jacobian $\mb{J}(\theta_j)$ is then evaluated at the time-spectral collocation points with $n_{\text{TS}} = 21$.
% [x] TODO HS-: avoid even number of time instance.
% MH-HS: fixed.

The sparsity structure of this discrete Jacobian can be exploited to efficiently compute the adjoint neutral mode and apply the transverse TSR operator as described in Appendix~\ref{app:efficient_resolvent}.
Due to the autonomous nature of the CGL system, we consider the operator $\mb{R}_{\Sigma}$ to map from the quasi-periodic forcing $\tilde{\mb{f}}_{\Sigma}$ to the steady state transverse response $\tilde{\boldsymbol{\eta}}_{\Sigma}$.
To allow the forcing to be quasi-periodic we again form the time-spectral input matrix as in Eq~\ref{eq:input_mat}.
To efficiently compute the SVD of the resolvent operator we implicitly apply the oblique projection through vector-vector products (Eq.~\ref{eq:implicit_projection}).
The forward and adjoint linear systems used to compute the action of $\mb{R}_{\Sigma}$ on randomized vectors are solved using the FD-preconditioned GMRES detailed in~\ref{sec:fd_preconditioner}.
This approach efficiently solves the SVD of $\mb{R}_{\Sigma}$ without ever forming the dense operator $\mb{P}_\omega\mb{L}_\text{TS}^{-1}\mb{P}_\omega$.

\begin{figure}[!h]
\centering
\includegraphics[width=0.75\linewidth]{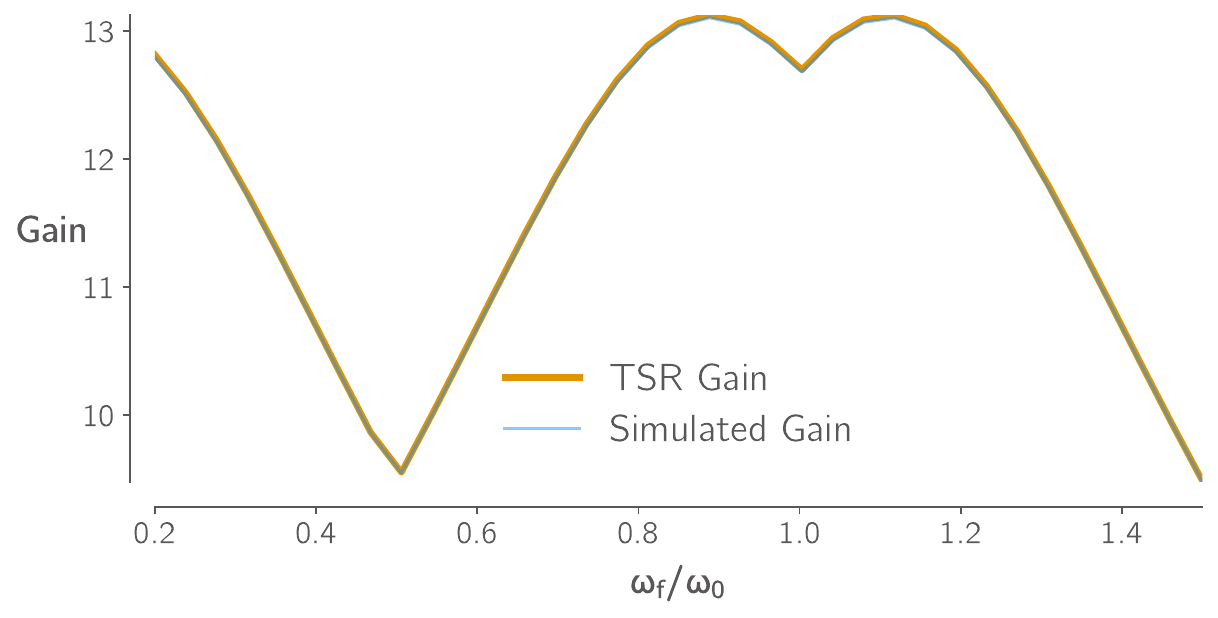}
\caption{The gains predicted from the TSR operator perfectly match the gains from the PDE simulation.}
\label{fig:cgl_gains}
\end{figure}

The reconstructed resolvent gain $G_\text{rec}(\omega_f)$ is compared to time-accurate integration of the linearized system at 35 different forcing frequencies between $0.2\omega_0$ and $1.5\omega_0$ in Figure~\ref{fig:cgl_gains}.
For each forcing frequency, the optimal transverse forcing $\tilde{\mb{f}}_{\Sigma}$ is extracted from the SVD of the transverse resolvent operator and used to drive the linearized dynamics via time-accurate integration.
The maximum energy amplification extracted from the simulated response can be seen to strongly agree with the reconstructed resolvent gain which considers the transverse response  $\tilde{\boldsymbol{\eta}}_{\Sigma}$ and the homogeneous phase component $c\tilde{\mb{p}}_0$.

\begin{figure}[!h]
\centering
\includegraphics[width=1\linewidth]{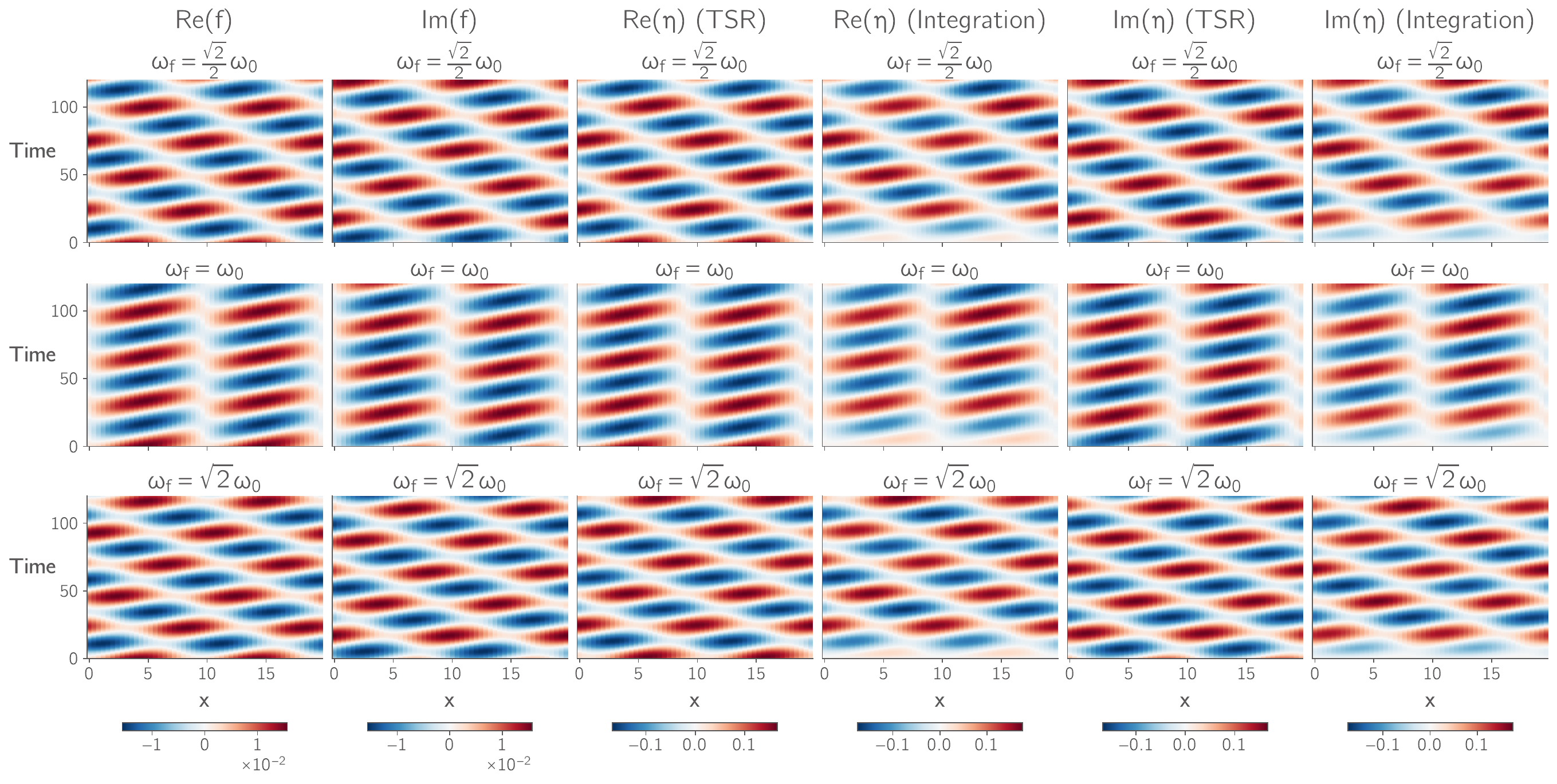}
\caption{When the PDE system is integrated with the optimal forcing vector (left) the simulated response matches that of the optimal response mode from the TSR operator.}
\label{fig:cgl_modes}
\end{figure}
% [x] TODO HS-: 1e-4 -> 10^{-4} 1e-6-> 10^{-6}
% [x] TODO HS-: make sure clarify that the response mode is quasi periodic.
% MH-HS: added some text about this. Also shown in the new figure
% [x] TODO HS-: how about in this figure, replace the forcing using the contour as well. Add another figure, monitor several points in the domain, show: line plots of the forcing mode, ODE response, ODE response - reference, TSA response mode.
% MH-HS: Added.

\begin{figure}[!h]
\centering
\includegraphics[width=1\linewidth]{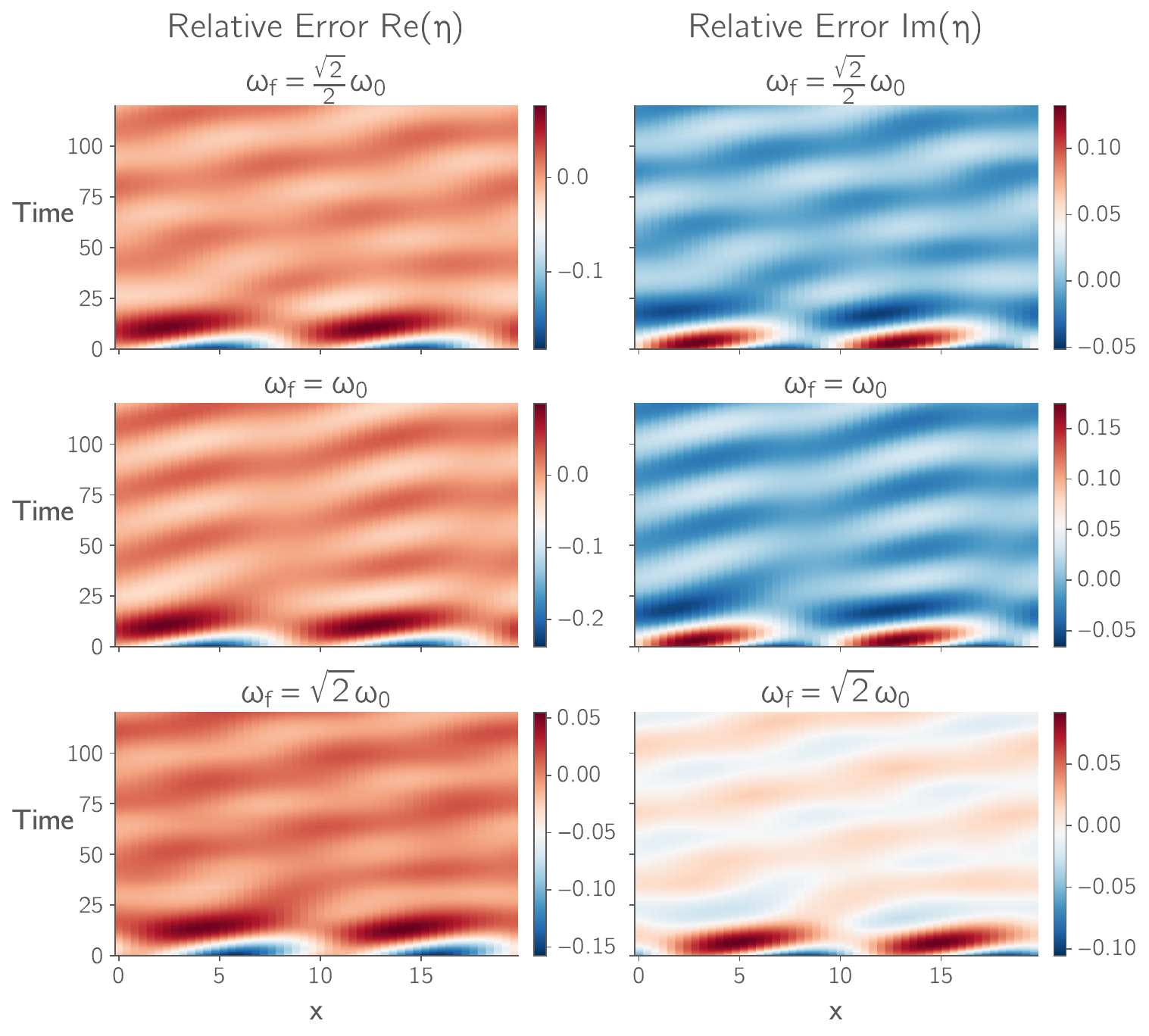}
\caption{The relative error between the PDE simulation and the TSR responses is minimal.}
\label{fig:cgl_error}
\end{figure}
% [x] TODO HS-: the large relative error comes from the small divisor? Try use abs(err)/(abs(value) + 1e-6) etc to avoid this.
% MH-HS: Yes, fixed.

The reconstructed spatiotemporal response $\tilde{\boldsymbol{\eta}}_{\Sigma} + c\tilde{\mb{p}}_0$ is also compared with full non-linear PDE simulation for selected forcing frequencies in Figure~\ref{fig:cgl_modes}.
The left two columns show the real and imaginary forcing component, which consists of the optimal quasi-periodic and spatially varying forcing.
The right columns show the reconstructed response predicted by the transverse TSR operator and the non-linear PDE simulation, demonstrating excellent agreement.
This is quantified in Figure~\ref{fig:cgl_error}, which shows that the relative error between the TSR prediction and the simulated PDE response is minimal.

\begin{figure}
\centering
\includegraphics[width=1\linewidth]{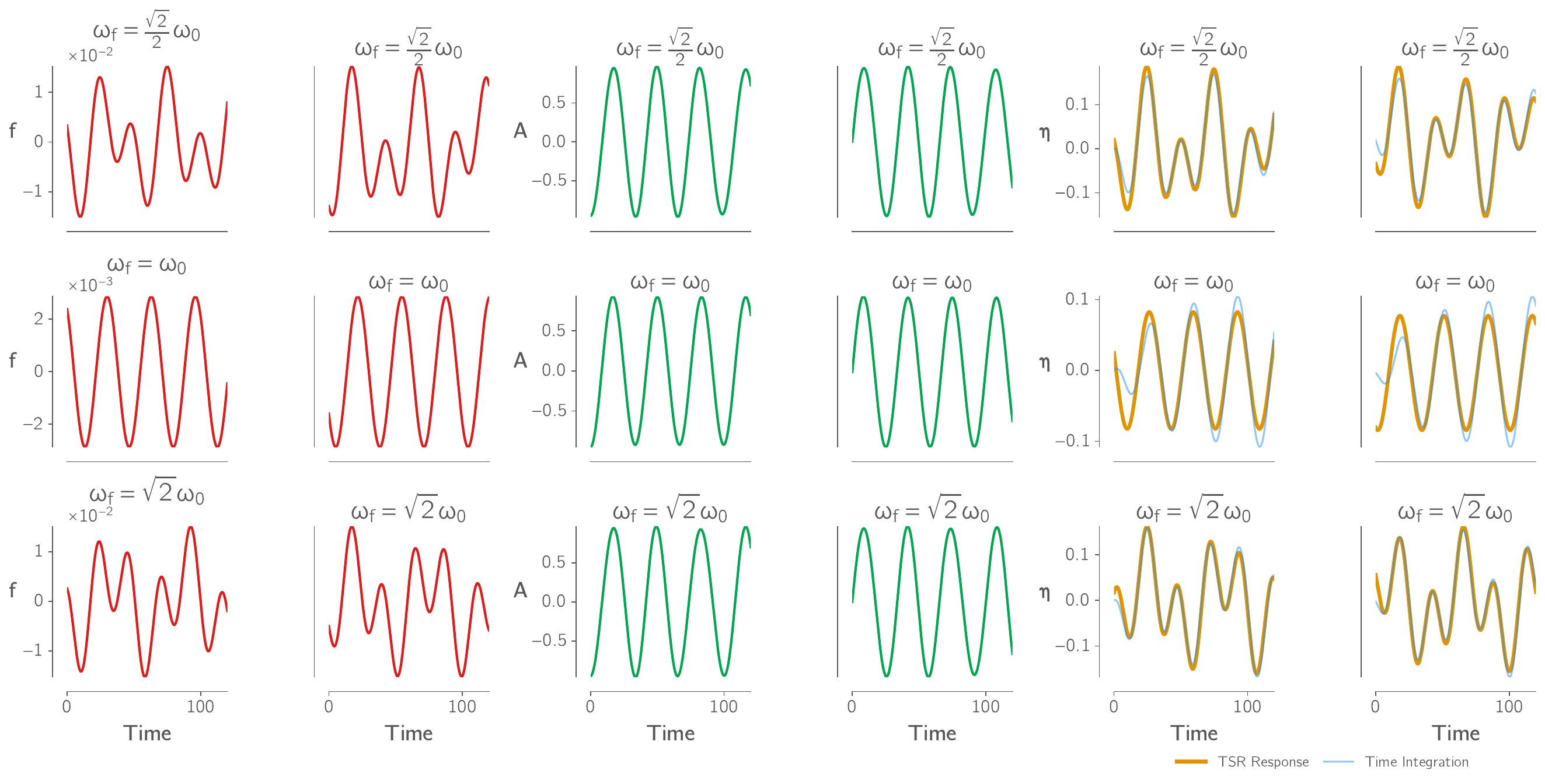}
\caption{Temporal forcing and responses of the full nonlinear system $A(L/2,t)$ as well as the perturbation from the base flow $\eta(L/2,t)$.}
\label{fig:cgl_lineplots}
\end{figure}

Note that at all frequencies, due to the quasi-periodic nature of the forcing spectrum, the response is also quasi-periodic.
This further demonstrated in Figure~\ref{fig:cgl_lineplots}, which shows the optimal response, full nonlinear response, and perturbation from the base flow for the spatial point $x=L/2$.
In the case of a non-commensurate forcing frequency, the response is quasi-periodic and consists of a combination of the base frequency $\omega_0$ and the forcing frequency $\omega_f$.

\section{Conclusion}

This paper presented the time-spectral resolvent (TSR) operator, a computationally efficient framework for harmonic resolvent analysis of both autonomous and non-autonomous systems with time-periodic base flows.
By leveraging Fourier collocation and the time-spectral method, the TSR operator maps forcing and response envelopes defined on a discrete temporal grid rather than truncated Fourier coefficients, avoiding the explicit construction of large block-Toeplitz matrices required by frequency-domain formulations.

The key advantages of the TSR formulation are threefold.
First, the method achieves spectral convergence with respect to the number of collocation points, enabling accurate gain predictions with relatively few temporal samples.
Second, the Jacobian can be evaluated directly at collocation points without requiring explicit computation of Fourier coefficients, simplifying implementation and reducing computational overhead.
Third, the formulation directly maps the discrete forcing to the discrete quasi-periodic response, avoiding the complication of reconstructing the response from the Fourier coefficients.

Validation against time-accurate integration demonstrated the accuracy of the TSR operator for three canonical problems.
For the parametrically forced Mathieu oscillator, the TSR operator accurately predicted the resolvent gain across a range of forcing frequencies using only $n_{\text{TS}} = 5$ collocation points.
Spectral convergence was also demonstrated using the Mathieu oscillator, showing that the relative error decreases exponentially with the number of collocation points used.
The TSR framework was also validated using the autonomous van der Pol oscillator, where it was shown that the TSR framework, with special handling, can accurately predict the response and resolvent gain for autonomous systems which exhibit a neutral Floquet mode.
For the complex Ginzburg-Landau equation with a plane wave base flow, the transverse TSR operator matched both the gain predictions and optimal response modes from full nonlinear PDE simulations with small magnitude forcing using $n_{\text{TS}} = 21$ collocation points.
In all cases, the optimal forcing extracted from the TSR operator produced responses that agreed with direct numerical integration.

Future work will extend the TSR framework to high-dimensional computational fluid dynamics applications, where the efficiency gains from avoiding block-Toeplitz construction become more significant.
The combination of the TSR operator with matrix-free iterative solvers and adjoint-based optimization presents a promising direction for resolvent-based flow control of unsteady aerodynamic systems.

\bibliographystyle{plainnat}
\bibliography{references}

@book{HornJohnson2013,
  author    = {Horn, Roger A. and Johnson, Charles R.},
  title     = {Matrix Analysis},
  edition   = {2nd},
  publisher = {Cambridge University Press},
  year      = {2013},
  isbn      = {978-0-521-54823-6}
}

@article{Saad1986,
  title = {GMRES: A Generalized Minimal Residual Algorithm for Solving Nonsymmetric Linear Systems},
  volume = {7},
  ISSN = {2168-3417},
  url = {http://dx.doi.org/10.1137/0907058},
  DOI = {10.1137/0907058},
  number = {3},
  journal = {SIAM Journal on Scientific and Statistical Computing},
  publisher = {Society for Industrial & Applied Mathematics (SIAM)},
  author = {Saad,  Youcef and Schultz,  Martin H.},
  year = {1986},
  month = jul,
  pages = {856–869}
}

@misc{bongarzone2025,
  doi = {10.48550/ARXIV.2503.08401},
  url = {https://arxiv.org/abs/2503.08401},
  author = {Bongarzone,  Alessandro and Content,  Cédric and Sipp,  Denis and Leclercq,  Colin},
  title = {Adjoint-free method for mean resolvent analysis of periodic flows},
  publisher = {arXiv},
  year = {2025},
  copyright = {Creative Commons Attribution 4.0 International}
}

@book{bottcher2012spectral,
  author    = {B{\"o}ttcher, Albrecht and Silbermann, Bernd},
  title     = {Introduction to Large Truncated {T}oeplitz Matrices},
  publisher = {Springer},
  year      = {1999},
  series    = {Universitext},
  doi       = {10.1007/978-1-4612-1426-7}
}

@article{Aranson2002,
  title = {The world of the complex Ginzburg-Landau equation},
  volume = {74},
  ISSN = {1539-0756},
  url = {http://dx.doi.org/10.1103/RevModPhys.74.99},
  DOI = {10.1103/revmodphys.74.99},
  number = {1},
  journal = {Reviews of Modern Physics},
  publisher = {American Physical Society (APS)},
  author = {Aranson,  Igor S. and Kramer,  Lorenz},
  year = {2002},
  month = feb,
  pages = {99–143}
}

@article{Koopman1931,
  title = {Hamiltonian Systems and Transformation in Hilbert Space},
  volume = {17},
  ISSN = {1091-6490},
  url = {http://dx.doi.org/10.1073/pnas.17.5.315},
  DOI = {10.1073/pnas.17.5.315},
  number = {5},
  journal = {Proceedings of the National Academy of Sciences},
  publisher = {Proceedings of the National Academy of Sciences},
  author = {Koopman,  B. O.},
  year = {1931},
  month = may,
  pages = {315–318}
}

@article{Eiter2022,
  title = {On the Oseen-Type Resolvent Problem Associated with Time-Periodic Flow past a Rotating Body},
  volume = {54},
  ISSN = {1095-7154},
  url = {http://dx.doi.org/10.1137/21M1456728},
  DOI = {10.1137/21m1456728},
  number = {4},
  journal = {SIAM Journal on Mathematical Analysis},
  publisher = {Society for Industrial & Applied Mathematics (SIAM)},
  author = {Eiter,  Thomas},
  year = {2022},
  month = aug,
  pages = {4987–5012}
}

@article{Susuki2021,
  title = {Koopman Resolvent: A Laplace-Domain Analysis of Nonlinear Autonomous Dynamical Systems},
  volume = {20},
  ISSN = {1536-0040},
  url = {http://dx.doi.org/10.1137/20M1335935},
  DOI = {10.1137/20m1335935},
  number = {4},
  journal = {SIAM Journal on Applied Dynamical Systems},
  publisher = {Society for Industrial & Applied Mathematics (SIAM)},
  author = {Susuki,  Yoshihiko and Mauroy,  Alexandre and Mezić,  Igor},
  year = {2021},
  month = jan,
  pages = {2013–2036}
}

@article{Padovan2020,
  title = {Analysis of amplification mechanisms and cross-frequency interactions in nonlinear flows via the harmonic resolvent},
  volume = {900},
  ISSN = {1469-7645},
  url = {http://dx.doi.org/10.1017/jfm.2020.497},
  DOI = {10.1017/jfm.2020.497},
  journal = {Journal of Fluid Mechanics},
  publisher = {Cambridge University Press (CUP)},
  author = {Padovan,  Alberto and Otto,  Samuel E. and Rowley,  Clarence W.},
  year = {2020},
  month = aug 
}

@article{Padovan2022,
  title = {Analysis of the dynamics of subharmonic flow structures via the harmonic resolvent: Application to vortex pairing in an axisymmetric jet},
  volume = {7},
  ISSN = {2469-990X},
  url = {http://dx.doi.org/10.1103/PhysRevFluids.7.073903},
  DOI = {10.1103/physrevfluids.7.073903},
  number = {7},
  journal = {Physical Review Fluids},
  publisher = {American Physical Society (APS)},
  author = {Padovan,  Alberto and Rowley,  Clarence W.},
  year = {2022},
  month = jul 
}

@article{Farghadan2024,
  title = {Efficient harmonic resolvent analysis via time stepping},
  volume = {38},
  ISSN = {1432-2250},
  url = {http://dx.doi.org/10.1007/s00162-024-00694-1},
  DOI = {10.1007/s00162-024-00694-1},
  number = {3},
  journal = {Theoretical and Computational Fluid Dynamics},
  publisher = {Springer Science and Business Media LLC},
  author = {Farghadan,  Ali and Jung,  Junoh and Bhagwat,  Rutvij and Towne,  Aaron},
  year = {2024},
  month = may,
  pages = {331–353}
}

@article{hwang2010linear,
  title={Linear non-normal energy amplification of harmonic and stochastic forcing in turbulent channel flow},
  author={Hwang, Yongyun and Cossu, Carlo},
  journal={Journal of Fluid Mechanics},
  volume={664},
  pages={51--73},
  year={2010},
  publisher={Cambridge University Press},
  doi={10.1017/S0022112010003629}
}

@article{Leclercq2023,
  title = {Mean resolvent operator of a statistically steady flow},
  volume = {968},
  ISSN = {1469-7645},
  url = {http://dx.doi.org/10.1017/jfm.2023.530},
  DOI = {10.1017/jfm.2023.530},
  journal = {Journal of Fluid Mechanics},
  publisher = {Cambridge University Press (CUP)},
  author = {Leclercq,  Colin and Sipp,  Denis},
  year = {2023},
  month = jul 
}

@article{taira2017modal,
  title={Modal analysis of fluid flows: An overview},
  author={Taira, Kunihiko and Hemati, Maziar S. and Brunton, Steven L. and Sun, Yiyang and Duraisamy, Karthik and Bagheri, Shervin and Dawson, Scott T. M. and Yeh, Chi-An},
  journal={AIAA Journal},
  volume={55},
  number={12},
  pages={4013--4041},
  year={2017},
  publisher={American Institute of Aeronautics and Astronautics},
  doi={10.2514/1.J056060}
}

@book{Trefethen2000,
  title = {Spectral Methods in MATLAB},
  ISBN = {9780898719598},
  url = {http://dx.doi.org/10.1137/1.9780898719598},
  DOI = {10.1137/1.9780898719598},
  publisher = {Society for Industrial and Applied Mathematics},
  author = {Trefethen,  Lloyd N.},
  year = {2000},
  month = jan 
}

@article{Bamieh2001,
  title = {Energy amplification in channel flows with stochastic excitation},
  volume = {13},
  ISSN = {1089-7666},
  url = {http://dx.doi.org/10.1063/1.1398044},
  DOI = {10.1063/1.1398044},
  number = {11},
  journal = {Physics of Fluids},
  publisher = {AIP Publishing},
  author = {Bamieh,  Bassam and Dahleh,  Mohammed},
  year = {2001},
  month = nov,
  pages = {3258–3269}
}

@book{boyd2013chebyshev,
  title={Chebyshev and Fourier Spectral Methods: Second Revised Edition},
  author={Boyd, J.P.},
  isbn={9780486141923},
  series={Dover Books on Mathematics},
  url={https://books.google.com/books?id=b4TCAgAAQBAJ},
  year={2013},
  publisher={Dover Publications}
}

@article{Junge2021,
  title = {A New Harmonic Balance Approach Using Multidimensional Time},
  volume = {143},
  ISSN = {1528-8919},
  url = {http://dx.doi.org/10.1115/1.4049698},
  DOI = {10.1115/1.4049698},
  number = {8},
  journal = {Journal of Engineering for Gas Turbines and Power},
  publisher = {ASME International},
  author = {Junge,  Laura and Frey,  Christian and Ashcroft,  Graham and K\"{u}geler,  Edmund},
  year = {2021},
  month = mar 
}

@article{Martini2021,
  title = {Efficient computation of global resolvent modes},
  volume = {919},
  ISSN = {1469-7645},
  url = {http://dx.doi.org/10.1017/jfm.2021.364},
  DOI = {10.1017/jfm.2021.364},
  journal = {Journal of Fluid Mechanics},
  publisher = {Cambridge University Press (CUP)},
  author = {Martini,  Eduardo and Rodríguez,  Daniel and Towne,  Aaron and Cavalieri,  André V.G.},
  year = {2021},
  month = may 
}

@article{Sipp2012,
  title = {Characterization of noise amplifiers with global singular modes: the case of the leading-edge flat-plate boundary layer},
  volume = {27},
  ISSN = {1432-2250},
  url = {http://dx.doi.org/10.1007/s00162-012-0265-y},
  DOI = {10.1007/s00162-012-0265-y},
  number = {5},
  journal = {Theoretical and Computational Fluid Dynamics},
  publisher = {Springer Science and Business Media LLC},
  author = {Sipp,  Denis and Marquet,  Olivier},
  year = {2012},
  month = apr,
  pages = {617–635}
}

@article{Mavriplis2011,
  title = {Time Spectral Method for Periodic and Quasi-Periodic Unsteady Computations on Unstructured Meshes},
  volume = {6},
  ISSN = {1760-6101},
  url = {http://dx.doi.org/10.1051/mmnp/20116309},
  DOI = {10.1051/mmnp/20116309},
  number = {3},
  journal = {Mathematical Modelling of Natural Phenomena},
  publisher = {EDP Sciences},
  author = {Mavriplis,  D. J. and Yang,  Z.},
  year = {2011},
  pages = {213–236}
}

@article{Wirth1997,
  title = {An Extension of Spectral Methods to Quasi-Periodic and Multiscale Problems},
  volume = {132},
  ISSN = {0021-9991},
  url = {http://dx.doi.org/10.1006/jcph.1996.5628},
  DOI = {10.1006/jcph.1996.5628},
  number = {2},
  journal = {Journal of Computational Physics},
  publisher = {Elsevier BV},
  author = {Wirth,  A.},
  year = {1997},
  month = apr,
  pages = {285–290}
}

@inproceedings{Mundis2013,
  title = {Quasi-periodic Time Spectral Method for Aeroelastic Flutter Analysis},
  url = {http://dx.doi.org/10.2514/6.2013-638},
  DOI = {10.2514/6.2013-638},
  booktitle = {51st AIAA Aerospace Sciences Meeting including the New Horizons Forum and Aerospace Exposition},
  publisher = {American Institute of Aeronautics and Astronautics},
  author = {Mundis,  Nathan and Mavriplis,  Dimitri},
  year = {2013},
  month = jan 
}

@article{Herrmann2021,
  title = {Data-driven resolvent analysis},
  volume = {918},
  ISSN = {1469-7645},
  url = {http://dx.doi.org/10.1017/jfm.2021.337},
  DOI = {10.1017/jfm.2021.337},
  journal = {Journal of Fluid Mechanics},
  publisher = {Cambridge University Press (CUP)},
  author = {Herrmann,  Benjamin and Baddoo,  Peter J. and Semaan,  Richard and Brunton,  Steven L. and McKeon,  Beverley J.},
  year = {2021},
  month = may 
}

@inbook{Loiseau2018,
  title = {Time-Stepping and Krylov Methods for Large-Scale Instability Problems},
  ISBN = {9783319914947},
  ISSN = {1871-3033},
  url = {http://dx.doi.org/10.1007/978-3-319-91494-7_2},
  DOI = {10.1007/978-3-319-91494-7_2},
  booktitle = {Computational Modelling of Bifurcations and Instabilities in Fluid Dynamics},
  publisher = {Springer International Publishing},
  author = {Loiseau,  J.-Ch. and Bucci,  M. A. and Cherubini,  S. and Robinet,  J.-Ch.},
  year = {2018},
  month = jul,
  pages = {33–73}
}

@article{MONOKROUSOS2010,
  title = {Global three-dimensional optimal disturbances in the Blasius boundary-layer flow using time-steppers},
  volume = {650},
  ISSN = {1469-7645},
  url = {http://dx.doi.org/10.1017/S0022112009993703},
  DOI = {10.1017/s0022112009993703},
  journal = {Journal of Fluid Mechanics},
  publisher = {Cambridge University Press (CUP)},
  author = {Monokrousos, Antonios and {\AA}kervik, Espen and Brandt, Luca and Henningson, Dan S.},
  year = {2010},
  month = mar,
  pages = {181–214}
}

@article{Bagheri2009,
  title = {Matrix-Free Methods for the Stability and Control of Boundary Layers},
  volume = {47},
  ISSN = {1533-385X},
  url = {http://dx.doi.org/10.2514/1.41365},
  DOI = {10.2514/1.41365},
  number = {5},
  journal = {AIAA Journal},
  publisher = {American Institute of Aeronautics and Astronautics (AIAA)},
  author = {Bagheri,  Shervin and Åkervik,  Espen and Brandt,  Luca and Henningson,  Dan S.},
  year = {2009},
  month = may,
  pages = {1057–1068}
}

@book{Schmid2001,
  title = {Stability and Transition in Shear Flows},
  ISBN = {9781461301851},
  ISSN = {0066-5452},
  url = {http://dx.doi.org/10.1007/978-1-4613-0185-1},
  DOI = {10.1007/978-1-4613-0185-1},
  journal = {Applied Mathematical Sciences},
  publisher = {Springer New York},
  author = {Schmid,  Peter J. and Henningson,  Dan S.},
  year = {2001}
}

@article{Halko2011,
  title     = {{Finding Structure with Randomness: Probabilistic Algorithms for Constructing Approximate Matrix Decompositions}},
  volume    = {53},
  issn      = {1095-7200},
  doi       = {10.1137/090771806},
  number    = {2},
  journal   = {SIAM Review},
  publisher = {Society for Industrial & Applied Mathematics (SIAM)},
  author    = {Halko,  N. and Martinsson,  P. G. and Tropp,  J. A.},
  year      = {2011},
  month     = jan,
  pages     = {217–288}
}

@article{Kanchi2025,
  title = {Differentiable singular value decomposition (SVD)},
  volume = {237},
  ISSN = {0888-3270},
  url = {http://dx.doi.org/10.1016/j.ymssp.2025.112817},
  DOI = {10.1016/j.ymssp.2025.112817},
  journal = {Mechanical Systems and Signal Processing},
  publisher = {Elsevier BV},
  author = {Kanchi,  Rohit Sunil and He,  Sicheng},
  year = {2025},
  month = August,
  pages = {112817}
}

@article{McKeon2010,
  author  = {Beverley J. McKeon and Adrian S. Sharma},
  title   = {A critical-layer framework for turbulent pipe flow},
  journal = {Journal of Fluid Mechanics},
  volume  = {658},
  pages   = {336--382},
  year    = {2010},
  doi     = {10.1017/S002211201000176X}
}

@article{Kojima2019,
  title     = {Resolvent analysis on the origin of two-dimensional transonic buffet},
  volume    = {885},
  issn      = {1469-7645},
  doi       = {10.1017/jfm.2019.992},
  journal   = {Journal of Fluid Mechanics},
  publisher = {Cambridge University Press (CUP)},
  author    = {Kojima,  Yoimi and Yeh,  Chi-An and Taira,  Kunihiko and Kameda,  Masaharu},
  year      = {2019},
  month     = dec
}

@article{JOVANOVI2005,
  title     = {Componentwise energy amplification in channel flows},
  volume    = {534},
  issn      = {1469-7645},
  doi       = {10.1017/s0022112005004295},
  journal   = {Journal of Fluid Mechanics},
  publisher = {Cambridge University Press (CUP)},
  author    = {Jovanović,  Mihailo r. and Bamieh,  Bassam},
  year      = {2005},
  month     = {June},
  pages     = {145–183}
}

@article{Trefethen1993,
  title     = {Hydrodynamic Stability Without Eigenvalues},
  volume    = {261},
  issn      = {1095-9203},
  doi       = {10.1126/science.261.5121.578},
  number    = {5121},
  journal   = {Science},
  publisher = {American Association for the Advancement of Science (AAAS)},
  author    = {Trefethen,  Lloyd N. and Trefethen,  Anne E. and Reddy,  Satish C. and Driscoll,  Tobin A.},
  year      = {1993},
  month     = {July},
  pages     = {578–584}
}

\appendix

\section{Neutral and Adjoint Mode Computation}
\label{app:modes}

For autonomous limit cycle systems, the transverse resolvent operator is defined as
\begin{equation}
\mb{R}_{\Sigma_\text{TS}}(\omega_f) \coloneqq \mb{P}_\text{TS} \mb{L}_\text{TS}^{-1} \mb{P}_\text{TS} \mb{B}_\text{TS},
\end{equation}
where the oblique projector is $\mb{P}_\text{TS} = \mb{I} - \tilde{\mb{p}}_0 \tilde{\mb{q}}_0^*$.
This operator requires the neutral Floquet mode $\tilde{\mb{p}}_0$ and its adjoint $\tilde{\mb{q}}_0$.
This appendix section details how these modes are computed in the time-spectral discretization and provides a proof of the shifted eigenpairs used in Section~\ref{sec:singularity_jacobian}.

\subsection{Neutral Mode Computation}

The neutral mode is the orbit-tangent direction of the base flow.
It is defined as
\begin{equation}
\tilde{\mb{p}}_0(t) \coloneqq \frac{\mathrm{d}\overbar{\mb{w}}}{\mathrm{d}t}.
\end{equation}
In the time-spectral discretization, a first-principles definition of the discrete neutral mode is as a right null vector of the collocation Jacobian.
That is, find a nontrivial vector $\tilde{\mb{p}}_0\in\mathbb{C}^{n n_\text{TS}}$ such that
\begin{equation}
\mb{J}_\text{TS}\tilde{\mb{p}}_0 = \mb{0}.
\end{equation}
In practice, when the nullspace is only approximate, $\tilde{\mb{p}}_0$ can be taken as the right singular vector associated with the smallest singular value of $\mb{J}_\text{TS}$.
In the time-spectral discretization, this is computed using the spectral differentiation matrix $\mb{D} \in \mathbb{R}^{n_\text{TS} \times n_\text{TS}}$ as
\begin{equation}
\tilde{\mb{p}}_0 = \omega_0 (\mb{D} \otimes \mb{I}_n) \overbar{\mb{w}},
\end{equation}
where $\overbar{\mb{w}} \in \mathbb{R}^{n_\text{TS} n}$ is the vectorized base flow and $n$ is the state dimension.
This differentiation-based construction is a convenient alternative that is exact when the base flow is temporally resolved and otherwise yields the neutral mode of the time-spectral approximation.
To quantify the discrepancy between the two constructions, define the neutral-mode residual
\begin{equation}
\mb{r}_{p} \coloneqq \mb{J}_\text{TS}\,\tilde{\mb{p}}_0,
\end{equation}
where $\tilde{\mb{p}}_0$ is computed by differentiation.
This residual can be expanded to reveal two distinct error sources.
First define the base-flow collocation residual
\begin{equation}
\mb{e}_{w} \coloneqq \omega_0(\mb{D}\otimes\mb{I}_n)\overbar{\mb{w}} - \mb{r}(\overbar{\mb{w}}),
\end{equation}
which measures how accurately the time-spectral base flow satisfies the nonlinear governing equation.
Using $\mb{J}_\text{TS}=\mb{J}_\text{block}-\omega_0(\mb{D}\otimes\mb{I}_n)$ and $\tilde{\mb{p}}_0=\omega_0(\mb{D}\otimes\mb{I}_n)\overbar{\mb{w}}$, we obtain
\begin{equation}
\begin{aligned}
\mb{r}_{p}
&= \mb{J}_\text{block}\tilde{\mb{p}}_0 - \omega_0(\mb{D}\otimes\mb{I}_n)\tilde{\mb{p}}_0 \\
&= \Big(\mb{J}_\text{block}\tilde{\mb{p}}_0 - \omega_0(\mb{D}\otimes\mb{I}_n)\mb{r}(\overbar{\mb{w}})\Big)
- \omega_0(\mb{D}\otimes\mb{I}_n)\mb{e}_{w}.
\end{aligned}
\end{equation}
The first term represents discrete chain-rule and temporal truncation error, while the second term represents the contribution of the base-flow collocation residual.
When the base flow is temporally resolved and the collocation residual $\mb{e}_{w}$ is small, both contributions vanish and $\tilde{\mb{p}}_0$ computed by differentiation approaches an exact null vector of $\mb{J}_\text{TS}$.
If the base flow is temporally resolved and the collocation equation is satisfied to numerical precision, then $\|\mb{r}_{p}\|_2$ is correspondingly small.
More generally, the normalized residual $\epsilon_{p} \coloneqq \|\mb{r}_{p}\|_2/\|\tilde{\mb{p}}_0\|_2$ provides a practical diagnostic for how accurately the differentiation-based vector lies in the null space of $\mb{J}_\text{TS}$.
When the null space is one-dimensional and $\sigma_2$ denotes the second-smallest singular value of $\mb{J}_\text{TS}$, the distance from $\tilde{\mb{p}}_0$ to the closest exact null vector is bounded by $\|\mb{r}_{p}\|_2/\sigma_2$.
The neutral mode is then normalized to unit norm $\tilde{\mb{p}}_0 \gets \tilde{\mb{p}}_0 / \|\tilde{\mb{p}}_0\|_2$.

\subsection{Shifted Eigenpairs of the Time-Spectral Jacobian}
\label{app:shifted_eigenpairs}

This subsection provides a proof of the shifted eigenpairs used in Section~\ref{sec:singularity_jacobian}.
Assume the neutral mode satisfies $\mathcal{J}_\text{TS}\tilde{\mb{p}}_0 = 0$ with the operator definition $\mathcal{J}_\text{TS}\mb{y} \coloneqq \mb{J}(t)\mb{y} - \omega_0\,\d\mb{y}/\d\theta$ and phase variable $\theta = \omega_0 t$.
For any integer $k$, define the shifted mode $\mb{y}_k(\theta) \coloneqq \tilde{\mb{p}}_0(\theta)e^{-jk\theta}$.
Applying the product rule gives
\begin{equation}
\frac{\d\mb{y}_k}{\d\theta} = \left(\frac{\d\tilde{\mb{p}}_0}{\d\theta}\right)e^{-jk\theta} - jk\,\tilde{\mb{p}}_0 e^{-jk\theta}.
\end{equation}
Substituting into the operator yields
\begin{equation}
\mathcal{J}_\text{TS}\mb{y}_k
= \left(\mb{J}(t)\tilde{\mb{p}}_0 - \omega_0\frac{\d\tilde{\mb{p}}_0}{\d\theta}\right)e^{-jk\theta}
+ jk\omega_0\,\tilde{\mb{p}}_0e^{-jk\theta}.
\end{equation}
The first parenthesis is $\mathcal{J}_\text{TS}\tilde{\mb{p}}_0 = 0$ by assumption, so $\mathcal{J}_\text{TS}\mb{y}_k = (jk\omega_0)\mb{y}_k$.
We next show the corresponding statement for the collocation matrix $\mb{J}_\text{TS}$.
Let $\mb{D}\in\mathbb{C}^{n_\text{TS}\times n_\text{TS}}$ be the Fourier spectral differentiation matrix (Section~\ref{sec:ts_method}) and let $\mathcal{F}\in\mathbb{C}^{n_\text{TS}\times n_\text{TS}}$ be the unitary DFT matrix.
Then $\mb{D}$ admits the diagonalization $\mb{D} = \mathcal{F}^{-1}\mb{\Lambda}\mathcal{F}$, where $\mb{\Lambda}=\mathrm{diag}(jk)$ contains the unique DFT wavenumbers $k\in\{-(n_\text{TS}-1)/2,\ldots,(n_\text{TS}-1)/2\}$.
Define the modulation matrix $\mb{S}_k \coloneqq \mathrm{diag}(e^{-jk\theta_0},\ldots,e^{-jk\theta_{n_\text{TS}-1}})$ and the lifted modulation $\mb{E}_k \coloneqq \mb{S}_k\otimes \mb{I}_n$ acting on stacked vectors in $\mathbb{C}^{n n_\text{TS}}$.
In Fourier space, multiplication by $e^{-jk\theta}$ corresponds to an index shift of Fourier coefficients.
Equivalently, one can verify that $\mb{S}_k = \mathcal{F}^{-1}\mb{\Pi}_k\mathcal{F}$ where $\mb{\Pi}_k$ is the permutation matrix that shifts the DFT modes by $k$ (with wrap-around).
Using this representation,
\begin{equation}
\mb{S}_k^{-1}\mb{D}\mb{S}_k
= \mathcal{F}^{-1}\mb{\Pi}_k^{-1}\mb{\Lambda}\mb{\Pi}_k\mathcal{F}
= \mathcal{F}^{-1}(\mb{\Lambda}+jk\,\mb{I})\mathcal{F}
= \mb{D}+jk\,\mb{I}.
\end{equation}
Since $\mb{J}_\text{block}$ is block-diagonal in time, it commutes with $\mb{E}_k$.
Therefore, for the collocation matrix $\mb{J}_\text{TS} = \mb{J}_\text{block}-\omega_0(\mb{D}\otimes\mb{I}_n)$ we obtain the similarity relation
\begin{equation}
\mb{E}_k^{-1}\mb{J}_\text{TS}\mb{E}_k
= \mb{J}_\text{TS} - jk\omega_0\,\mb{I}.
\end{equation}
If $\mb{J}_\text{TS}\tilde{\mb{p}}_0=\mb{0}$, then multiplying by $\mb{E}_k$ and using the identity above yields $\mb{J}_\text{TS}(\mb{E}_k\tilde{\mb{p}}_0) = (jk\omega_0)(\mb{E}_k\tilde{\mb{p}}_0)$.
This proves that $jk\omega_0$ is an eigenvalue of $\mb{J}_\text{TS}$ for each unique DFT mode $|k|\leq (n_\text{TS}-1)/2$.
For $|k|>(n_\text{TS}-1)/2$, the modulation aliases to an equivalent DFT mode because $e^{-j(k+n_\text{TS})\theta_j}=e^{-jk\theta_j}$ on the collocation grid.

\subsection{Adjoint Mode Computation}

The adjoint mode $\tilde{\mb{q}}_0$ is the left eigenvector of $\mb{J}_\text{TS}$ corresponding to the zero eigenvalue, satisfying $\mb{J}_\text{TS}^* \tilde{\mb{q}}_0 = \mb{0}$.
Direct computation via eigen decomposition is computationally prohibitive for large scale system.
Furthermore, numerical errors in the base flow can result in a non-zero eigenvalue satisfying $\mb{J}_\text{TS}^* \tilde{\mb{q}}_0 \ll \mb{1}$, making eigenvalue-based extraction unreliable.

To address both issues, we follow the optimization approach of~\citet{Padovan2022} in which $\tilde{\mb{q}}_0$ is the solution to the optimization problem
\begin{equation}
\begin{aligned}
\tilde{\mb{q}}_0 &= \arg \min_{\mb{q}} \|\mb{J}_\text{TS}^*\mb{q}\|^2 \\
& \quad \text{subject to } \langle \mb{q}, \tilde{\mb{p}}_0 \rangle = 1,
\end{aligned}
\end{equation}
which computes the adjoint mode $\tilde{\mb{q}}_0$ as the vector closest to the null space of $\mb{J}_\text{TS}^*$.
Using a Lagrange multiplier $\lambda$ for the constraint, the first-order optimality conditions yield the augmented system
\begin{equation}
\label{eq:adjoint_augmented}
\begin{bmatrix}
\mb{J}_\text{TS} \mb{J}_\text{TS}^* & -\tilde{\mb{p}}_0 \\
\tilde{\mb{p}}_0^* & 0
\end{bmatrix}
\begin{bmatrix}
\tilde{\mb{q}}_0 \\
\lambda
\end{bmatrix}
=
\begin{bmatrix}
\mb{0} \\
1
\end{bmatrix}.
\end{equation}
It can be shown that when $\mb{J}_\text{TS}$ is exactly singular the solution $\tilde{\mb{q}}_0$ is the exact left null vector of $\mb{J}_\text{TS}$ by identifying that $\lambda = 0$ in the first equation which leads to $\mb{J}_\text{TS}\mb{J}^*_\text{TS}\tilde{\mb{q}}_0 = \mb{0}$.

For large systems with a sparse $\mb{J}_\text{TS}$, explicitly forming $\mb{J}_\text{TS} \mb{J}_\text{TS}^*$ results in unacceptable fill-in and memory usage.
Instead, we define $\mb{J}_\text{TS} \mb{J}_\text{TS}^*$ as a linear operator through its action on a vector $\mb{x}$
\begin{equation}
(\mb{J}_\text{TS} \mb{J}_\text{TS}^*) \mb{x} = \mb{J}_\text{TS} (\mb{J}_\text{TS}^* \mb{x}),
\end{equation}
which requires only two sparse matrix-vector products.
We solve the augmented system~\eqref{eq:adjoint_augmented} using GMRES~\cite{Saad1986} with restarts, avoiding explicit matrix formulation.

\section{Efficient Computation of the Transverse Resolvent Operator}
\label{app:efficient_resolvent}

For large-scale systems, we avoid explicitly forming the projection matrix and inverting $\mb{L}_\text{TS}$ by instead applying the projection implicitly through vector-vector products and solving $\mb{L}_\text{TS}$ iteratively with a preconditioned GMRES solver.
To compute the resolvent gain as well as the optimal forcing and response modes, we use the randomized SVD algorithm detailed in~\citet[Algorithm 5.1]{Halko2011}.
This appendix section details the required forward and adjoint linear solves and the implicit application of the oblique projector.

\subsection{Implicit Oblique Projection}

The randomized SVD algorithm~\citet[Algorithm 5.1]{Halko2011} requires solving forward and adjoint systems of the form
\begin{equation}
\mb{L}_\text{TS}\tilde{\boldsymbol{\eta}}_\Sigma = \tilde{\mb{f}}_\Sigma, \quad \mb{L}_\text{TS}^*\tilde{\boldsymbol{\zeta}}_{\Sigma^*} = \tilde{\mb{g}}_{\Sigma^*} ,
\label{eq:forward_adjoint}
\end{equation}
where $\tilde{\boldsymbol{\eta}}_\Sigma \in \Sigma$, $\tilde{\tilde{\mb{f}}}_\Sigma \in \Sigma$, $\tilde{\boldsymbol{\zeta}}_{\Sigma^*} \in \Sigma^*$, and $\tilde{\mb{g}}_{\Sigma^*} \in \Sigma^*$, where $\Sigma^* = \mathrm{Range}(\mb{P}_k^*)$.
Following the approach by~\citet{Padovan2022}, we do not explicitly form the dense matrix $\mb{P}_\omega$; rather we implicitly apply the projection through vector-vector products as
\begin{equation}
\label{eq:implicit_projection}
\mb{P}_\omega \mb{x} = \mb{x} - \tilde{\mb{p}}_\omega \langle \tilde{\mb{q}}_\omega, \mb{x} \rangle, \quad \mb{P}_\omega^* \mb{x} = \mb{x} - \tilde{\mb{q}}_\omega \langle \tilde{\mb{p}}_\omega, \mb{x} \rangle
\end{equation}
for both the forward and adjoint projections on any vector $\mb{x}$.
This approach has the advantage of never computing the dense product $\mb{L}_\text{TS}\mb{P}_\omega$.
Instead the forward and adjoint systems are solved using only $\mb{L}_\text{TS}$ which is a sparse matrix.

The forward projected TSR operator action (Eq.~\ref{eq:forward_adjoint}) on a vector $\tilde{\mb{f}}$ can be computed in the following three steps.
1. Project the right hand side $\tilde{\mb{f}}_{\Sigma} = \mb{P}_\omega\tilde{\mb{f}}$.
In the context of the randomized SVD algorithm, $\tilde{\mb{f}}$ could be a random Gaussian vector.
2. Solve the linear system $\mb{L}_\text{TS}\tilde{\boldsymbol{\eta}} = \tilde{\mb{f}}_{\Sigma}$ for $\tilde{\boldsymbol{\eta}}$.
For this solve, we use the preconditioned GMRES solver descried next in~\ref{sec:fd_preconditioner}.
3. Project the solution $\tilde{\boldsymbol{\eta}}_{\Sigma} = \mb{P}_\omega\tilde{\boldsymbol{\eta}}$.
As mentioned in~\citet{Padovan2022}, this step is not absolutely necessary as $\Sigma_\text{TS}$ is an invariant subspace; however we preform it to ensure machine precision.
The adjoint projected TSR action can be computed in the same manner, substituting the matrices $\mb{P}_\omega$ and $\mb{L}_\text{TS}$ for their adjoint matrices.

\subsection{Finite Difference Preconditioned GMRES}
\label{sec:fd_preconditioner}

The forward and adjoint linear systems are solved using GMRES with restarts:
\begin{equation}
\mb{L}_\text{TS} \tilde{\boldsymbol{\eta}} = \tilde{\mb{f}}_{\Sigma}, \quad
\mb{L}_\text{TS}^* \boldsymbol{\zeta} = \mb{g}.
\end{equation}
To accelerate convergence, we precondition the system using a sparse approximation of $\mb{L}_\text{TS}$.
The time-spectral Jacobian has the structure
\begin{equation}
\mb{J}_\text{TS} = \mb{J}_\text{blk} - \omega_0 (\mb{D} \otimes \mb{I}_n),
\end{equation}
where $\mb{J}_\text{blk} = \mathrm{blkdiag}(\mb{J}(\theta_1), \ldots, \mb{J}(\theta_{n_\text{TS}}))$ is the block-diagonal spatial Jacobian and $\mb{D}$ is the dense Fourier spectral differentiation matrix.
The Kronecker product $\mb{D} \otimes \mb{I}_n$ which couples all time instances forms dense patterns inside $\mb{L}_\text{TS}$.

The preconditioner matrix $\mb{M}$ replaces the Fourier spectral matrix $\mb{D}$ with a sixth-order periodic finite difference approximation $\mb{D}_\text{FD}$, yielding
\begin{equation}
\mb{M} = j\omega_f \mb{I} - \mb{J}_\text{blk} + \omega_0 (\mb{D}_\text{FD} \otimes \mb{I}_n).
\end{equation}
Since $\mb{D}_\text{FD}$ is a banded matrix with $\mathcal{O}(1)$ nonzeros per row, the preconditioner $\mb{M}$ inherits the sparsity of $\mb{J}_\text{blk}$ and can be factored once via sparse LU decomposition.
The factored preconditioner is reused across all GMRES solves at a given forcing frequency, amortizing the factorization cost.

\section{Invariance and Bijectivity of $\mb{L}_\text{TS}$ in $\Sigma_\text{TS}$}
\label{app:invariance_bijectivity}

\subsection{Invariance of $\mb{L}_\text{TS}$}

To establish the well-posedness of the transverse resolvent, we must prove that the operator $\mathbf{L}_\text{TS}$ is both invariant and bijective when restricted to the subspace $\Sigma_\text{TS}$ at resonance.
We consider the resonant case where $\omega_f = k\omega_0$ and the null space of $\mathbf{L}_\text{TS}$ is one-dimensional and spanned by the modulated neutral mode $\tilde{\mathbf{p}}_k$.
The subspace $\Sigma_\text{TS}$ is defined as the range of the oblique projector $\mathbf{P}_{\omega} = \mathbf{I} - \tilde{\mathbf{p}}_{\omega}\tilde{\mathbf{q}}_{\omega}^*$.

First, we prove that $\Sigma_\text{TS}$ is invariant under $\mathbf{L}_\text{TS}$ by testing an arbitrary vector $\tilde{\mathbf{v}} \in \Sigma_\text{TS}$, which implies $\langle \tilde{\mathbf{q}}_k, \tilde{\mathbf{v}} \rangle = 0$.
We verify that the product $\mathbf{L}_\text{TS}\tilde{\mathbf{v}}$ remains orthogonal to $\tilde{\mathbf{q}}_k$ by taking the inner product
\begin{equation}
\langle \tilde{\mathbf{q}}_k, \mathbf{L}_\text{TS}\tilde{\mathbf{v}} \rangle = \tilde{\mathbf{q}}_k^* (j\omega_f \mathbf{I} - \mathbf{J}_\text{TS})\tilde{\mathbf{v}}.
\end{equation}
Using the property $\langle \tilde{\mathbf{q}}_k, \mathbf{J}_\text{TS}\tilde{\mathbf{v}} \rangle = \langle \mathbf{J}_\text{TS}^* \tilde{\mathbf{q}}_k, \tilde{\mathbf{v}} \rangle$ and the modulated adjoint eigenpair $\mathbf{J}_\text{TS}^* \tilde{\mathbf{q}}_k = -j\omega_f \tilde{\mathbf{q}}_k$, we derive
\begin{equation}
\langle \tilde{\mathbf{q}}_k, \mathbf{L}_\text{TS}\tilde{\mathbf{v}} \rangle = j\omega_f \langle \tilde{\mathbf{q}}_k, \tilde{\mathbf{v}} \rangle - \langle \mathbf{J}_\text{TS}^* \tilde{\mathbf{q}}_k, \tilde{\mathbf{v}} \rangle.
\end{equation}
Substituting the eigenvalue relation in gives
\begin{equation}
\langle \tilde{\mathbf{q}}_k, \mathbf{L}_\text{TS}\tilde{\mathbf{v}} \rangle = j\omega_f \langle \tilde{\mathbf{q}}_k, \tilde{\mathbf{v}} \rangle - (j\omega_f) \langle \tilde{\mathbf{q}}_k, \tilde{\mathbf{v}} \rangle = 0.
\end{equation}
Thus, for any $\tilde{\mathbf{v}} \in \Sigma_\text{TS}$, the resulting vector $\mathbf{L}_\text{TS}\tilde{\mathbf{v}}$ remains in $\Sigma_\text{TS}$, confirming invariance.

\subsection{Bijectivity of $\mb{L}_\text{TS}$}

Next, we establish bijectivity by showing that a unique solution $\tilde{\boldsymbol{\eta}}_{\Sigma} \in \Sigma_\text{TS}$ exists for any forcing $\tilde{\mathbf{f}}_{\Sigma} \in \Sigma_\text{TS}$.
Existence is guaranteed by the Fredholm alternative, which states that a solution exists if and only if the forcing is orthogonal to the null space of the adjoint operator $\mathbf{L}_\text{TS}^*$.
Since $\tilde{\mathbf{f}}_{\Sigma} \in \Sigma_\text{TS}$ is orthogonal to the adjoint null vector $\tilde{\mathbf{q}}_k$, a solution $\tilde{\boldsymbol{\eta}}$ is guaranteed to exist.
The general solution is $\tilde{\boldsymbol{\eta}} = \tilde{\boldsymbol{\eta}}_0 + c\tilde{\mathbf{p}}_k$ for an arbitrary constant $c$.
Uniqueness is achieved by requiring the solution to reside in $\Sigma_\text{TS}$, such that $\langle \tilde{\mathbf{q}}_k, \tilde{\boldsymbol{\eta}}_0 + c\tilde{\mathbf{p}}_k \rangle = 0$.
Solving for the constant $c = -\langle \tilde{\mathbf{q}}_k, \tilde{\boldsymbol{\eta}}_0 \rangle / \langle \tilde{\mathbf{q}}_k, \tilde{\mathbf{p}}_k \rangle$ yields a unique transverse response.
Therefore, the restricted operator is bijective and invertible within the subspace $\Sigma_\text{TS}$.

\section{Derivation of the Phase Drift Rate}
\label{app:phase_drift_derivation}

To derive the evolution equation for the phase perturbation magnitude $c(t)$, we begin with the forced linear time-periodic system given in Eq~\ref{eq:linear_dynamics}.
We decompose the physical response into a phase-shift component and a transverse component, $\boldsymbol{\eta}(t) = c(t)\tilde{\mathbf{p}}_0(t) + \mathbf{v}(t)$, where $\tilde{\mathbf{p}}_0(t)$ is the neutral Floquet mode representing the orbit tangent direction.
Substituting this decomposition into the linearized governing equation yields
\begin{equation}
\frac{d}{dt}\left(c(t)\tilde{\mathbf{p}}_0(t) + \mathbf{v}(t)\right) = \mathbf{J}(t)\left(c(t)\tilde{\mathbf{p}}_0(t) + \mathbf{v}(t)\right) + \mathbf{f}(t).
\end{equation}
Applying the product rule to the left-hand side results in
\begin{equation}
\dot{c}(t)\tilde{\mathbf{p}}_0(t) + c(t)\dot{\tilde{\mathbf{p}}}_0(t) + \dot{\mathbf{v}}(t) = c(t)\mathbf{J}(t)\tilde{\mathbf{p}}_0(t) + \mathbf{J}(t)\mathbf{v}(t) + \mathbf{f}(t).
\end{equation}
Since the neutral mode satisfies the unforced linearized dynamics $\dot{\tilde{\mathbf{p}}}_0(t) = \mathbf{J}(t)\tilde{\mathbf{p}}_0(t)$, the terms involving $c(t)$ cancel exactly
\begin{equation}
\dot{c}(t)\tilde{\mathbf{p}}_0(t) + \dot{\mathbf{v}}(t) = \mathbf{J}(t)\mathbf{v}(t) + \mathbf{f}(t).
\end{equation}
To isolate the scalar drift rate $\dot{c}(t)$, we project the entire equation onto the adjoint neutral mode $\tilde{\mathbf{q}}_0(t)$ by left-multiplying by $\tilde{\mathbf{q}}_0^*(t)$
\begin{equation}
\tilde{\mathbf{q}}_0^*(t)\tilde{\mathbf{p}}_0(t)\dot{c}(t) + \tilde{\mathbf{q}}_0^*(t)\dot{\mathbf{v}}(t) = \tilde{\mathbf{q}}_0^*(t)\mathbf{J}(t)\mathbf{v}(t) + \tilde{\mathbf{q}}_0^*(t)\mathbf{f}(t).
\end{equation}
Using the normalization $\tilde{\mathbf{q}}_0^*(t)\tilde{\mathbf{p}}_0(t) = 1$, we rearrange the equation as
\begin{equation}
\dot{c}(t) + \left(\tilde{\mathbf{q}}_0^*(t)\dot{\mathbf{v}}(t) - \tilde{\mathbf{q}}_0^*(t)\mathbf{J}(t)\mathbf{v}(t)\right) = \tilde{\mathbf{q}}_0^*(t)\mathbf{f}(t).
\end{equation}
The adjoint mode is mathematically guaranteed to be orthogonal to the transverse dynamics at all times, such that $\tilde{\mathbf{q}}_0^*(t)\mathbf{v}(t) = 0$.
Differentiating this orthogonality condition with respect to time gives $\dot{\tilde{\mathbf{q}}}_0^*(t)\mathbf{v}(t) + \tilde{\mathbf{q}}_0^*(t)\dot{\mathbf{v}}(t) = 0$, which allows us to substitute $\tilde{\mathbf{q}}_0^*(t)\dot{\mathbf{v}}(t) = -\dot{\tilde{\mathbf{q}}}_0^*(t)\mathbf{v}(t)$
\begin{equation}
\dot{c}(t) - \left(\dot{\tilde{\mathbf{q}}}_0^*(t) + \tilde{\mathbf{q}}_0^*(t)\mathbf{J}(t)\right)\mathbf{v}(t) = \tilde{\mathbf{q}}_0^*(t)\mathbf{f}(t).
\end{equation}
The term in the parentheses vanishes because the adjoint mode satisfies the adjoint linear dynamics $\dot{\tilde{\mathbf{q}}}_0(t) + \mathbf{J}^*(t)\tilde{\mathbf{q}}_0(t) = 0$.
Finally, we obtain the exact expression for the instantaneous phase drift rate
\begin{equation}
\dot{c}(t) = \tilde{\mathbf{q}}_0^*(t)\mathbf{f}(t),
\end{equation}
which is used to show that a forcing orthogonal to the modulated adjoint mode produces a response with zero phase drift in~\ref{sec:physical_response}.

\section{Derivation of Quasi-Periodic Response}
\label{app:quasi_periodic}

This appendix derives the quasi-periodic response Eq.~\ref{eq:periodic_response} for a linear time-periodic (LTP) system subject to quasi-periodic forcing with base $\omega_0$ and forcing $\omega_f$ frequency components.
Consider the forced LTP system
\begin{equation}
\dot{\boldsymbol{\eta}}(t) = \mb{J}(t)\boldsymbol{\eta}(t) + \mb{B}{\mb{f}}(t),
\end{equation}
where $\mb{J}(t) = \mb{J}(t + T_0)$ is the $T_0$-periodic Jacobian with base frequency $\omega_0 = 2\pi/T_0$.

Since $\mb{J}(t)$ is periodic, it can be expanded with a Fourier series as
\begin{equation}
\mb{J}(t) = \sum_{\ell=-\infty}^{\infty} \hat{\mb{J}}_\ell e^{j\ell\omega_0 t},
\end{equation}
where $\hat{\mb{J}}_\ell$ are the Fourier coefficients.

We seek a response $\boldsymbol{\eta}(t)$ to the harmonic forcing input.
Due to the periodic modulation of the Jacobian, the response is potentially quasi-periodic for forcing frequencies which are incommensurate with the base frequency.
The ansatz takes the following form
\begin{equation}
\boldsymbol{\eta}(t) = \sum_{p=-\infty}^{\infty} \hat{\boldsymbol{\eta}}_p e^{j(\omega_f + p\omega_0)t},
\label{eq:app_eta_ansatz}
\end{equation}
where $\hat{\boldsymbol{\eta}}_p \in \mathbb{C}^n$ are the Fourier coefficients of the response.
This represents a quasi-periodic signal with frequency content at the forcing frequency $\omega_f$ plus all integer multiples of the base frequency $\omega_0$.

Substituting this ansatz into the governing equation, we first compute the term $\mb{J}(t)\boldsymbol{\eta}(t)$ on the right hand side.
This gives the multiplication of two infinite series
\begin{equation}
\begin{aligned}
\mb{J}(t)\boldsymbol{\eta}(t) &= \left(\sum_{\ell=-\infty}^{\infty} \hat{\mb{J}}_\ell e^{j\ell\omega_0 t}\right)\left(\sum_{p=-\infty}^{\infty} \hat{\boldsymbol{\eta}}_p e^{j(\omega_f + p\omega_0)t}\right) \\
&= \sum_{\ell=-\infty}^{\infty}\sum_{p=-\infty}^{\infty} \hat{\mb{J}}_\ell \hat{\boldsymbol{\eta}}_p e^{j(\omega_f + (\ell+p)\omega_0)t}.
\end{aligned}
\end{equation}
To balance the harmonics, we group terms by their net frequency.
We define the index $k = \ell + p$ ($p = k - \ell$) to represent the harmonic shift.
Substituting this summation gives
\begin{equation}
\mb{J}(t)\boldsymbol{\eta}(t) = \sum_{k=-\infty}^{\infty} \left(\sum_{\ell=-\infty}^{\infty} \hat{\mb{J}}_\ell \hat{\boldsymbol{\eta}}_{k-\ell}\right) e^{j(\omega_f + k\omega_0)t},
\end{equation}
which is the complete form of the product $\mb{J}(t)\boldsymbol{\eta}(t)$.

We now consider the time derivative on the left hand side $\dot{\boldsymbol{\eta}}(t)$.
Differentiating the ansatz gives
\begin{equation}
\dot{\boldsymbol{\eta}}(t) = \sum_{k=-\infty}^{\infty} j(\omega_f + k\omega_0)\hat{\boldsymbol{\eta}}_k e^{j(\omega_f + k\omega_0)t},
\end{equation}
using index $k$ to match the right hand side.

Expanding the quasi-periodic forcing term by its Fourier coefficients gives
\begin{equation}
\mb{B}{\mb{f}}(t) = \sum_{k=-\infty}^\infty \mb{B} \hat{\mb{f}}_k e^{j(\omega_f + k\omega_0)t}.
\end{equation}
Using the harmonic balance technique, the coefficients of the exponential $e^{j(\omega_f + k\omega_0)t}$ in each term must satisfy the equation independently for each $k$.
Equating the coefficients by canceling out the exponential on each term gives
\begin{equation}
j(\omega_f + k\omega_0)\hat{\boldsymbol{\eta}}_k = \sum_{\ell=-\infty}^{\infty} \hat{\mb{J}}_\ell \hat{\boldsymbol{\eta}}_{k-\ell} + \mb{B}\hat{\mb{f}}_k,
\end{equation}
which is identical to the standard harmonic balance equation.

This infinite system of coupled algebraic equations demonstrates that quasi-periodic forcing with spectral content of both the forcing and base frequencies excites responses at all sideband frequencies $\omega_f + k\omega_0$ through the coupling induced by the time-periodic Jacobian.
The off-diagonal Fourier coefficients $\hat{\mb{J}}_\ell$ ($\ell \neq 0$) transfer energy between harmonics, while the mean Jacobian $\hat{\mb{J}}_0$ acts independently at each frequency.
Rearranging into matrix form gives the block-Toeplitz operator $\mb{L}_{\text{HR}}(\omega_f)$ in Eq.~\ref{eq:L_operator_matrix}.

\section{Harmonic Forcing Input Matrix}
\label{app:parseval}

This appendix describes the construction of the time-spectral input matrix $\mb{B}_\text{TS}$ for the special case where it is desired to restrict the forcing to be purely harmonic $\tilde{\mb{f}}(t) = \hat{\mb{f}}e^{j\omega_f t}$, where $\hat{\mb{f}}$ is constant.
The resolvent gain is physically defined as the maximum amplification over the long-time-averaged $L^2$ norms.
Due to the difference in size of the quasi periodic response and harmonic forcing vectors, Parseval scaling is needed to ensure the singular values of the discrete TSR operator match the physical gain.

Restricting the forcing envelope to be constant ($\hat{\mb{f}}e^{j\omega_f t} \in \mathbb{R}^m$) requires the appropriately sized input matrix $\mb{B}_\text{TS} = (1 / \sqrt{n_\text{TS}}) (\mb{1}_{n_\text{TS}} \otimes \mb{B}) \in \mathbb{R}^{n_\text{TS}n\times m}$ which implicitly forces the optimal forcing envelope to be constant.

The Parseval $1 / \sqrt{n_\text{TS}}$ scaling comes from the size difference in the forcing and response vectors as the discrete response vector is $\tilde{\boldsymbol{\eta}} \in \mathbb{C}^{N\cdot n_\text{TS}}$.
This section derives the $1 / \sqrt{n_\text{TS}}$ scaling required in the input matrix $\mb{B}_\text{TS}$ to ensure the discrete singular values correspond to physical energy amplification.

For the quasi-periodic response $\boldsymbol{\eta}(t) = \hat{\boldsymbol{\eta}}(t)e^{j\omega_f t}$, the physical energy amplification defined by the long-time-averaged $L^2$ norm is
\begin{equation}
||\boldsymbol{\eta}(t)||_{L^2}^2 = \lim_{T\to\infty} \f{1}{T}\int_0^T |\hat{\boldsymbol{\eta}}(t)|^2 \d t = \f{1}{T_0}\int_0^{T_0} |\hat{\boldsymbol{\eta}}(t)|^2 \d t,
\end{equation}
where the second equality follows from the $T_0$-periodicity of the envelope $\hat{\boldsymbol{\eta}}(t)$.
Using the trapezoidal rule (which is spectrally accurate for periodic functions), this integral is approximated by the discrete sum at collocation points $\theta_j$
\begin{equation}
||\boldsymbol{\eta}(t)||_{L^2}^2 \approx \f{1}{n_{\text{TS}}} \sum_{j=0}^{n_{\text{TS}}-1} |\hat{\boldsymbol{\eta}}(\theta_j)|^2 = \f{1}{n_{\text{TS}}} ||\tilde{\boldsymbol{\eta}}||_2^2.
\end{equation}
Taking the square root gives
\begin{equation}
||\boldsymbol{\eta}(t)||_{L^2} \approx \f{1}{\sqrt{n_{\text{TS}}}}||\tilde{\boldsymbol{\eta}}||_2,
\end{equation}
which relates the continuous and discrete norms of the quasi-periodic response.

For the purely harmonic forcing $\mb{f}(t) = \hat{\mb{f}}e^{j\omega_f t}$ with constant amplitude $\hat{\mb{f}}$, the long-time-averaged norm simplifies to
\begin{equation}
||\mb{f}(t)||_{L^2}^2 = \lim_{T\to\infty} \f{1}{T}\int_0^T |\hat{\mb{f}}|^2 \d t = |\hat{\mb{f}}|^2 = ||\hat{\mb{f}}||_2^2,
\end{equation}
since the amplitude is constant in time.

Substituting these expressions into the resolvent gain Eq.~\ref{eq:resolvent_def} gives
\begin{equation}
G = \f{||\boldsymbol{\eta}(t)||_{L^2}}{||\mb{f}(t)||_{L^2}} \approx \f{1}{\sqrt{n_{\text{TS}}}} \f{||\tilde{\boldsymbol{\eta}}||_2}{||\hat{\mb{f}}||_2}.
\end{equation}
To recover this gain directly from $\sigma_{\max}(\mb{R}_{\text{TS}})$, the input matrix must absorb the $1/\sqrt{n_{\text{TS}}}$ scaling factor, leading to the scaled input matrix defined above.
In the case of quasi-periodic forcing (Eq~\ref{eq:input_mat}), the long-time-averaged norm of the forcing has the same scaling as the response $\|\mb{f}(t)\|_{L^2} \approx (1/\sqrt{n_\text{TS}})\|\tilde{\mb{f}}\|_2$.
This leads the the factor canceling out and Parseval scaling of the input matrix is not needed in the more general quasi-periodic case.

\section{Equivalence of TSR and HR Operators}
\label{app:equivalence}

This appendix establishes the equivalence between the time-spectral resolvent (TSR) and harmonic resolvent (HR) operators when using the same number of frequency components.
The time-spectral governing equation Eq.~\ref{eq:ts_governing} and the harmonic balance governing equation Eq.~\ref{eq:hb_governing} are related by the discrete Fourier transform (DFT).
When the DFT is applied to the TSR solution, the result is the harmonic balance Fourier coefficients of the HR operator.
Therefore, when both methods use the same truncation ($n_{\text{TS}} = 2n_{\text{har}} + 1$), they produce identical periodic orbits $\bar{\mb{w}}(t)$ and consequently identical Jacobians $\mb{J}(t)$.

Let $\mb{F} \in \mathbb{C}^{n_{\text{TS}} \times n_{\text{TS}}}$ denote the standard unitary DFT matrix with entries
\begin{equation}
F_{pq} = \frac{1}{\sqrt{n_{\text{TS}}}}e^{-j 2\pi pq / n_{\text{TS}}}, \quad p,q = 0, \ldots, n_{\text{TS}} - 1.
\end{equation}
To apply the DFT to the temporal dimension, the full transformation matrix is given as $\mathcal{F} = \mb{F} \otimes \mb{I}_n$.
This acts as identity to the state dimension.
For odd $n_{\text{TS}} = 2n_{\text{har}} + 1$, the DFT maps between time-domain samples and Fourier coefficients for harmonics $k = -n_{\text{har}}, \ldots, n_{\text{har}}$.

The key observation is that the spectral differentiation matrix $\mb{D}$ is diagonalized by the DFT.
This relationsip is given as
\begin{equation}
\mathcal{F}(\mb{D} \otimes \mb{I}_n)\mathcal{F}^{-1} = \mathrm{diag}(jk\mb{I}_n)_{k=-n_{\text{har}}}^{n_{\text{har}}}.
\end{equation}
This corresponds exactly to the frequency-domain differentiation $\partial_t \leftrightarrow jk\omega_0$.

For the block-diagonal Jacobian term, the DFT converts pointwise multiplication in time to convolution in frequency.
Specifically, if $\mb{J}_{\text{block}} = \mathrm{blkdiag}(\mb{J}_0, \ldots, \mb{J}_{n_{\text{TS}}-1})$ contains the Jacobian at collocation points, then
\begin{equation}
\mathcal{F}\mb{J}_{\text{block}}\mathcal{F}^{-1} = \mb{T}_{n_{\text{har}}},
\end{equation}
where $\mb{T}_{n_{\text{har}}}$ is the block-Toeplitz matrix with blocks $\hat{\mb{J}}_{k-\ell}$ truncated to $|k|, |\ell| \leq n_{\text{har}}$. This is precisely the structure appearing in $\mb{L}_{\text{HR}}$ in Eq.~\ref{eq:L_operator_matrix}.

Combining these results, the TSR and HR operators are related by
\begin{equation}
\mathcal{F}\mb{L}_{\text{TS}}\mathcal{F}^{-1} = \mb{L}_{\text{HR}}^{(n_{\text{har}})}, \quad \mb{R}_{\text{TS}} = \mathcal{F}^{-1}\mb{R}_{\text{HR}}^{(n_{\text{har}})}\mathcal{F}.
\end{equation}
Since $\mathcal{F}$ is unitary, the singular values of the two operators are equal: $\sigma_{\text{TS}} = \sigma^{(n_{\text{har}})}$.
The singular vectors differ only by the DFT transformation, representing the same modes in either time or frequency domain.

This equivalence shows that TSR and HR are dual representations of the same underlying operator, with TSR operating on time-domain samples and HR operating on Fourier coefficients.

\section{Padding and Truncation Error}
\label{app:padding}

To compare $\mb{R}_{\text{HR}}^{(n)}$ with $\mb{R}_{\text{HR}}^{(m)}$ for $m > n$, we embed the smaller operator into the larger space. Define the padded operator $\widetilde{\mb{L}}_{\text{HR}}^{(n\to m)}$ on the $m$-harmonic space by
\begin{equation}
\widetilde{\mb{L}}_{\text{HR}}^{(n\to m)} =
\begin{bmatrix}
\mb{D}_- & \mb{0} & \mb{0} \\
\mb{0} & \mb{L}_{\text{HR}}^{(n)} & \mb{0} \\
\mb{0} & \mb{0} & \mb{D}_+
\end{bmatrix},
\label{eq:padded_L_block_structure}
\end{equation}
where $\mb{D}_- = \operatorname{diag}(\mb{A}_{-m}^{(n)}, \ldots, \mb{A}_{-(n+1)}^{(n)})$ and $\mb{D}_+ = \operatorname{diag}(\mb{A}_{n+1}^{(n)}, \ldots, \mb{A}_{m}^{(n)})$, with $\mb{A}_k^{(n)} = j(\omega_f + k\omega_0)\mb{I} - \hat{\mb{J}}_0^{(n)}$ the uncoupled diagonal blocks using the Jacobian from the $n$-truncated base flow. Under the regularity assumption of Section~\ref{sec:convergence}, the blocks $\mb{A}_k^{(n)}$ are invertible for all $k$.

Similarly, define the padded input matrix $\widetilde{\mb{B}}_{\text{HR}}^{(n\to m)}$ by
\begin{equation}
\widetilde{\mb{B}}_{\text{HR}}^{(n\to m)} =
\begin{bmatrix}
\mb{0} \\ \mb{B}_{\text{HR}}^{(n)} \\ \mb{0}
\end{bmatrix}
= \mb{B}_{\text{HR}}^{(m)},
\label{eq:padded_B_block_structure}
\end{equation}
where the equality holds because $\mb{B}_{\text{HR}}^{(n)}$ is nonzero only in the $k=0$ block.

Since $\widetilde{\mb{L}}_{\text{HR}}^{(n\to m)}$ is block-diagonal and $\widetilde{\mb{B}}_{\text{HR}}^{(n\to m)}$ is zero outside the central block, the padded resolvent has the form
\begin{equation}
\widetilde{\mb{R}}_{\text{HR}}^{(n\to m)} = (\widetilde{\mb{L}}_{\text{HR}}^{(n\to m)})^{-1}\widetilde{\mb{B}}_{\text{HR}}^{(n\to m)} =
\begin{bmatrix}
\mb{0} \\ \mb{R}_{\text{HR}}^{(n)} \\ \mb{0}
\end{bmatrix}.
\label{eq:padded_R_block_structure}
\end{equation}
This zero-padded embedding preserves all singular values of $\mb{R}_{\text{HR}}^{(n)}$ (\citet{HornJohnson2013}, Theorem 7.3.3); in particular, $\sigma_1(\widetilde{\mb{R}}_{\text{HR}}^{(n\to m)}) = \sigma^{(n)}$.

The truncation error $\mb{L}_{\text{HR}}^{(m)} - \widetilde{\mb{L}}_{\text{HR}}^{(n\to m)}$ has two contributions: (i) off-diagonal couplings neglected in the padding, and (ii) Jacobian coefficient differences due to base flow truncation:
\begin{equation}
\mb{L}_{\text{HR}}^{(m)} - \widetilde{\mb{L}}_{\text{HR}}^{(n\to m)} =
\begin{bmatrix}
\mb{E}_{-,-} & \mb{E}_{-,0} & \mb{E}_{-,+} \\
\mb{E}_{0,-} & \mb{E}_{0,0} & \mb{E}_{0,+} \\
\mb{E}_{+,-} & \mb{E}_{+,0} & \mb{E}_{+,+}
\end{bmatrix}.
\label{eq:truncation_error}
\end{equation}
The central block captures coefficient differences:
\begin{equation}
[\mb{E}_{0,0}]_{k,\ell} = -(\hat{\mb{J}}_{k-\ell}^{(m)} - \hat{\mb{J}}_{k-\ell}^{(n)}), \quad |k|, |\ell| \leq n.
\end{equation}
The inner-to-outer couplings contain the neglected off-diagonal terms:
\begin{equation}
[\mb{E}_{0,+}]_{k,\ell} = -\hat{\mb{J}}_{k-\ell}^{(m)}, \quad |k| \leq n, \; n+1 \leq \ell \leq m,
\end{equation}
\begin{equation}
[\mb{E}_{0,-}]_{k,\ell} = -\hat{\mb{J}}_{k-\ell}^{(m)}, \quad |k| \leq n, \; -m \leq \ell \leq -(n+1),
\end{equation}
with $\mb{E}_{+,0} = \mb{E}_{0,+}^T$ and $\mb{E}_{-,0} = \mb{E}_{0,-}^T$. The outer diagonal blocks have coefficient differences on the diagonal and full couplings off-diagonal (since $\widetilde{\mb{L}}_{\text{HR}}^{(n\to m)}$ has only diagonal entries $\mb{A}_k^{(n)}$ there):
\begin{equation}
[\mb{E}_{\pm,\pm}]_{k,\ell} =
\begin{cases}
-(\hat{\mb{J}}_{0}^{(m)} - \hat{\mb{J}}_{0}^{(n)}) & k = \ell, \\
-\hat{\mb{J}}_{k-\ell}^{(m)} & k \neq \ell.
\end{cases}
\end{equation}
The outer-to-outer cross couplings are entirely neglected in the padding:
\begin{equation}
[\mb{E}_{-,+}]_{k,\ell} = -\hat{\mb{J}}_{k-\ell}^{(m)}, \quad -m \leq k \leq -(n+1), \; n+1 \leq \ell \leq m,
\end{equation}
with $\mb{E}_{+,-} = \mb{E}_{-,+}^T$. Under assumptions (J1) and (J2), all blocks are $O(e^{-\alpha n})$.

\textbf{Example ($n=1$, $m=2$).} For concreteness, consider padding from $n=1$ to $m=2$ harmonics. Using shorthand $\mb{A}_k^{(n)} = j(\omega_f + k\omega_0)\mb{I} - \hat{\mb{J}}_0^{(n)}$ and $\delta\mb{J}_k = \hat{\mb{J}}_k^{(2)} - \hat{\mb{J}}_k^{(1)}$, the operators are (rows/columns ordered $k = -2, -1, 0, 1, 2$):
\begin{equation}
\mb{L}_{\text{HR}}^{(2)} =
\begin{bmatrix}
\mb{A}_{-2}^{(2)} & -\hat{\mb{J}}_{-1}^{(2)} & -\hat{\mb{J}}_{-2}^{(2)} & -\hat{\mb{J}}_{-3}^{(2)} & -\hat{\mb{J}}_{-4}^{(2)} \\
-\hat{\mb{J}}_{1}^{(2)} & \mb{A}_{-1}^{(2)} & -\hat{\mb{J}}_{-1}^{(2)} & -\hat{\mb{J}}_{-2}^{(2)} & -\hat{\mb{J}}_{-3}^{(2)} \\
-\hat{\mb{J}}_{2}^{(2)} & -\hat{\mb{J}}_{1}^{(2)} & \mb{A}_{0}^{(2)} & -\hat{\mb{J}}_{-1}^{(2)} & -\hat{\mb{J}}_{-2}^{(2)} \\
-\hat{\mb{J}}_{3}^{(2)} & -\hat{\mb{J}}_{2}^{(2)} & -\hat{\mb{J}}_{1}^{(2)} & \mb{A}_{1}^{(2)} & -\hat{\mb{J}}_{-1}^{(2)} \\
-\hat{\mb{J}}_{4}^{(2)} & -\hat{\mb{J}}_{3}^{(2)} & -\hat{\mb{J}}_{2}^{(2)} & -\hat{\mb{J}}_{1}^{(2)} & \mb{A}_{2}^{(2)}
\end{bmatrix},
\end{equation}
\begin{equation}
\widetilde{\mb{L}}_{\text{HR}}^{(1\to 2)} =
\begin{bmatrix}
\mb{A}_{-2}^{(1)} & 0 & 0 & 0 & 0 \\
0 & \mb{A}_{-1}^{(1)} & -\hat{\mb{J}}_{-1}^{(1)} & -\hat{\mb{J}}_{-2}^{(1)} & 0 \\
0 & -\hat{\mb{J}}_{1}^{(1)} & \mb{A}_{0}^{(1)} & -\hat{\mb{J}}_{-1}^{(1)} & 0 \\
0 & -\hat{\mb{J}}_{2}^{(1)} & -\hat{\mb{J}}_{1}^{(1)} & \mb{A}_{1}^{(1)} & 0 \\
0 & 0 & 0 & 0 & \mb{A}_{2}^{(1)}
\end{bmatrix},
\end{equation}
\begin{equation}
\mb{L}_{\text{HR}}^{(2)} - \widetilde{\mb{L}}_{\text{HR}}^{(1\to 2)} =
\begin{bmatrix}
-\delta\mb{J}_{0} & -\hat{\mb{J}}_{-1}^{(2)} & -\hat{\mb{J}}_{-2}^{(2)} & -\hat{\mb{J}}_{-3}^{(2)} & -\hat{\mb{J}}_{-4}^{(2)} \\
-\hat{\mb{J}}_{1}^{(2)} & -\delta\mb{J}_{0} & -\delta\mb{J}_{-1} & -\delta\mb{J}_{-2} & -\hat{\mb{J}}_{-3}^{(2)} \\
-\hat{\mb{J}}_{2}^{(2)} & -\delta\mb{J}_{1} & -\delta\mb{J}_{0} & -\delta\mb{J}_{-1} & -\hat{\mb{J}}_{-2}^{(2)} \\
-\hat{\mb{J}}_{3}^{(2)} & -\delta\mb{J}_{2} & -\delta\mb{J}_{1} & -\delta\mb{J}_{0} & -\hat{\mb{J}}_{-1}^{(2)} \\
-\hat{\mb{J}}_{4}^{(2)} & -\hat{\mb{J}}_{3}^{(2)} & -\hat{\mb{J}}_{2}^{(2)} & -\hat{\mb{J}}_{1}^{(2)} & -\delta\mb{J}_{0}
\end{bmatrix}.
\end{equation}
The central $3\times 3$ block contains the coefficient differences $\delta\mb{J}_k$, while the outer entries are the neglected couplings $\hat{\mb{J}}_k^{(2)}$.

\end{document}